\documentclass{amsart}
\usepackage{amscd}
\usepackage{amssymb}
\usepackage{amsmath}
\usepackage{amsthm}

\long\def\comment#1\endcomment{}
\usepackage{color}

\title[Random groups, random graphs and $p$-Laplacians]{Random groups, random graphs and eigenvalues of $p$-Laplacians}
\author{Cornelia Dru\c{t}u}
\email{drutu@maths.ox.ac.uk}
\address{Mathematical Institute \\
 University of Oxford \\ Oxford, UK.}
\author{John M. Mackay}\thanks{The research of both authors was supported in part by
the EPSRC grant no. EP/K032208/1 entitled ``Geometric and analytic aspects of infinite groups". The research of the first author was also supported by the project ANR Blanc ANR-10-BLAN 0116, acronyme GGAA, and by Labex CEMPI}
\email{john.mackay@bristol.ac.uk}
\address{School of Mathematics \\ University of Bristol \\ Bristol, UK}
\date{\today}%{18/5--30/5, 1/7--\today}
\usepackage{amsmath,amsthm,amsfonts,graphicx,amssymb}
\usepackage{hyperref}

\definecolor{darkgreen}{cmyk}{1,0,1,.2}
\allowdisplaybreaks[1] % but discouraged

\long\def\comment#1\endcomment{}

\numberwithin{equation}{section}
\newtheorem{theorem}[equation]{Theorem}
\newtheorem{proposition}[equation]{Proposition}
\newtheorem{corollary}[equation]{Corollary}
\newtheorem{lemma}[equation]{Lemma}

\newtheorem{notation}[equation]{Notation}
\newtheorem{cvn}[equation]{Convention}

\newtheorem{question}[equation]{Question}

\newtheorem{definition}[equation]{Definition}
\newtheorem{remark}[equation]{Remark}

\newtheorem{assumption}[equation]{Assumption}

\newtheoremstyle{citing}% name
  {3pt}%      Space above, empty = `usual value'
  {3pt}%      Space below
  {\itshape}% Body font
  {}%         Indent amount (empty = no indent, \parindent = para indent)
  {\bfseries}% Thm head font
  {}%        Punctuation after thm head
  {.5em}%     Space after thm head: " " = normal interword space;
  {\thmnote{#3}}% Thm head spec

\theoremstyle{citing}
\newtheorem*{varthm}{}% all text supplied in the note

\DeclareMathOperator{\val}{val}

\newcommand{\alp}{\alpha}
\newcommand{\bet}{\beta}
\newcommand{\del}{\delta}

\newcommand{\eps}{\epsilon}

\newcommand{\lam}{\lambda}

\newcommand{\bdry}{\partial_\infty}

\newcommand{\cK}{\mathcal{K}}
\newcommand{\cB}{\mathcal{B}}
\newcommand{\cA}{\mathcal{A}}
\newcommand{\cC}{\mathcal{C}}

\newcommand{\cD}{\mathcal{D}}
\newcommand{\cE}{\mathcal{E}}

\DeclareMathOperator{\Confdim}{Confdim}

\newcommand{\cG}{\mathcal{G}}

\newcommand{\cM}{\mathcal{M}}

\newcommand{\cS}{\mathcal{S}}

\newcommand{\ra}{\rightarrow}
\newcommand{\R}{\mathbb{R}}

\newcommand{\N}{\mathbb{N}}
\newcommand{\Z}{\mathbb{Z}}

\newcommand{\bP}{\mathbb{P}}
\newcommand{\bE}{\mathbb{E}}
\newcommand{\bI}{\mathbb{I}}

\newcommand{\bd}{\mathbf{d}}
\newcommand{\bG}{\mathbb{G}}

\newcommand{\frakF}{\mathfrak{F}}
\newcommand{\cF}{\mathcal{F}}

\newcommand{\tilR}{\widetilde{\Re}}
\newcommand{\tilX}{\widetilde{X}}
\newcommand{\ovX}{\overline{X}}

\newcommand{\ovP}{\overline{\Re }}
\newcommand{\vtx}{\mathcal{V}}

\numberwithin{equation}{section}

\begin{document}

\begin{abstract}

  We prove that a random group in the triangular density model has, for density larger than $1/3$, fixed point properties for actions on $L^p$--spaces (affine isometric, and more generally $(2-2\epsilon)^{1/2p}$--uniformly Lipschitz) with $p$ varying in an interval increasing with the set of generators. In the same model, we establish a double inequality between the maximal $p$ for which $L^p$--fixed point properties hold and the conformal dimension of the boundary.   

In the Gromov density model, we prove that for every $p_0 \in [2,\infty )$ for a sufficiently large number of generators and for any density larger than $1/3$, a random group satisfies the fixed point property for affine actions on $L^p$--spaces that are $(2-2\epsilon)^{1/2p}$--uniformly Lipschitz, and this for every $p\in [2,p_0]$.  

To accomplish these goals we find new bounds on the first eigenvalue of the $p$-Laplacian on random graphs, using methods adapted from Kahn and Szemer\'edi's approach to the $2$-Laplacian.
These in turn lead to fixed point properties using arguments of Bourdon and Gromov, which extend to $L^p$-spaces previous results for Kazhdan's Property (T) established by \.Zuk and Ballmann-\'Swi\c atkowski.
\end{abstract}


\maketitle

%%%%%%%%%%%%%%%%%%%%%%%%%%%%%%%%%%%%%%%%%%%%%%%%%%%%%%%%%%%%%%%%%
\section{Introduction}\label{sec-intro}

One way to study infinite groups is through their actions on various classes of spaces. From this point of view, of particular importance are the fixed point properties, that is the properties stating that a group can act isometrically on a certain type of metric space only when the action has a global fixed point. For Hilbert spaces, this is the so called property $FH$ of J.P.\ Serre, which for locally compact second countable topological groups (and continuous actions) is equivalent to Kazhdan's property (T). The research around similar properties for various types of Banach spaces, or of non-positively curved spaces, has been very lively in recent years. The relevance of fixed point properties is manifest in many important areas, from combinatorics to ergodic theory, smooth dynamics, operator algebras and the Baum-Connes conjecture. 

Despite their importance, many questions related to fixed point properties remain open, even in cases such as the $L^p$--spaces, which are in a sense the closest relatives to Hilbert spaces, among the Banach spaces. In this paper we investigate fixed point properties on $L^p$--spaces and on spaces whose finite dimensional geometry is related to that of $L^p$--spaces, in the following sense. 

\begin{definition}\label{def:splpL}
Let $p\in (0,\infty)$, $L\geq 1$ and $m\in \N$. A Banach space is said to have an 
{\emph{$L$-bi-Lipschitz $L^p$ geometry above dimension $m$}} if every $m$-dimensional subspace of it is contained in a subspace $L$-bi-Lipschitz equivalent either to an $\ell^p_n$ for some $n\geq m$, or to $\ell^p_\infty$ or to some space $L^p(X,\mu )$.

When the above property holds for every $m$, we say that the Banach space has an
{\emph{$L$-bi-Lipschitz $L^p$ geometry in finite dimension}}.
\end{definition}

Particular cases of spaces with $L$-bi-Lipschitz $L^p$ geometry in finite dimension are given by the usual spaces $L^p(X,\mu)$, or spaces $L$-bi-Lipschitz equivalent to a $L^p(X,\mu)$.
Examples of spaces with {{$L$-bi-Lipschitz $L^p$ geometry above dimension $m$}} include spaces of cotype $p$ with all subspaces of dimension $k$ $(L-\epsilon )$--bi-Lipschitz equivalent, for some $k\geq m$ and some small $\epsilon >0$ \cite[Theorem G.5]{BenyaminiLindenstrauss}. 

The combinatorial construction that we use in this paper has a natural connection to fixed point properties for actions on this type of spaces, which we now formulate. 

\begin{definition}\label{def:IflpL}
A topological group $\Gamma $ {\emph{has property $FL^p_{m,L}$}} if every affine isometric continuous action of $\Gamma$ on a Banach space with $L$-bi-Lipschitz $L^p$ geometry above dimension $m$ has a global fixed point.

We say that $\Gamma$ {\emph{has property $FL^p_{L}$}} if every affine isometric continuous action of $\Gamma$ on a Banach space with $L$-bi-Lipschitz $L^p$ geometry in finite dimension has a global fixed point.
\end{definition}
The continuity condition requires simply that the orbit map $g\mapsto gv$ is continuous, for every vector $v$ in the considered Banach space. Recall that for $p\in (0,1)$ the metric considered on $L^p(X,\mu )$ is given by the $p$--power of the $p$--norm, otherwise the triangular inequality would not be satisfied.

In the case when Definition~\ref{def:IflpL} is restricted to isometric actions ($L=1$) and the Banach spaces are only $L^p$--spaces, the property is also called the  $FL^p$--{\emph{property}}. A theorem of Delorme--Guichardet \cite{Guichardet:bull,Delorme} together with a standard Functional Analysis result \cite[Theorem 4.10]{WellsWilliams:Embeddings} imply that for every $p\in (1,2]$ property $FL^p$ is equivalent to Kazhdan's property (T) (see also \cite[Theorem 1.3]{BFGM}). For $p=1$ the equivalence is proved in \cite{BGM}.

For $p\geq 2$ property $FL^p$ implies property (T), but the converse is not true, at least not for $p$ large.

Indeed, it
follows from work of P. Pansu \cite{Pansu:cohomologiepubl} and of Cornulier, Tessera \& Valette \cite{CornTessValette:infinity} that given ${\mathbb{H}}^n_{\textbf{H}}$, the $n$-dimensional hyperbolic space over the field of quaternions, its group of isometries $\Gamma =Sp(n,1)$, which has property (T), does not have property $FL^p$ for $p>4n+2$ (where $4n+2$ is the conformal dimension of the boundary $\partial_\infty {\mathbb{H}}^n_{\textbf{H}}$); moreover, it admits a proper action on such an $L^p$-space. Also, a result of M.\ Bourdon \cite{Bourdon16-properLp}, strengthening work of G.\ Yu \cite{Yu:hyplp}, implies that non-elementary hyperbolic groups $\Gamma$ have fixed-point-free---in fact, proper--- isometric actions on an $\ell^p$-space for $p$ larger than the conformal dimension of the boundary $\partial_\infty \Gamma $ (see also Bourdon \& Pajot \cite{BourdonPajot:cohomologie} and Nica \cite{Nica13-ProperLp}). In particular this holds for hyperbolic groups with property (T).

This shows that for large $p>2$ property $FL^p$ is strictly stronger than property (T). The comparison between the two properties when $p>2$ is close to $2$ is unclear. It is known that every group with property (T) has property $FL^p$ for $p\in [2, 2+ \epsilon )$, where $\epsilon $ depends on the group \cite{BFGM, DrutuKapovich}.

Like other strong versions of property (T), the family of properties $FL^p$ separates the simple Lie groups of rank one from the simple Lie groups of rank at least $2$ (and their respective lattices). Indeed, all rank one groups and their uniform lattices fail to have $FL^p$ for $p$ large enough \cite{Yu:hyplp}, while lattices in simple Lie groups of higher rank have property $FL^p$ for all $p\geq 1$ \cite{BFGM}.

Interestingly, the other possible version of property (T) in terms of $L^p$-spaces, requiring that
``almost invariant vectors imply invariant vectors for linear isometric actions'', behaves quite
differently with respect to the standard property (T); namely the standard property (T) is equivalent
to this $L^p$ version of it, for $1<p<\infty$ \cite[Theorem A]{BFGM}. This shows in particular that the
two definitions of property (T) (i.e. the fixed point definition and the almost invariant implies
invariant definition) are no longer equivalent in the setting of $L^p$ spaces, for $p$ large.

The importance of the properties $FL^p$ comes for instance from the fact that in various rigidity results known for groups with property (T), similar results requiring weaker conditions of smoothness hold for groups with property $FL^p$. See for instance  \cite{Navas:BullBMS06}, where the theorem of reduction of cocycles taking values in the group of diffeomorphisms of the circle $\mathrm{Diff}^{1+\tau } ({\mathbb{S}}^1)$ to cocycles taking values in the group of rotations is true for $\tau =\frac{1}{p}$ when the group has property $FL^p$.

Thus, the problem of estimating the maximal $p$ for which a given group has property $FL^p$ is natural and useful, and several questions can be asked
related to this. To begin with, we note that for every group $\Gamma $ with property (T) the set $\frakF (\Gamma )$ of positive real numbers $p$ for which $\Gamma $ has $FL^p$ is open \cite{DrutuKapovich}. Let $\wp (\Gamma )$ be the supremum of the set $\frakF (\Gamma )$, possibly infinite. 

\begin{question}[\cite{Bourdon:FLp}, $\S 0.2$, Question 2; \cite{CDH-T}, Question 1.9]\label{ques1}
\ 

\begin{itemize}
\item[(a)] Do there exist, for any $p_0\geq 2$, groups such that $\frakF (\Gamma )$ contains $(0, p_0)$ and $\wp (\Gamma )$ is finite~?  

\medskip

\item[(b)] Do there exist groups as above that moreover fail to have $FL^p$ for all $p\geq \wp (\Gamma )$, and eventually have proper actions on $L^p$--spaces for $p\geq \wp (\Gamma )$~? 

\medskip

\item[(c)]Does $\wp (\Gamma )$ have any geometric significance ? 
\end{itemize}
\end{question}

Up to now, the only known examples of groups with property (T) that fail to have $FL^p$ for all $p$ larger than some $p_0$ are the hyperbolic groups. 
In particular, for hyperbolic groups the question about the geometric significance of $\wp (\Gamma )$ can be made more precise. 

\begin{question}[\cite{Bourdon:FLp}, $\S 0.2$, Question 2; \cite{CDH-T}, Question 1.9]\label{ques2}
When $\Gamma $ is hyperbolic and with property (T), is $\wp (\Gamma )$ equal to the conformal dimension of $\partial_\infty \Gamma $~?
\end{question} 

\medskip

Most examples of hyperbolic groups with property (T) come from the theory of random groups, hence it is natural to consider the questions above in the particular setting of random groups. It is what we undertake in this paper: a study of random groups from the viewpoint of the properties $FL^p$, both in the triangular model and in the Gromov density model. 

\subsection{Random groups and fixed point properties}

We follow the notation of~\cite{ALS-15-random-triangular-at-third}. Also, in what follows we write $f\simeq g$ for two real functions $f,g$ defined on a subset $A\subseteq \R$ if there exists $C>0$ such that $f(a) \leq C g(Ca+C)$ and $g(a) \leq C f(Ca+C)$, $\forall a\in A$.

The source of the theory of random groups is in the work of Gromov \cite{Gromov:Asymptotic,Gromov:random}, and in the context of the triangular model and of property (T) it has been reformulated by \.{Z}uk \cite{Zuk:propertyT}. 

The triangular model of random groups that appears the most often in the literature is the \emph{triangular density model} $\cM(m, d)$, defined for a density $d \in (0,1)$. This is the model in which, for a fixed set of generators $S$, with $|S| = m$, a set of $(2m-1)^{3d}$ relations $R$ is chosen uniformly and independently at random, among all the subsets of this cardinality in the full set of cyclically reduced relators of length 3 (with cardinality $\simeq m^3$). (As is standard, quantities such as $(2m-1)^{3d}$ are rounded to the nearest integer.)  The groups  $\Gamma = \langle S | R \rangle$ are the elements composing the model.  For more details on this model, we refer to~\cite{Zuk:propertyT} and~\cite{KK-11-zuk-revisited}.

A variation of this model, which is an analog for random groups of the Erd\"os--Renyi model of random graphs, is the following.

\begin{definition}\label{def:triangMod}
	Let $\rho$ be a function defined on $\N$ and taking values in $(0,1)$.
	For every $m\in \N$,  the 
	\emph{binomial triangular model} $\Gamma(m,\rho)$ 
	is defined by taking a finite set of generators $S$ with $|S| = m$, and groups $\Gamma = \langle S | R \rangle$, where 
	$R$ is a subset of the set of all $\simeq m^3$ possible 
	cyclically reduced relators of length 3,
	each relator chosen independently with probability $\rho (m)$.
	
	A property $P$ holds \emph{asymptotically almost surely (a.a.s.)}
	in this model if
	\[ \lim_{m \ra \infty} \bP( \Gamma \in \Gamma(m,\rho) \text{ satisfies } P) =1.\]
\end{definition}

One of our main theorems is the following.

\begin{theorem}\label{thm:Itri-eventually-flp}
	For any $\delta >0$ there exists $C>0$ so that for $\rho = \rho(m) \geq m^{\delta }/m^2$, and for every $\varepsilon >0$ a.a.s.\ 
	a random group in the binomial triangular model $\Gamma(m,\rho)$ has  $FL^p_{(2-2\varepsilon)^{1/2p}}$ for every $p \in \left[ 2, C (\log m/\log\log m)^{1/2} \right]$.
In particular, a.a.s.\ we have $FL^p$ for all $p$ in this range.
\end{theorem}

The model $\Gamma(m,\rho)$ is closely related to the density model
$\cM(m, d)$, when $\rho m^3 \simeq (2m-1)^{3d}$.
Property $FL^p$ is preserved by quotients, in particular by adding more relations,
so it is a ``monotone property'' in the sense of \cite[Proposition~1.13]{JLR-00-random-graphs} (see Section~\ref{sec:mono-confdim}).  
Thus, general results on random structures mean that our theorem implies the following in the density model $\cM(m,d)$.
\begin{corollary}\label{cor:Itri-density-flp}
	For any fixed density $d>1/3$ there exists $C >0$ so that for every $\varepsilon >0$ a.a.s.\ a random group in the triangular density model $\cM(m,d)$ has $FL^p_{(2-2\varepsilon)^{1/2p}}$  for every $p \in \left[ 2, C (\log m/\log\log m)^{1/2} \right]$. 
In particular, a.a.s.\ we have $FL^p$ for all $p$ in this range.
\end{corollary}

In the case of $FL^2$, that is, property (T), this is a result of \.{Z}uk \cite{Zuk:propertyT},
with steps clarified by Kotowski--Kotowski~\cite{KK-11-zuk-revisited}.

Note that for any density $d<1/2$ a random group in  $\cM(m,d)$ is hyperbolic \cite{Zuk:propertyT}.  

The picture drawn by Theorem \ref{thm:Itri-eventually-flp} is completed by the results of Antoniuk, Luczak and \'{S}wi\c{a}tkowski \cite{ALS-15-random-triangular-at-third}, improving previous estimates of \.{Z}uk \cite{Zuk:propertyT}, and stating that:

\begin{itemize}
\item there exists a constant $\kappa$ such that if $\rho \leq \frac{\kappa }{m^2}$, then a.a.s.\ a group in the model $\Gamma(m,\rho)$ is free; 
\item there exist constants $\kappa_1, \kappa_2$ such that if $\frac{\kappa_1 }{m^2}\leq \rho \leq \frac{\kappa_2 \log m }{m^2}$, then a.a.s.\ a group in $\Gamma(m,\rho)$ is neither free nor with property (T); 
\item  there exists a constant $\kappa_3$ such that if $ \rho \geq \frac{\kappa_3 \log m }{m^2}$, then a.a.s.\ a group in $\Gamma(m,\rho)$ has property (T).   
\end{itemize}
In the first two of these cases, the failure of property (T) implies that the groups have none of the $FL^p$ properties~\cite{BFGM}.  In the third case we show a result like Theorem~\ref{thm:Itri-eventually-flp}, with a bound growing a little slower than $(\log\log m)^{1/2}$, see Theorem~\ref{thm:tri-eventually-flp2}.

As far as Corollary \ref{cor:Itri-density-flp} is concerned, \.{Z}uk had already proven \cite{Zuk:propertyT}  that for any density $d<1/3$,  a random group in the triangular density model $\cM (m,d)$ had free factors, and hence property (T) and all $FL^p$ properties fail.  We give some partial information at $d=1/3$ in Section~\ref{sec:mono-confdim}.  Note that the results in \cite{ALS-15-random-triangular-at-third} for the first two cases do not immediately apply here, since the properties they deal with are not monotone. 

\medskip

Another model for random groups is the Gromov model, that we define here first in a generalized form, and then in the usual Gromov density model form.

\begin{definition}[Gromov model]\label{def:gromov-model}
Consider a function $f:\N \to \N $, a fixed integer $k\geq 2$ and a fixed set of generators $\cA$ with $|\cA|=k$.

A random group in the \emph{Gromov model} $\cG(k,l,f )$ is a group $\Gamma = \langle \cA | R \rangle$ with presentation defined by a collection $R$ of cyclically reduced relators of length $l$, $R$ of cardinality $f(l)$, chosen randomly with uniform probability.

When $f(l)$ is the integral part of $(2k-1)^{dl}$ for a fixed constant $d \in (0,1)$, the Gromov model becomes the usual \emph{Gromov density model at density $d$}, for which we use here a specific simplified notation, $\cD (k,l,d)$.   
\end{definition}

Like the triangular model, the Gromov model has a version that is closer to the Erd\"os--Renyi model for graphs.

\begin{definition}[Gromov binomial model]\label{def:gromov-model-binom}
	Fix a number of generators $k\geq 2$, a set of generators $\cA$ with $|\cA|=k$, and a function $\rho :\N \to (0,1)$.
	
	A group $\Gamma = \langle \cA | R \rangle$ in the
	\emph{$k$-generated Gromov binomial model} $\cB (k,l,\rho)$ is defined by taking $R$ a collection of cyclically
	reduced relators of length $l$ in the alphabet $\cA$, each chosen independently with probability $\rho(l)$.
	
	A property $P$ holds \emph{asymptotically almost surely (a.a.s.)}
	in this model if
	\[
		\lim_{l \ra \infty} \bP(\Gamma \in \cB (k,l,\rho) \text{ satisfies } P) = 1.
	\]
\end{definition}

\begin{remark}\label{rmk:gromov-rho-d-monotone}

 For a fixed  number of generators $k\geq 2$, a fixed density $d\in(0,1)$, and the function $\rho(l)=(2k-1)^{-(1-d)l}$, the model $\cB (k,l,\rho)$ is closely related to the Gromov density model $\cD (k,l,d)$, since there are $\asymp (2k-1)^{dl}$ cyclically reduced words of length $l$ in $R$, where $A \asymp B$ means $\frac{1}{C}A \leq B \leq CA$, for some constant $C>0$.
	See Section~\ref{sec:mono-confdim} for more details.
\end{remark}

In the density model $\cD (k,l,d)$ as well, when $d<1/2$ a random group is non-elementary hyperbolic \cite[Chapter 9]{Gromov:Asymptotic}. When $d>1/3$ a random group moreover has property (T) \cite{Zuk:propertyT,KK-11-zuk-revisited}. Unlike in the triangular case though, it is not known whether $1/3$ is the threshold density for property (T). J. Mackay and P. Przytycki proved in \cite{Mac-P} that when $d<5/24$ a random group acts on a finite dimensional CAT(0)-cubical complex with unbounded orbits, hence it does not have property (T). This improves a previous result of Ollivier--Wise \cite{Oll-W} for density $d<1/5$. For density $d<1/6$, Ollivier--Wise moreover proved in  \cite{Oll-W} that a random group acts properly on a CAT(0)-cubical complex, hence it is a-T-menable.

We prove the following. 

\begin{theorem}\label{thm:Igromov-flp}
	Choose $p \geq 2$, $\eps >0$ arbitrary small and $k \geq 10 \cdot 2^p$.
	Fix a density $d>1/3$.
	Then a.a.s.\ a random group in the Gromov density model $\cD (k,l,d)$ has $FL^{p'}_{(2-2\eps )^{1/2p'}}$ for all $2 \leq p'
	\leq p$.
	In particular, a.a.s.\ we have $FL^{p'}$ for all $p'$ in this range.
\end{theorem}

Note that the results of Mackay--Przytycki and of Ollivier--Wise mentioned previously imply that, below density $5/24$, random groups act on spaces with measured walls with unbounded orbits, respectively below density $1/6$ random groups have proper actions on spaces with measured walls. These results, and a standard argument that can be found for instance in \cite[Lemma 3.10]{CDH-T}, imply that, for $d<5/24$, a random group has actions with unbounded orbits on $L^p$--spaces, for the whole range  $p\in (0, \infty )$; respectively that, below density $1/6$, a random group has proper actions on $L^p$--spaces for every $p\in (0,\infty )$.      

Theorem \ref{thm:Igromov-flp} follows from the corresponding theorem in the Gromov binomial model $\cB (k,l,\rho)$, see Theorem~\ref{thm:gromov-flp}.  For any fixed $k \geq 2$ is it natural to expect a result where $p \ra \infty$ as in Theorem~\ref{thm:Itri-eventually-flp}, however our methods currently do not show this.
We do find a new proof of property (T) for any fixed $k\geq 2$ and $d >1/3$, which moreover applies at $d=1/3$ as well (see Theorem \ref{thm:gromov-T} for a precise statement).

Previous progress on the problem of $FL^p$--properties with $p>2$ for random groups in the Gromov density model had been made by P.~Nowak in \cite{Nowak-15-PI-Banach} (see Remark \ref{rmk:nowak}).

\medskip

In the class of groups with property (T), the subclass of hyperbolic groups plays a special role, since by \cite[$\S III.3$]{Oll-05-rand-grp-survey} and \cite{Cornulier:hypT} every countable group with property (T) is the quotient of a torsion-free hyperbolic group with property (T). Therefore, Theorems \ref{thm:Itri-eventually-flp} and \ref{thm:Igromov-flp} may be seen as an indication that the generic countable groups with property (T) also have $FL^p$ for $p$ in an arbitrarily large interval $(2, p_0)$.

\subsection{Conformal dimension}

Another setting emphasizing the interest of the properties $FL^p$ lies in their connection with P.\ Pansu's conformal dimension.  For a hyperbolic group $\Gamma$, the boundary $\bdry \Gamma$ comes with a canonical family of metrics; the infimal Hausdorff dimension among these is the \emph{conformal dimension} $\Confdim(\bdry \Gamma)$.  This is an invariant of the group, and in fact, if two hyperbolic groups are quasi-isometric then they have the same conformal dimension.  For more details, see~\cite{MacTys-10-Confdim-book}.

Conformal dimension can sometimes be used to distinguish hyperbolic groups even if their boundaries are homeomorphic, see Bourdon~\cite{Bourdon97-Fuchsian-buildings}.  For random groups in the Gromov density model at densities $d<1/8$, the second author has found sharp asymptotics for the conformal dimension using small cancellation methods~\cite{Mac-12-conf-dim-rand-groups,Mac-15-conf-dim-rand-groups}.  

However, small cancellation methods completely fail for random groups at densities $d>1/4$, and certainly do not work for random groups in the triangular models.  Therefore it is of interest that we are able to bound the conformal dimension in a new way at densities $d>1/3$ using the $FL^p$ properties.

As mentioned above, Bourdon showed that if a Gromov hyperbolic group has property $FL^p$ for some $p >0$, then
	the conformal dimension of its boundary is at least $p$.  
A consequence of this inequality, an upper bound computation, and Corollary \ref{cor:Itri-density-flp} is the following.

\begin{theorem}\label{thm:Iconfdim-tri-density-both-bounds}
	For any density $d\in (\frac{1}{3},\frac{1}{2})$, 
	there exists $C>0$ so that a.a.s.\ 
	$\Gamma \in \cM(m,d)$ is hyperbolic, and satisfies
	\[
	\frac{1}{C} \left(\frac{\log m}{\log\log m}\right)^{1/2} \leq \wp (\Gamma )\leq  \Confdim(\bdry \Gamma) \leq C \log m.
	\]
	In particular, as $m \ra \infty$, the quasi-isometry class of $\Gamma$ keeps changing.
\end{theorem}

% \begin{remarks}
% \begin{enumerate}
% \item

\begin{remark}\label{rmk:nowak} P.\ Nowak has also obtained a lower bound for the parameter $\wp (\Gamma )$ and hence conformal dimension, that can be explicitly calculated, in the triangular and in the Gromov density models, using spectral methods~\cite[Corollary 6.4]{Nowak-15-PI-Banach}.  However, his bound is an explicit decreasing function slightly larger than $2$.
\end{remark}  
\begin{remark}	Theorem \ref{thm:Iconfdim-tri-density-both-bounds} provides in particular a positive answer to Question \ref{ques1}(a). 
The first such example, also among hyperbolic groups, was provided by Naor and Silberman~\cite[Theorem 1.1]{NaorSilberman}. Theorem \ref{thm:Iconfdim-tri-density-both-bounds} brings the additional information that the situation described in  Question \ref{ques1}(a), is in fact generic in this standard model of random groups.  In view of the remark following  Theorem \ref{thm:Igromov-flp}, it is expected that this same situation is generic for the whole class of countable groups with property (T).
\end{remark}

%\medskip

%\item 
\begin{remark}	A consequence of Theorem \ref{thm:Iconfdim-tri-density-both-bounds} is that for a generic hyperbolic group $\Gamma $ in the model $\cM(m,d)$ with $d\in (\frac{1}{3},\frac{1}{2})$, there exists a constant $\kappa = \kappa (d)$ such that 
\begin{equation}
\frac{1}{\kappa } \left[ \Confdim(\bdry \Gamma)\right]^{1/2-\eps}\leq \wp (\Gamma )\leq  \Confdim(\bdry \Gamma),
\end{equation}  
where $\eps>0$ is fixed.
This illustrates that a formula relating $\wp (\Gamma )$ and $\Confdim(\bdry \Gamma)$  is plausible, in particular an equality as conjectured in Question \ref{ques2}.
\end{remark} %\medskip

\subsection{Random graphs and strong expansion}

Our results on random groups rely on spectral results on random graphs. Indeed, every finitely presented group  $\Gamma$ has a presentation in which all relators are of length three, and every such presentation yields an action of $\Gamma$ on a simplicial $2$-complex $X$, the Cayley complex. The link of every vertex is a graph $L(S)$, and if the smallest positive eigenvalue $\lambda_1(L(S))$ 
of the Laplacian of this graph satisfies $\lambda_1(L(S)) > \frac{1}{2}$,
then $\Gamma$ has property (T).
This has been shown by \.Zuk and Ballmann-\'Swi\c atkowski~\cite{Zuk:propertyT,Ballmann-S(1997)}, and appears implicitly in \cite{Gromov:random}.
(See Sections~\ref{sec:evals-intro} and \ref{sec:fixed-pt-prop}.)

In the case of a random group $\Gamma \in \Gamma(m,\rho)$, the link
graph is nearly a union of three random graphs coming from a suitable random
graph model.  

There is a large literature on the first positive eigenvalue of the Laplacian
of a random graph.  In the case of constant degree the problem is equivalent to bounding the second largest eigenvalue of the adjacency matrix, and this opens up methods used by Friedman to give very precise asymptotics.
In our context the bound $\lambda_1(L(S)) > \frac{1}{2}$ follows from
a result of Friedman and Kahn--Szemer\'edi~\cite{FKS-89-eigenvalue-rand-graph}
that random graphs have $\lambda_1$ close to $1$.

In our setting we must replace the Laplacian by a non-linear generalization of it, the $p$--Laplacian, for $p\in (1,\infty )$, see Section~\ref{sec:evals-intro}. The $p$-Laplacian has been used in combinatorics and computer science \cite{BuhlerHein} and turns out to be a useful tool for estimates of the $L^p$--distortion \cite{JolissaintValette:lp}.

We use a sufficient condition for property $FL^p_L$, described in the theorem below, which can be obtained by slightly modifying arguments of Bourdon. The latter arguments use Garland's method of harmonic maps, initiated in \cite{Garland}, developed by \.{Z}uk \cite{Zuk:CRAS} and Wang \cite{Wang:JDiff98}, and further used and developed by Ballmann-\'Swi\c atkowski~\cite{Ballmann-S(1997)}, Pansu \cite{Pansu:Garland} Gromov \cite{Gromov:random}, Izeki, Nayatani and Kondo \cite{IzekiNayatani05,IKN09,IKN12} etc.

Here, given a graph $L$ we denote by $\lambda_{1,p}(L)$ of a graph $L$ the first positive eigenvalue
of the $p$-Laplacian of $L$ (Definition~\ref{def:p-eigen}).

\begin{theorem}[Bourdon \cite{Bourdon:FLp}]\label{thm:bourdon-flp}

Let $p \in (1,\infty)$ and $\varepsilon < \frac12$. Suppose $X$ is a simplicial $2$-complex where the 
	link $L(x)$ of every vertex $x$ has $\lambda_{1,p}(L(x)) > 1-\varepsilon$, and has at most $m$ vertices.
	If a group $\Gamma$ acts on $X$ simplicially, properly, and cocompactly, 
	then $\Gamma$ has the property $FL^p_{m+1,(2-2\varepsilon)^{1/2p}}$.
\end{theorem}

Bounding $\lambda_{1,p}(L)$ away from zero corresponds to showing that 
$L$ is an expander, but in changing $p$ we can lose a lot of control, see Proposition~\ref{prop:bipartite-est}.  So to show that $\lambda_{1,p}$ is as close to $1$ as we wish, we have to prove new results for random graphs.

Given $m \in \N$ and $\rho \in [0,1]$, let 
$\bG(m,\rho)$ be the model of simple random graphs on $m$ vertices, where each pair of vertices
is connected by an edge with probability~$\rho$.
 
\begin{theorem}\label{thm:G-m-rho-eval-bound}
Given a function $\chi : \N \to (0,\infty )$ with $\lim_{m\to \infty } \chi (m) =0$, for every $\xi >0$ and every $p\geq 2$ there exists positive constants $\kappa = \kappa (\xi )$, $C = C (\xi )$ and $C'= C'(\xi, \chi )$, such that the following holds.

For every $m \in \N$ and every $\rho$ satisfying 
$$\frac{\kappa \log m}{m}\leq \rho \leq \frac{\chi (m)m^{1/3}}{m}$$ 
 we have that with probability at least $1-\frac{C'}{m^\xi}$ a graph $G \in \bG(m, \rho)$ satisfies  
	\[
		\forall p' \in [2,p], \quad \lambda_{1,p'}(G) \geq 1-
			\frac{C p^4 }{(\rho m)^{1/2p^2}}
			- \frac{ C\sqrt{\log m}}{(\rho m)^{1/2}}\,\bI_{p'<3}, 
	\]
	where $\bI_{p'<3}=1$ if $p'<3$ and $\bI_{p'<3}=0$ otherwise.
\end{theorem}

The methods of Friedman are not available in this non-linear situation, but Kahn--Szemer\'edi's approach does adapt, as we discuss further in Section~\ref{sec:rand-graphs}.

\subsection{Recent result for a larger class of Banach spaces}

About a year after this paper has been finished, Tim de Laat and Mikael de la Salle proved in \cite{LaatSalle} that, given a uniformly curved Banach space $X$, for any density $d>\frac{1}{3}$, a.a.s. a random group in the triangular density model ${\mathcal{M}}(m,d)$ has the fixed point property $F_X$ (i.e. every action of such a group by affine isometries on $X$ has a global fixed point). Uniformly curved Banach spaces were introduced by G.~Pisier in \cite{Pisier:MemAMS}, examples of such spaces are $L^p$--spaces, interpolation spaces between a Hilbert and a Banach space, their subspaces and equivalent renormings. A uniformly curved space $X$ has the following key property. Given a finite graph $G$ with set of vertices $G_0$ and set of edges $G_1$, $G_0$ equipped with the stationary probability measure $\nu$ for the random walk on $G$, defined by $\nu (x)= \frac{\val(x)}{\sum_{y\in G_0} \val(y)}$ (where $\val (x)$ denotes the valency of the vertex $x$), the norm of the Markov operator $A_G$ on $L^2_0(G_0,\nu;X)$ is small provided that the norm of the Markov operator $A_G$ on $L^2_0(G_0,\nu)$ is small. (Here by $L^2_0$ we mean square integrable functions with expectation zero.)  

The outline of the proof of the de Laat-de la Salle theorem is as follows. They use, like {\.Z}uk \cite{Zuk:propertyT} and Kotowski--Kotowski~\cite{KK-11-zuk-revisited}, the permutation model for groups and, correspondingly, the configuration model for random graphs, and a theorem of Friedman stating that for a random graph $G$ in the latter model a.a.s. the norm of the Markov operator $A_G$ on $L^2_0(G_0,\nu)$ is small, provided that the number of permutations taken is large enough.

It follows that, for a random group $\Gamma$ in the permutation model (with a large enough number of permutations), given the simplicial complex $\Delta_\Gamma$ of the corresponding triangular presentation of $\Gamma$, a.a.s. for every vertex link $L$ in  $\Delta_\Gamma$, the norm of the Markov operator $A_L$ on $L^2_0(L_0,\nu;X)$ is small, uniformly in $L$. The space $X$ being uniformly curved, it is also superreflexive, hence by a result of Pisier \cite{Pisier:1975} it admits an equivalent norm that is $p$--uniformly convex, for some $p\in [2, \infty )$, and preserved by the isometries of the initial norm on the space $X$. Thus the statement is reduced to the case when $X$ is $p$--uniformly convex, and $\Gamma$ is a group that acts properly discontinuously cocompactly on a simplicial complex  $\Delta_\Gamma$ with the property that for all the vertex links $L$ the Markov operators $A_L$ on $L^2_0(L_0,\nu;X)$ have uniformly small norms. An adaptation of an argument of Oppenheim \cite{Oppenheim(2014)} can then be applied to conclude that the random group $\Gamma$ must have property~$F_X$.

In the particular case of $L^p$--spaces the theorem of de Laat-de la Salle gives that for every $p_0\geq 2$, for any density $d\in \left( \frac{1}{3},\frac{1}{2} \right)$, a.a.s. a group $\Gamma$ in the triangular model ${\mathcal{M}}(m,d)$ satisfies all the fixed point properties $FL^p$ with $p\in (0, p_0]$. We believe that if, instead of using the permutation model for groups, respectively, the configuration model for random graphs, and Friedman's Theorem, de Laat-de la Salle would use the binomial triangular model for groups, the Erd\"os-Renyi model for random graphs and the estimate in Theorem 1.17 for $p=2$, then they would obtain a version of Corollary 1.7 with a slightly larger interval, that is with $p\in [2, C (\log m)^{\frac{1}{2}}]$. Therefore, in Theorem 1.11 one could suppress the $\log \log m$ from the denominator in the lower bound, and in Remark 1.14 the exponent in the first term in (1.15) would become $\frac{1}{2}$ instead of  $\frac{1}{2}-\epsilon$. 

We think nevertheless that the proof provided in this paper has its own intrinsic value, firstly because it relies on elementary mathematics only, it is self contained and independent of Pisier's results, and secondly because we find it intriguing that by two different approaches approximately the same lower bound estimates are obtained. This may suggest that the first inequalities in Theorem 1.11 and Remark 1.14 may in fact be asymptotic equalities.       

\subsection{Plan of the paper}

Section \ref{sec:evals-intro} is an introduction to the $p$--Laplacian, with several interpretations and estimates of its first non-zero eigenvalue.

In Sections \ref{sec:rand-graphs} to \ref{sec:heavy-bound}, Theorem \ref{thm:G-m-rho-eval-bound} is proven, by reducing the problem to a small enough upper bound to be obtained for a finite number of sums varying with the set of vertices, then by splitting each sum into light and heavy terms, and estimating separately the two sums of light, respectively heavy terms.

Section \ref{sec:fixed-pt-prop} links values of $\lambda_{1,p}$ to the properties $FL^p_{m,L}$.  This is then used in Section \ref{sec:fp-rand-tri-group} to show the results on random groups in the triangular model, deduced from Theorem \ref{thm:G-m-rho-eval-bound}.

We describe how to use monotonicity to switch between models, and the application to conformal dimension in Section~\ref{sec:mono-confdim}.

In Sections \ref{sec:multi-partite} and \ref{sec:Gromov}, the same strategy is applied to prove a similar result of generic $p$-expansion for multi-partite graphs, and the latter is then applied to prove Theorem \ref{thm:Igromov-flp}.   

\subsection{Notation}
We use the standard asymptotic notation, which we now recall. When $f$ and $g$ are both
real-valued functions of one real variable, we write
 $f= O(g)$ to mean that there exists a constant $L>0$ such that $f(x) \leq L g(x)$ 
 for every $x$; in particular $f= O(1)$ means that $f$ is uniformly bounded, 
 and $f= g +O(1)$ means that $f-g$ is uniformly bounded. 
 The notation $f= o(g)$ means that $\lim_{x\to \infty } \frac{f(x)}{g(x)}=0\, $.

\subsection{Acknowledgements}
We thank Damian Orlef for pointing out an issue with the $p\in(2,3)$ case in an earlier version of the paper.
We also gratefully thank the referee for helpful suggestions.

%%%%%%%%%%%%%%%%%%%%%%%%%%%%%%%%%%%%%%%%%%%%%%%%%%%%%%%%%%%%%%%%%
\section{Eigenvalues of $p$-Laplacians}\label{sec:evals-intro}

In what follows $G$ is a graph, possibly with loops and multiple edges (a multigraph). When the graph has no loops or multiple edges and the edges are considered without orientation we call it {\emph{simple}}.
Let $G_0$ be its set of
vertices and $G_1$ its set of edges. Given two vertices $u,v$, we write $u\sim v$ if there exists (at least) one edge with endpoints $u,v$ and we say that $u,v$ are {\emph{neighbours}}.

Fix an arbitrary orientation on the edges of $G$, so that each edge $e \in G_1$
has an initial endpoint $e_-$ and a target endpoint $e_+ $ in $G_0$.
Given a function $x:G_0 \ra \R$,
the {\emph{total derivative}} of $x$ is defined as $dx: G_1 \ra \R$, $dx(e) = x(e_+)-x(e_-)$.
For $e \in G_1$, we write the unordered set (with multiplicities) of endpoints of 
$e$ as $\vtx(e) = \{e_-,e_+\}$.  Observe that $|dx(e)|$, or indeed any symmetric function
of $e_-$ and $e_+$, is independent of the choice of orientation of $e \in G_1$.

Fix $p \in (1, \infty)$.
Given $x \in \R$, we define
$\{x\}^{p-1} = \mathrm{sign} (x) |x|^{p-1}$ when $x\neq 0$, and we set $\{0\}^{p-1} = 0$.
The \emph{graph $p$-Laplacian} on $G$ (see \cite{Amg-03-p-laplacian,BuhlerHein}) is an operator from $\R^{G_0}$
to $\R^{G_0}$ defined by  
\[
	(\Delta_p x)(u) = \frac{1}{\val(u)} \sum_{e \in G_1,\ \vtx(e) = \{u,v \}} \{x_u - x_v\}^{p-1} \mbox{ for every }u \in G_0,
\]
where $\val(u)$ is the valency of $u$. 
The operator $\Delta_p$ is linear only when $p=2$. 
Still, by abuse of language, one can define eigenvalues and eigenfunctions which 
serve the purpose in the $L^p$-setting as well.

\begin{definition}\label{def:p-eigen}
	We say $\lambda \in \R$ is an \emph{eigenvalue} of $\Delta_p$ for $G$ if there 
	exists a non-zero function
	$x \in \R^{G_0}$ so that $\Delta_p x = \lambda \{x\}^{p-1}$. 
	We call such a function $x$ an \emph{eigenfunction} of $\Delta_p$.
	
	We denote by $\lambda_{1,p}(G)$ the smallest eigenvalue of $\Delta_p$ which corresponds to a non-constant eigenfunction.
\end{definition}

The standard (normalised) graph Laplacian $\Delta=\Delta_2$ can equivalently be defined using 
a weighted inner product on $\R^{G_0}$.
Consider the degree sequence $\bd = (d_u) \in \N^{G_0}, d_u=\val(u)$,
and define 
$\left\langle x, y \right\rangle_\bd = \sum_{u\in G_0} x_uy_u d_u$.
Then for $x \in \R^{G_0}$, $\Delta$ is the linear operator such that
$$
\left\langle x, \Delta x \right\rangle_\bd =\|dx\|_2^2\, ,
$$ where 
the norm on the right hand side is on $\R^{G_1}$.
When the right hand side becomes $\|dx\|_p^p = \sum_{e \in G_1} |dx(e)|^p$,
the same equality defines $\Delta_p$~\cite[Section 3]{BuhlerHein}; consequently all eigenvalues are $\geq 0$. Note that in \cite[Section 3]{BuhlerHein} the equality definining $\Delta_p$ is $\langle x, \Delta_p x \rangle_\bd = \dfrac{1}{2}\| dx\|^p_p$. The reason is that in that paper, what stands for $\| dx\|^p_p$, also denoted by $Q_p (f)$, is a sum where each term $| x(e_+) - x(e_-)|^p$ appears twice (in other words, no orientation is chosen on the edges).  

The value of $\lambda_{1,p}$ for a multigraph $G$ may be calculated as follows.
The {\emph{Poincar\'e $p$-constant}} $\pi_p$ is defined as in the classical case 
to be the minimal constant $\pi $ such that for every function $x\in \R^{G_0}$,
$$
\inf_{c \in \R } \sum_{u\in G_0 } |x_u - c |^p \val (u) \leq \pi \| dx \|_p^p\, .
$$
We will use the following Rayleigh Quotient characterisation of 
$\lambda_{1,p}(G)$~\cite[Theorem~1]{Amg-03-p-laplacian}, see also
\cite[Theorem~3.2]{BuhlerHein} and \cite[Proposition~1.2]{Bourdon:FLp}.
Note that the constant functions are eigenfunctions with eigenvalues $0$, and that for every $p>1$, the minimal eigenvalue for non-constant functions, $\lambda_{1,p}(G)$, is $0$ if and only if $G$ is disconnected; in this case we interpret $\pi_p$ as $\infty$.
\begin{proposition}\label{prop:lp-bourdon}
	Fix $p \in (1, \infty)$ and a multigraph $G$.  Then
	\begin{align}\label{eq:lp-1}
		\lambda_{1,p}(G) = \frac{1}{\pi_p} 
		& = \inf \left\{ \frac{ \| dx \|_p^p }
			{ \inf_{c \in \R} \sum_{u \in G_0} |x_u-c|^p \val(u)} :
			\ x \in \R^{G_0}\, \mbox{non-constant}  \right\}
	\\
	\label{eq:lp-11}
		& = \inf \left\{ \frac{ \| dx \|_p^p }
			{ \| x \|_{p,\bd}^p} :
			\ x \in \R^{G_0} \setminus \{0\}, 
			\sum_{u\in G_0} \{x_u\}^{p-1} \val(u) = 0 \right\}
	\\
	\label{eq:lp-12}
	& =  \inf \left\{  \| dx\|_p^p  :
			\ x \in S_{p,\bd}(G_0) \right\},
	\end{align}
	where in the above $\| x \|_{p,\bd}^p$ stands for $\sum_{u \in G_0} |x_u|^p \val(u)$, for the degree sequence $\bd = (d_u) \in \N^{G_0}, d_u=\val(u)$, and
	\begin{equation}\label{eq:spd}
		S_{p,\bd}(G_0) = \left\{ x \in \R^{G_0} :
			\sum_{u \in G_0} \{x_u\}^{p-1} d_u = 0, 
			\ \| x \|_{p,\bd}^p = 1 \right\}.	
	\end{equation}
\end{proposition} 

\subsection{Varying $p$}
Later we need the following estimate on how $\lambda_{1,p}(G)$ varies as a function of $p$.
\begin{lemma}\label{lem:eval-right-lower-semicts}
	For a graph $G$, $\lambda_{1,p}(G)$ is a right lower semi-continuous function of $p$.
	To be precise, for $p \geq p' \geq 2$,
	\[
		\lambda_{1,p}(G) \geq E^{1-p/p'} \lambda_{1,p'}(G)^{p/p'},
	\]
	where $E$ is the number of edges in $G$.
\end{lemma}
\begin{proof}
	Let $x \in \R^{G_0}$ be a non-constant function which attains $\lambda_{1,p}(G)$ in \eqref{eq:lp-1}, i.e.,
	\[
		\lambda_{1,p}(G) = \frac{ \| dx \|_p^p }
			{ \inf_{c \in \R} \sum_{u \in G_0} |x_u-c|^p \val(u)}.
	\]
	Now, let $x' = x+c$ where $c$ is a constant chosen so that 
	$\sum_{u \in G_0} \{x'_u\}^{p'-1} \val(u) = 0$, i.e., $c$ is the unique minimiser
	of the convex function $c \mapsto \sum_{u \in G_0} |x_u - c|^{p'} \val(u)$.
	
	Let $y$ be a scaled copy of $x'$ so that 
	\[
		1 = \sum_{u \in G_0} |y_u|^{p'} \val(u) 
		= \inf_{c \in \R} \sum_{u \in G_0} |y_u - c|^{p'} \val(u),
	\]
	where the last equality follows from $\sum_{u\in G_0} \{y_u\}^{p'-1} \val(u)=0$.
	In particular, for each $u \in G_0$, $|y_u| \leq 1$ and thus $|y_u|^p \leq |y_u|^{p'}$.
	So
	\begin{align*}
		\lambda_{1,p}(G) 
		& =  \frac{ \| dx \|_p^p }{ \inf_{c \in \R} \sum_{u \in G_0} |x_u-c|^p \val(u)} \\
		& \geq \frac{ \| dx' \|_p^p }{ \sum_{u \in G_0} |x'_u|^p \val(u)} 
		 = \frac{ \| dy \|_p^p }{ \sum_{u \in G_0} |y_u|^p \val(u)} \\
		& \geq \frac{ \| dy \|_p^p }{ \sum_{u \in G_0} |y_u|^{p'} \val(u)} 
		 = \| dy \|_p^p.
	\end{align*}
	Now H\"older's inequality gives
	\[
		\| dy \|_{p'}^{p'} 
		= \sum_{e \in G_1} |dy(e)|^{p'}
		\leq \left(\sum_{e \in G_1} |dy(e)|^p \right)^{p'/p} E^{1-p'/p}
		= \| dy\|_{p}^{p'} E^{1-p'/p},
	\]
	so
	\begin{align*}
		\lambda_{1,p}(G) 
		& \geq \| dy \|_{p'}^p E^{1-p/p'}
		= E^{1-p/p'} \left(\frac{ \|dy\|_{p'}^{p'}}{1}\right)^{p/p'}
		\\ & = E^{1-p/p'} \left(\frac{ \|dy\|_{p'}^{p'}}
			{\inf_{c \in \R} \sum_{u \in G_0} |y_u - c|^{p'} \val(u)}
			\right)^{p/p'}
		\\ & \geq E^{1-p/p'} \cdot \lambda_{1,p'}(G)^{p/p'} \qedhere
	\end{align*}
\end{proof}

\section{Bounding $\lambda_{1,p}$ for random graphs}\label{sec:rand-graphs}

Given $m \in \N$ and $\rho \in [0,1]$, recall that a random graph in the model $\bG(m,\rho)$ is a simple graph on $m$ vertices, with each pair of vertices connected by an edge with probability $\rho$.

Our goal, from now until the end of Section~\ref{sec:heavy-bound}, is to show the following bound on $\lambda_{1,p}$ for a random graph in this model.
\begin{varthm}[Theorem \ref{thm:G-m-rho-eval-bound}.]
Given a function $\chi : \N \to (0,\infty )$ with $\lim_{m\to \infty } \chi (m) =0$, for every $\xi >0$ and every $p\geq 2$ there exists positive constants $\kappa = \kappa (\xi )$, $C = C (\xi)$ and $C'= C'(\xi, \chi )$, such that the following holds.

For every $m \in \N$ and every $\rho$ satisfying 
$$\frac{\kappa \log m}{m}\leq \rho \leq \frac{\chi (m)m^{1/3}}{m}$$ 
 we have that with probability at least $1-\frac{C'}{m^\xi}$ a graph $G \in \bG(m, \rho)$ satisfies  
	\[
		\forall p' \in [2,p], \quad \lambda_{1,p'}(G) \geq 1-
			\frac{C p^4 }{(\rho m)^{1/2p^2}}
			- \frac{ C\sqrt{\log m}}{(\rho m)^{1/2}}\,\bI_{p'<3}, 
	\]
	where $\bI_{p'<3}=1$ if $p'<3$ and $\bI_{p'<3}=0$ otherwise.
\end{varthm}

In fact, we prove lower bounds on $\lambda_{1,p}(G)$ when $G$ is chosen
from a more restrictive random graph model, $\bG(m,\bd)$.
For convenience, we let $G \in \bG(m,\bd)$ have vertex set $G_0 = \{1,2,\ldots,m\}$. Let $\bd = (d_i) \in \N^m$ denote a sequence of vertex degrees, where we assume that
$\sum d_i$ is even (a necessary condition). The random graph model $\bG(m,\bd)$ is defined by letting $G \in \bG(m,\bd)$ be chosen
uniformly at random from all simple graphs with this degree sequence.
For example, in the case that $d_i=d$ for all $i$, this is the model of random $d$-regular graphs.

\begin{theorem}\label{thm:G-m-deg-eval-bound}
Consider a constant $\theta \geq 1$ and a function $\chi : \N \to (0,\infty )$ satisfying $\lim_{m\to \infty } \chi (m) =0$.

Then for every $\xi>0$ there exists $C$ and $C'$ depending on $\theta , \xi$, with $C'$ moreover depending on the function $\chi$, so that for every $m \in \N$ and $p \geq 2$, and every degree sequence $\bd \in \N^m$ with $\sum_i d_i$ even and $\min_i d_i \geq 3$, with moreover $d = \max_i d_i \leq \theta \min_i d_i$, and ${d} \leq \chi (m)m^{1/3}$,
	\begin{multline}\label{eq:lambdad}
\bP\Big( G \in \bG(m,\bd) \text{ has } \forall p' \in [2,p], \\ \lambda_{1,p'}(G) \geq 1
- \frac{Cp^4}{d^{1/2p^2}} -8\bI_{p'<3}(1-\theta^{-2})\Big) \geq 1 - \frac{C'}{m^\xi}.	
	\end{multline}
\end{theorem}

Theorem \ref{thm:G-m-deg-eval-bound} implies the result in $\bG(m,\rho)$, when combined
 with the following lemma.

\begin{lemma}\label{lem:degree-concentration}
	For any $\delta >0$, with probability $\geq 1-2m \exp \left(- \frac{1}{3}\delta^2 m \rho\right)$
	a graph $G \in \bG(m,\rho)$ has the property that
	every vertex has degree within $\delta m\rho$ of the expected value $m\rho$.
\end{lemma}
\begin{proof}

	Let $X_i = \val_G(i)$ be the degree of $i \in G_0$, which is a binomial
	random variable with $\bE(X_i)=(m-1)\rho = (1+o(1))m\rho$.
	Then using a standard Chernoff bound~\cite[Corollary 2.3]{JLR-00-random-graphs},
	we have
	\begin{align*}
		\bP(\exists i : |X_i - \bE X_i| > \delta \bE X_i)
		& \leq m \bP(|X_1- \bE X_1 |> \delta \bE X_1)
		\leq 2m \exp\left( - \frac{\delta^2}{3} \bE X_1\right)
		\\ & \leq 2m \exp\left( - \frac{1}{3}\delta^2 m \rho \right)\, . \qedhere
	\end{align*}
\end{proof}

\begin{proof}[Proof of Theorem~\ref{thm:G-m-rho-eval-bound}]
Fix arbitrary $\xi >0$.

	Lemma~\ref{lem:degree-concentration} applied for some small $\delta$ implies that there exists $\kappa$ so that, provided $\rho$ is greater than $\kappa \log(m)/m$, 
	the ratio between the minimum and maximum degrees of 
	$G$ is bounded by a constant $\theta =(1+\delta)/(1-\delta) \leq 2$ and the degree is $(1+o(1)) \rho m$, with probability at least $1-{2}/\big({m^{-1+{\delta^2\kappa}/{3}}}\big)$. The latter probability is at least $1-{1}/{m^\xi}$ if $\kappa$ is large enough. 

	The same argument gives that, as long as $\rho\geq \kappa\log(m)/m$, we can choose $\delta = \sqrt{3(\xi+1)\log(m)/\rho m}$.  In this case:
	\[
	1-\theta^{-2} = \frac{(1+\delta)^2-(1-\delta)^2}{(1+\delta)^2}
	\leq 4\delta \leq B \sqrt{\frac{\log m}{\rho m}}
	\]
	where $B = 4\sqrt{3(\xi+1)}$.
	
All graphs $G \in \bG(m,\rho)$ with $E$ edges arise with the same probability, namely $\rho^E (1-\rho)^{\binom{m}{2}-E}$.
Consequently, for a degree sequence $\bd$ with $\sum d_i = 2E$,
all graphs $G \in \bG(m, \bd)$ have the same probability of arising
in $\bG(m,\rho)$.

	For every degree sequence $\bd$ in $[d/\theta,d]^m$, the inequality \eqref{eq:lambdad} gives,
	with a probability at least $1-C'/m^{\xi}$ (uniform in $\bd$),
	that $G \in \bG(m, \bd)$ has $\inf_{p' \in [2,p]}\lambda_{1,p'}(G)$ 
	greater than 
	$1-C''p^4/(\rho m)^{1/2p^2}-B \bI_{p'<3} (\log m)^{1/2} / (\rho m)^{1/2}$ (where $C'$ and $C''$ depend only on $\xi$, and $C'$ further depends on $\chi$).
	Therefore, we get our desired bound in $\bG(m, \rho)$.
\end{proof}

To approach Theorem~\ref{thm:G-m-deg-eval-bound}, we consider again 
the characterisation \eqref{eq:lp-12} of $\lambda_{1,p}$ 
when we have a fixed degree sequence $\bd = (d_i)$ for a graph $G$.
For every $x\in \R^{G_0}$ we have
\[
	\sum_{e \in G_1 : \vtx(e) = \{u,v\}} \left( |x_u|^p + |x_v|^p \right) 
	= \sum_{u \in G_0} |x_u|^p d_u = \|x\|_{p,\bd}^p.
\]
So we can rewrite $Z_x(G) = \|dx\|_p^p$ in \eqref{eq:lp-12} as 
\begin{align}\label{eq:lp-2}
	Z_x(G) & = \|x\|_{p,\bd}^p - \sum_{e \in G_1, \vtx (e) = \{u,v \}} (|x_u|^p + |x_v|^p - |x_u-x_v|^p)
\end{align}
This motivates the following notation.
\begin{notation}\label{notat:P1}
Given a real number $p\geq 2$ and two real numbers $a,b$, we define
\begin{align*}
\Re_p(a,b)  = |a|^p + |b|^p - |a-b|^p.
\end{align*}
\end{notation}

Using this notation we can write
\[ X_x(G) = \|x\|_{p,\bd}^p- Z_x(G) = \sum_{e \in G_1, \vtx (e) = \{u,v \}} \Re_p (x_u,x_v).\]
In the particular case that $x \in S_{p,\bd}(G_0)$, $Z_x(G) = 1- X_x(G)$.

Therefore, to prove Theorem~\ref{thm:G-m-deg-eval-bound}
it suffices to show that with high probability $X_x (G)$ is bounded from above by a 
suitable uniform small positive term, for all $x \in S_{p,\bd}(G_0)$.
This is proved using a variation of the Kahn-Szemer\'edi 
method~\cite{FKS-89-eigenvalue-rand-graph} for bounding 
$\lambda_{1,2}$, which is roughly as follows:
every $x\in S_{p,\bd}(G_0)$ can be approximated by some function $x'$ in 
a suitable finite net, and if the approximation is accurate enough then 
it suffices to show that with high enough probability $X_{x'} (G)$ has a uniform small positive upper bound
for every $x'$ in the net, see Section~\ref{sec:net-approx}.  The reason for switching from the Erd\"os-Renyi model $\bG(m,\rho)$ to the prescribed degree model $\bG(m,\bd)$ is that in our case this net is defined in terms of the vertex degrees $\bd$.
For each point $x'$ in this net, the terms in $X_{x'}(G)$ split into small and large values, and the two contributions are bounded independently.  We discuss this further in sections~\ref{sec:G-m-d-overview}--\ref{sec:heavy-bound}.

We remark that Kahn and Szemer\'edi worked in the permutation model for random regular graphs.
However, their method was adapted to the model $\bG(m,\bd)$ by 
Broder--Frieze--Suen--Upfal~\cite[Theorem~7]{BFSU-99-rand-graphs}, and it is their
proof that we follow more closely.

\section{Approximating on finite sets}\label{sec:net-approx}
In this section we define a net of points approximating well enough the points in the set $S_{p,\bd}$, we provide bounds on the size of this net, and we show that good enough bounds on an infimum defined as in \eqref{eq:lp-12} but with $S_{p,\bd}$ replaced by the net suffice to bound $\lambda_{1,p}$.

\subsection{The net and its size}
Suppose we have a graph $G$ with vertex set $G_0=\{1,2, \ldots, m\}$ and 
degree sequence $\bd = (d_i) \in \N^{m}$, with $d = \max_i d_i$.
Recall that $S_{p,\bd}(G_0)$ is the set of $x \in \R^{m}$
with $\sum_{i} \{x_i\}^{p-1} d_i = 0$, and 
$\|x\|_{p,\bd}^p = \sum_i |x_i|^p d_i = 1$.
For any $R\geq 1$ and small enough constant $\eps>0$, 
we define a corresponding finite net that will be used to approximate $S_{p,\bd}(G_0)$.
\begin{multline*}
	T_{p,\bd,R}(G_0) = \bigg\{ x \in \R^{m} :  \ 
		\forall i,\ \{x_i\}^{p-1} \in  \frac{\eps d^{1/p}}{d_i m^{1/q}} \Z, \ 
		\sum_{i \in G_0} \{x_i\}^{p-1} d_i = 0, \\
		\| x \|_{p,\bd}^p \leq R \ \bigg\}
		.
\end{multline*}
Here we follow:
\begin{cvn}
We let $q$ denote the H\"older conjugate $\frac{p}{p-1}$ of $p$.
\end{cvn}
Throughout all that follows, to simplify estimates we assume $\eps$ satisfies:
\begin{assumption}\label{assump:eps-theta}
	We have $\eps \theta \leq 1$.
\end{assumption}
Recall that $\theta \geq \max d_i/d_j$ and is close to $1$ in our applications.

Later we will take $R = (1+\eps \theta^{1/p})^q$, which by 
Assumption~\ref{assump:eps-theta} satisfies $R \leq 4$.  

We need to know the size of $T_{p,\bd,R}(G_0)$.  Before we bound this, it is
helpful to recall the following.
\begin{lemma}\label{lem:volume-q-ball}
	There exists $m_0$ so that for all $p \geq 2$ and $m \geq m_0$,
	the volume $V_q(R)$ of the radius $R$ ball in $\R^m$ endowed with the norm $\|\cdot \|_q$ is bounded by
	\[
		V_q(R) \leq \left( \frac{2eR}{m^{1/q}} \right)^{m}.
	\]
\end{lemma}
\begin{proof}
	It suffices to consider $R=1$, where
	\begin{align}\label{eq:size-t-1}
		V_q(1) = \frac{\left(2\Gamma\left(\frac{1}{q}+1 \right)\right)^m}{\Gamma \left( \frac{m}{q}+1 \right)}.
	\end{align}
	Since $1 \leq (1/q)+1 \leq 2$, we have $\Gamma\left( \frac{1}{q}+1 \right) \leq 1$.
	Moreover, for $m \geq m_0$, where $m_0 \geq 2$ is independent of $p \geq 2$,
	Stirling's approximation
	$\Gamma(1+z) / \sqrt{2 \pi z} (\frac{z}{e})^z \ra 1$ (as $|z|\ra\infty$) gives us
	\[
		\Gamma\left(\frac{m}{q}+1\right) \geq \sqrt{\frac{m}{q}} \left(\frac{m}{qe}\right)^{m/q}.
	\]
	Applying this to \eqref{eq:size-t-1}, we see that
	\[
		V_q(1) \leq \frac{\sqrt{q}}{\sqrt{m}} 
			\left( \frac{2(eq)^{1/q}}{m^{1/q}} \right)^m
		\leq \left(\frac{2e}{m^{1/q}}\right)^m
			.\qedhere
	\]	
\end{proof}

\begin{proposition}\label{prop:size-of-t}
	Suppose the degree sequence $\bd=(d_i) \in \N^m$ satisfies
	$ d_i \geq \frac{1}{\theta}d,\, \forall i,$ where $\theta \geq 1$ and $d = \max_i d_i$.
	Then the size of $T_{p,\bd,R}(G_0)$ is bounded by
	\[
		|T_{p,\bd,R}(G_0)| \leq \left(\frac{4e R}{\eps}\right)^m.
	\]
\end{proposition}
\begin{proof}
	Consider the set 
	\[
		T' = \left\{ y \in \R^m : y_i \in \frac{\eps d^{1/p}}{d_i m^{1/q}} \Z, \ 
			\sum |y_i|^q d_i \leq R \right\}.
	\]
	We inject $T$ into $T'$ by mapping $x \in T_{p,\bd,R}(G_0)$
	to $y \in T'$, where for each $i$, $y_i = \{x_i\}^{p-1}$.
	
	For each $y \in T'$, let 
	$Q_y= \left\{z \in \R^m : y_i < z_i < y_i + \frac{\eps d^{1/p}}{d_im^{1/q}}\right\}$.
	Clearly, each $Q_y$ has volume 
	\begin{equation}\label{eq:qycubevol}
		V(Q)=\frac{1}{\prod d_i} \left(\frac{\eps d^{1/p}}{m^{1/q}}\right)^m,
	\end{equation}
	and $Q_y \cap Q_{y'}= \emptyset$ for $y \neq y'$.
	If $z \in Q_y$, then Minkowski's inequality for the weighted 
	norm $\|z\|_{q,\bd} = \left(\sum_i |z_i|^q d_i\right)^{1/q}$ shows that
	\begin{align*}
		\| z \|_{q,\bd} & \leq \|y\|_{q,\bd} + \frac{\eps d^{1/p}}{m^{1/q}} \cdot 
			\left(\sum_i d_i^{-q} d_i\right)^{1/q}
			\leq R^{1/q} + \frac{\eps d^{1/p}}{m^{1/q}} \cdot \frac{\theta^{1/p} m^{1/q}}{d^{1/p}}
			\leq R + \eps \theta^{1/p},
	\end{align*}
	so each $Q_y$ is contained in the $R'$ ball $B$
	in $\R^n$ with the norm $\|\cdot\|_{q,\bd}$,
	where
	\[
		R' = R + \eps \theta^{1/p} \leq R + 1 \leq 2R
	\]
	by Assumption~\ref{assump:eps-theta}.
	This ball $B$ is an affine transformation of the ball $V_q(R')$, so by 
	Lemma~\ref{lem:volume-q-ball} it has volume
	\begin{equation}\label{eq:affineballvol}
		V(B) = \Big( \prod_i d_i^{-1/q} \Big) V_q(R') 
		\leq \Big( \prod_i d_i^{-1/q} \Big) \left(\frac{2eR'}{m^{1/q}}\right)^m \ .
	\end{equation}
		
	We combine \eqref{eq:qycubevol} and \eqref{eq:affineballvol} to conclude:
	\begin{align*}
		|T| & \leq |T'| \leq \frac{V(B)}{V(Q)}
			\leq \Big( \prod_i d_i^{-1/q} \Big) \left(\frac{2eR'}{m^{1/q}}\right)^m 
			\cdot \Big(\prod_i d_i\Big) \left(\frac{\eps d^{1/p}}{m^{1/q}}\right)^{-m} \\
			& = \left(\frac{ \prod_i d_i}{d^m}\right)^{1/p} 
				\left(\frac{2eR'}{\eps}\right)^m
			\leq \left(\frac{2eR'}{\eps}\right)^m. \qedhere
	\end{align*}
\end{proof}

\subsection{Bounds on the net suffice} 
The following proposition shows that to bound $Z_x(G)$ for $x \in S_{p,\bd}(G_0)$, it suffices
to bound $X_x(G)$ for $x \in T_{p,\bd,R}(G_0)$.
\begin{proposition}\label{prop:suffices-bound-t2}
	Let $R = (1+\eps\theta^{1/p})^q$ and $R_- = (1-\eps\theta^{1/p})^q$.
	For any $x \in S_{p,\bd}(G_0)$ there exists $x' \in T_{p,\bd,R}(G_0)$ with
	$\|x'\|_{p,\bd}^p \geq R_-$, such that if $Z_x(G) \leq 1$ then 
	$|Z_x(G) - Z_{x'}(G)| \leq 2p(\eps\theta)^{1/(p-1)}(1+2(\eps\theta)^{1/(p-1)})^{p-1}$.

	In particular, if $X_{x'}(G) \leq \eta$ for every $x' \in T_{p,\bd,R}(G_0)$
	then for every $x \in S_{p,\bd}(G_0)$ we have
	\begin{align*}
	Z_x(G) & \geq 1 - \eta - 4p(\eps\theta)^{1/(p-1)} \left( 1+ 2(\eps\theta)^{1/(p-1)} \right)^{p-1}.
	\end{align*}
\end{proposition}
\begin{proof}
	Suppose $x \in S_{p,\bd}(G_0)$ is given.  
	Inspired by the proof of \cite[Lemma 14]{BFSU-99-rand-graphs}, for each $i$, write
	\[
		\{x_i\}^{p-1} = \frac{\eps d^{1/p}}{d_i m^{1/q}} \cdot k_i + r_i,
	\]
	for some $k_i \in \Z$ and $r_i \in [0, \eps d^{1/p} d_i^{-1} m^{-1/q})$.
	Since $x \in S_{p,\bd}(G_0)$,
	\begin{equation}\label{eq:shift-x-to-xprime-remainder}
		0 = \sum_i \{x_i\}^{p-1}d_i 
		= \frac{\eps d^{1/p}}{m^{1/q}} \Big(\sum_i k_i\Big) + \sum_i r_i d_i,
	\end{equation}
	and so $\sum_i r_i d_i = r \eps d^{1/p} m^{-1/q}$ for some $r \in \Z$,
	in fact $r \in \{0,1,\ldots,m-1\}$. 
	We define $x' \in \R^m$ by setting
	\begin{equation*}
		\{x'_i\}^{p-1} = \begin{cases}
		      	\eps d^{1/p} d_i^{-1} m^{-1/q} (k_i +1) & \text{if } i \leq r, \text{ or} \\
		      	\eps d^{1/p} d_i^{-1} m^{-1/q} k_i & \text{if } i > r.
		      \end{cases}
	\end{equation*}
	By \eqref{eq:shift-x-to-xprime-remainder} we have
	\[
		\sum_i \{x'_i\}^{p-1} d_i = \frac{\eps d^{1/p}}{m^{1/q}} \Big(\sum_i k_i\Big) + 
			\frac{\eps d^{1/p}}{m^{1/q}} \cdot r  = 0.
	\]
	
	We now bound the size of $x'$ in the norm $\|\cdot\|_{p,\bd}$, using
	weighted H\"older's inequalities.
	\begin{align*}
		\|x'\|_{p,\bd}^p & =
		\sum_i |x'_i|^{p-1}|x'_i| d_i
		\leq \sum_i |x_i|^{p-1}|x'_i|d_i + 
			\sum_i \frac{\eps d^{1/p}}{d_i m^{1/q}} \cdot |x'_i|d_i 
		\\ & \leq \left(\sum_i |x_i|^p d_i \right)^{1/q}
			\left(\sum_i |x'_i|^p d_i \right)^{1/p} +
			\frac{\eps d^{1/p}}{m^{1/q}}
			\left(\sum_i \frac{1}{d_i^q} d_i \right)^{1/q}
			\left(\sum_i |x'_i|^p d_i \right)^{1/p}
		\\ & \leq \|x\|_{p,\bd}^{p/q} \| x'\|_{p,\bd} +
			\frac{\eps d^{1/p}}{m^{1/q}} \cdot 
			\left( m \cdot \frac{\theta^{q-1}}{d^{q-1}}\right)^{1/q}
			\| x'\|_{p,\bd}
		\\ & = \left( \|x\|^{p-1}_{p,\bd} + \eps \theta^{1/p}\right) \| x'\|_{p,\bd}
			= \left( 1 + \eps \theta^{1/p}\right) \| x'\|_{p,\bd}, 
	\end{align*}
	and so 
	\[
		\|x'\|_{p,\bd}^p = \|x'\|_{p,\bd}^{(p-1)q} \leq (1+\eps \theta^{1/p})^q=R.
	\]
	Likewise, $\| x \|_{p,\bd}^{p-1} \leq \left( \|x'\|^{p/q}_{p,\bd} + \eps \theta^{1/p}\right)$,
	and so
	\[
		\|x'\|_{p,\bd}^p \geq (1-\eps \theta^{1/p})^q = R_-.
	\]
	
	It remains to bound $|Z_{x'}(G)-Z_x(G)|=\big| \|dx'\|_{G,p}^p - \|dx\|_{G,p}^p \big|$.  
	Recall that, by construction, for each $i$
	we have 
	\[
		\big|\{x_i\}^{p-1}-\{x_i'\}^{p-1}\big| \leq 
		\frac{\eps d^{1/p}}{d_i m^{1/q}} \leq \frac{\eps \theta}{(dm)^{1/q}}.
	\]	
	Now, if we have $a,b,\del \geq 0$ with $0 \leq a^{p-1} \leq b^{p-1} \leq a^{p-1}+\delta$,
	i.e.\ $|b^{p-1}-a^{p-1}| \leq \delta$, then since $p \geq 2$,
	$a \leq b \leq (a^{p-1}+\del)^{1/(p-1)} \leq a+\del^{1/(p-1)}$.
	Since $x_i$ and $x'_i$ are either both non-positive or both non-negative, we find that $|x_i-x'_i| \leq {(\eps\theta)^{1/(p-1)}}/{(dm)^{1/p}}$.
	This implies that $\left|dx'(e)-dx(e)\right| \leq 2(\eps\theta)^{1/(p-1)}/(dm)^{1/p}$.

By the Mean Value Theorem applied to $t \ra |t|^p$, 
for each $e \in G_1$ there exists $t(e)\in\R$ with $|t(e)| \leq 2(\eps\theta)^{1/(p-1)}/(dm)^{1/p}$ 
so that
\begin{align*}
	\big| |dx'(e)|^p - |dx(e)|^p \big|
	& = p \left| dx'(e)-dx(e) \right| \cdot \left| dx(e)+t(e) \right|^{p-1}
	\\ & \leq \frac{2p(\eps\theta)^{1/(p-1)}}{(dm)^{1/p}} \left| dx(e)+t(e) \right|^{p-1}.
\end{align*}
Therefore, by H\"older's and Minkowski's inequalities, and our assumption~$Z_x(G)\leq 1$,
\begin{equation}\label{eq:dx-diff-bound} \begin{aligned}
	\big| \|dx'\|_{G,p}^p - \|dx\|_{G,p}^p \big| 
	& \leq  \frac{2p(\eps\theta)^{1/(p-1)}}{(dm)^{1/p}} \sum_{e\in G_1} \left| dx(e)+t(e) \right|^{p-1} \\
	& \leq  \frac{2p(\eps\theta)^{1/(p-1)}}{(dm)^{1/p}} \left( \sum_{e \in G_1} \left| dx(e)+t(e) \right|^{p}\right)^{(p-1)/p} \left(dm\right)^{1/p} \\
	& \leq {2p(\eps\theta)^{1/(p-1)}} \left( \big\|dx\big\|_{G,p}+ \big\|t\big\|_{G,p} \right)^{p-1} \\ 
	& \leq 2p(\eps\theta)^{1/(p-1)} (1+ 2(\eps\theta)^{1/(p-1)})^{p-1}.
\end{aligned}\end{equation}
	
	The final remark follows from the following argument. If $Z_x(G) \geq 1$ then it is trivial.  Otherwise,
	for $p \geq 2$, $R_- \geq (1-\eps\theta^{1/p})^2 \geq 1-2\eps\theta^{1/p}$.
	So if $X_{x'}(G) \leq \eta$ for every $x' \in T_{p,\bd,R}(G_0)$
	then for every $x \in S_{p,\bd}(G_0)$ we have
	\begin{align*}
	Z_x(G) & \geq Z_{x'}(G) - 2p(\eps\theta)^{1/(p-1)} (1+ 2(\eps\theta)^{1/(p-1)})^{p-1}
	\\ & = \|x'\|_{p,\bd}^p- X_{x'}(G)-2p(\eps\theta)^{1/(p-1)} (1+ 2(\eps\theta)^{1/(p-1)})^{p-1}
	\\ & \geq R_- - \eta - 2p(\eps\theta)^{1/(p-1)} (1+ 2(\eps\theta)^{1/(p-1)})^{p-1}
	\\ & \geq 1- \eta - 4p(\eps\theta)^{1/(p-1)} (1+ 2(\eps\theta)^{1/(p-1)})^{p-1}.\qedhere
	\end{align*}
\end{proof}

\subsection{Preliminary bounds}\label{ssec:prelim-bounds}
The quantity $\Re_p(a,b) = |a|^p+|b|^p-|a-b|^p$ is difficult to work with, so in the following sections we have occasions to use more convenient quantities, described below.
\begin{notation}\label{notat:P2}
Given a real number $p\geq 2$ and two real numbers $a,b$, we define
\begin{align*}
\tilR_p(a,b) & = \{a\}^{p-1} b + a \{ b \}^{p-1}, \text{  and  }
 \\ \ovP_p(a,b) & = |a|^{p-1}|b| \vee |a||b|^{p-1},
\end{align*}
where $x \vee y$ denotes the maximum of $x$ and $y$.
\end{notation}

For example, $\Re_2(a,b)=2ab =\widetilde{\Re}_2 (a,b)$ and $\ovP_2(a,b) =|a||b|$.

\begin{proposition}\label{prop:p-estimate}
	For every $a, b \in \R$ the following hold.
  \begin{gather}\label{eq:p-ovp}
		\ovP_p(a,b) \leq |\Re_p(a,b)| \leq \left(1+ p2^{p-1}\right) \ovP_p(a,b), \text{ and}
  \\ \label{eq:p-tvp}
		\tilR_p(a,b) \leq |\widetilde{\Re}_p(a,b)| \leq 2 \ovP_p(a,b).
	\end{gather}
	Moreover, for $p\geq 3$ we have
	\begin{equation}\label{eq:p-tvp-fix}
	  \Re_p(a,b) \leq p \tilR_p(a,b).
	\end{equation}
\end{proposition}
\begin{proof}

 For every $a,b \in \R$ and $\lambda >0$, and for $F= \Re_p$, $F=\ovP_p $ or  
 $F= \widetilde{\Re}_p$ we have
	\[
		F(a,b)=F(-a, -b) = F(b,a) = \lam^{-p} F(\lam a, \lam b).
	\]
	It therefore suffices to show all inequalities for $a=1$ and $-1 \leq b \leq 1$.
	In this case $\ovP_p(1,b) = |b|$ and  $\widetilde{\Re}_p(1,b) = b+ \{ b\}^{p-1}$.

\eqref{eq:p-ovp}\quad The second inequality is immediate:
we apply the Mean Value Theorem to the function
$t\mapsto t^p$ to find $x$ between $1$ and $1-b>0$ so that $1-(1-b)^p = bpx^{p-1}$; in particular $0 \leq x \leq 2$.  Therefore,
\begin{align*}
|\Re_p(1,b)| & = \left| 1+ |b|^p -(1-b)^p \right| = \left| |b|^p +bp x^{p-1} \right|
\\ & \leq |b|^p + |b| p2^{p-1}\leq |b|\left(1+ p2^{p-1}\right)\, .
\end{align*}

We now prove the first inequality.
	
	If $b \geq 0$ we have $|\Re_p(1,b)| = 1+b^p - (1-b)^p \geq 1+0-(1-b) = b$.

	If $b < 0$, we have $|\Re_p(1,b)| = (1-b)^p-1-(-b)^p \geq (1-pb)-1-(-b) = (p-1)|b|$.	 

\medskip

\eqref{eq:p-tvp}\quad This inequality is trivial.

\medskip

\eqref{eq:p-tvp-fix}\quad
Assume that $b \geq 0$, and so $\widetilde{\Re}_p(1,b) = b+ b^{p-1}$.  Applying the Mean Value Theorem as above, we find $x$ satisfying $0 \leq 1-b \leq x \leq 1$ so that $\Re_p(1,b) = 1+b^p -(1-b)^p = b^p + bpx^{p-1}\leq b^p + bp\leq p\widetilde{\Re}_p(1,b)\, .$

Assume now that $b<0$ and let $c=-b \in (0,1]$. Then $\widetilde{\Re}_p(1,b) = -(c+ c^{p-1})$ while  $\Re_p(1,b) = 1+c^p -(1+c)^p$, so we want to show that $(1+c)^p \geq 1+pc+pc^{p-1}+c^p$.
For $p\geq 3$ we have that $t\mapsto t^p$ is in $C^3([1,1+c])$ so the Lagrange form of the remainder in Taylor's theorem gives that there exists $y \in (1,1+c)$ with
\[
  (1+c)^p = 1+pc + \frac{p(p-1)}{2}c^2+\frac{p(p-1)(p-2)}{6}y^{p-3}c^3.
\]
For $p\geq 3$ and $c \in [0,1]$ we have
$\frac{p(p-1)}{2} c^2 \geq p c^{p-1}$
and $\frac{p(p-1)(p-2)}{6}y^{p-3}c^3 \geq c^p$ so we are done.
\end{proof}

\begin{cvn}
In what follows we frequently drop the index $p \geq 2$ 
from Notations \ref{notat:P1} and \ref{notat:P2}.
\end{cvn}

%%%%%%%%%%%%%%%%%%%%%%%%%%%%%%%%%%%%%%%%%%%%%%%%%%%%%%%%%%%%%%%%%%%%%%%%

\section{Bounding $X$ on the net}\label{sec:G-m-d-overview}

Given $x \in T_{p,\bd,R}(G_0)$, 
consider a random graph $G $ in the model $ \bG(m,\bd)$. In what follows we always assume that the sequence of vertex degrees $\bd = (d_i) \in \N^m$ has $\sum_i d_i$ even, $\min_i d_i \geq 3$, and $d = \max_i d_i \leq \theta \min_i d_i$ for a fixed constant $\theta \geq 1$.

Adapting ideas from Kahn--Szemer\'edi~\cite{FKS-89-eigenvalue-rand-graph, BFSU-99-rand-graphs},
we split the set of edges of $G$ into two subsets 
with respect to the function $x$, the \emph{light} and \emph{heavy} edges,
whose definitions depend on a parameter $\bet = {p}/{(2+2p)}$:
\[
	E_l = \{e \in G_1, \vtx (e) = \{u,v \} : \ovP(x_u, x_v) \leq d^\beta/dm\}\mbox{ and }
	E_h = G_1 \setminus E_l\, . 
\]
Consequently $X_x (G)$ splits into two sums, of \emph{light} and 
\emph{heavy} terms:
\[
	X_x^l(G) = \sum_{e \in E_l, \vtx (e) = \{u,v \}} \Re(x_u, x_v) 
	\text{ and } 
	X_x^h(G) = \sum_{e \in E_h, \vtx (e) = \{u,v \}} \Re(x_u, x_v).
\]
The strategy is to bound these two sums separately: the (many) light terms have small expected value and likely small deviation from that value, while the (few) heavy terms can bounded by estimating the number of edges joining groups of similarly valued vertices.
To be specific, we have the following bounds.

\begin{proposition}\label{prop:light-bound}
	For $p\geq 3$, for $\beta = p/(2+2p)$ and $d = o(m^{1/3})$
	we have that for any function $K=K(m)>0$ 
	\begin{multline*}
		\bP \left( \text{For all } x \in T_{p,\bd,R}(G_0), 
		\ X_x^l(G)  \leq p \frac{128\theta^3+K}{d^{\beta/p}} \right) 
		\\ \geq 1-
		2 \exp \left( -\frac{1}{128} K^2 m 
		+ m \log \left(\frac{16e}{\eps}\right) + o(m^{2/3})
		\right) 		
		,
	\end{multline*} 
	where $o(m^{2/3})$ represents a convergence to zero as $m \ra \infty$, independent of $K$.
	For $p \in [2,3]$, we have 
	\begin{multline*}
		\bP \left( \text{For all } x \in T_{p,\bd,R}(G_0), 
		\ X_x^l(G)  \leq R(1-\theta^{-2}) + \frac{1200\theta^3+K}{d^{\beta/p}} \right) 
		\\ \geq 1-
		2 \exp \left( -\frac{1}{6000} K^2 m 
		+ m \log \left(\frac{16e}{\eps}\right) + o(m^{2/3})
		\right).
	\end{multline*} 
\end{proposition}

\begin{proposition}\label{prop:heavy-bound}
	For $\beta=p/(2+2p)$ and $d=o(m^{1/2})$ we have that 
for any $\xi >0$, there exists $C' = C'(\theta, \xi)$ so that
	\[
		\bP\left( \text{For all } x \in T_{p,\bd,R}(G_0), 
		\ X_x^h(G) \leq \frac{C'p}{d^{\beta/p}} + \frac{C'p^4}{d^{1/p^2}}+\frac{C'p^4\eps^{-q}}{d^{1/p}}
		\right)
		\geq 1-o(m^{-\xi}),
	\]
	and moreover this probability holds on a set in $\bG(m,\bd)$ defined independently of $p$.
\end{proposition}

We postpone the proofs of Propositions \ref{prop:light-bound} and \ref{prop:heavy-bound} until sections~\ref{sec:light-bound} and \ref{sec:heavy-bound}, respectively, and in the remainder of this section we use these two propositions to prove Theorem~\ref{thm:G-m-deg-eval-bound}.

\subsection{Proof for a single $p$}\label{ssec:G-m-d-single-p}
We begin by finding a high probability bound on $\lambda_{1,p}$
for a single value of $p = p(m)$.

	By Propositions~\ref{prop:light-bound} and \ref{prop:heavy-bound},
	for any $x \in T_{p,\bd,R}(G_0)$ we have
	\[
		X_x(G) \leq p \frac{1200\theta^3+K}{d^{\beta/p}} +
		\frac{C'p}{d^{\beta/p}} + 
		\frac{C' p^4 }{d^{1/p^2}} +
		\frac{C' p^4 \eps^{-q}}{d^{1/p}}
		+\bI_{p<3} R(1-\theta^{-2}),
	\]
	where $\bI_{p<3}$ is $1$ if $p<3$ and $0$ otherwise,
	with probability at least 
	\[
		1-2 \exp \left( -\frac{1}{6000} K^2 m 
		+ m \log \left(\frac{16e}{\eps}\right) + o(m^{2/3}) \right)
		-o(m^{-\xi}).
	\]

	By Proposition~\ref{prop:suffices-bound-t2},
	this gives with the same probability that
	for any $x \in S_{p,\bd}(G_0)$
	\[
	Z_x(G) \geq 1 - p Y_x(G)-\bI_{p<3}R(1-\theta^{-2})
	\]
	where
	\[
	Y_x(G) = \frac{1200\theta^3+K}{d^{\beta/p}}+
		\frac{C'}{d^{\beta/p}} +
		\frac{C' p^3 }{d^{1/p^2}} +
		\frac{C' p^3 \eps^{-q}}{d^{1/p}}+
		4(\eps \theta)^{1/(p-1)} (1+2(\eps\theta)^{1/(p-1)})^{p-1}
	.
	\]
	Recall that $\bet = p/(2+2p)$.
	 Then there exists $C_1=C_1(\theta,\xi)$ so that
	for any $x \in S_{p,\bd}(G_0)$
	\[
	Y_x(G) \leq C_1 \left( \frac{1+K}{d^{1/(2+2p)}}+ 
	\frac{p^3}{d^{1/p^2}} + \frac{p^3\eps^{-q}}{d^{1/p}} + \eps^{1/(p-1)}(1+C_1\eps^{1/(p-1)})^{p-1} \right)
	\]
	with probability at least
	\begin{equation}\label{eq:eval-combine1}
		1-2 \exp \left( -\frac{1}{6000} K^2 m 
		+ m \log \left(\frac{C_1}{\eps}\right) +
		 o\left( m^{2/3} \right) \right)
		-o(m^{-\xi}).
	\end{equation}
	To balance the terms $\eps^{1/(p-1)}$ and $\eps^{-q}/d^{1/p}$,
	we set $\eps^{1/(p-1)} = d^{-\kappa}$ and solve 
	$\kappa = \frac{1}{p}-q(p-1)\kappa = \frac{1}{p}-p\kappa$ to find 
	$\kappa = 1/p(p+1)$; in other words, $\eps = d^{-(p-1)/p(p+1)}$.
	
	Observe that the bound on $\lambda_{1,p}$ claimed by Theorem~\ref{thm:G-m-deg-eval-bound} is vacuous unless $p^4/d^{1/2p^2}$ is small.  Therefore, because $(1+C_1\eps^{1/(p-1)})^{p-1} \leq \exp(C_1 d^{-1/p(p+1)}(p-1))$, we can assume that $(1+C_1\eps^{1/(p-1)})^{p-1} \leq 2$.

	As we are not trying to optimise for small $p\geq 2$, we 
	use $2+2p \leq 3p \leq 2p^2$, $p^2 \leq 2p^2$ and $p(p+1) \leq 2p^2$ to find
	\begin{equation}\label{eq:eval-combine1a}
		Y_x(G) \leq C_1 \left( 
		\frac{1}{d^{1/2p^2}} + \frac{K}{d^{1/3p}} + \frac{p^3}{d^{1/2p^2}} + \frac{p^3}{d^{1/2p^2}}
			+ \frac{2}{d^{1/2p^2}} \right) 
		= C_1 \left( \frac{3+2p^3}{d^{1/2p^2}} + \frac{K}{d^{1/3p}} \right).
	\end{equation}
	To have the probability \eqref{eq:eval-combine1} going to one, it suffices that
	\[
		-\frac{1}{6000} K^2 + \log \left(\frac{C_1}{\eps}\right)
		= -\frac{1}{6000} K^2 + \log C_1 + \frac{p-1}{p(p+1)} \log(d) \leq -1,
	\]
	for then the lower bound in \eqref{eq:eval-combine1} is at least $1-2e^{-m+o(m)}-o(m^{-\xi})$.
	We choose a suitably large constant $C_2$ so that
	for $K = C_2(1+\sqrt{\log(d)/p})$ we have
	\begin{equation}\label{eq:eval-combine1b}
		-\frac{1}{6000} K^2 + \log \left(\frac{C_1}{\eps}\right)
		\leq
		-\frac{1}{6000} K^2 + \log C_1 + \frac{2}{p} \log(d) \leq -1.
	\end{equation}
	
	Consider 
	\[
		\frac{\sqrt{\log(d)/p}}{d^{1/3p}} \cdot d^{1/{2p^2}}
		\leq \frac{\sqrt{\log(d)/p}}{d^{1/12p}}
		= \exp \left(-\tfrac{1}{2}\log(p) + \tfrac{1}{2}\log\log(d) - \frac{1}{12p} \log(d) \right).
	\]
	A brief calculus estimate shows this is maximised for $p = \frac{1}{6}\log(d)$,
	and so is bounded by a constant.
	Consequently, $d^{-1/3p}\sqrt{\log(d)/p}$ is bounded by a multiple of $d^{-1/2p^2}$. 
	
	Applying this to \eqref{eq:eval-combine1a} we see that, for some $C_3$,
	\[
		Y_x(G) \leq C_3 \cdot \frac{p^3}{d^{1/2p^2}}.
	\]
	and this holds for all $x \in S_{p,\bd}(G_0)$ with probability at least $1-2e^{-m+o(m)}-o(m^{-\xi})$.
	Now $Z_x(G) \geq 1-p Y_x(G)-\bI_{p<3} R(1-\theta^{-2}) \geq 1-p Y_x(G)-\bI_{p<3} 4(1-\theta^{-2})$, so
	by \eqref{eq:lp-12} we have 
	\begin{equation}\label{eq:eval-combine2}
	  \lambda_{1,p}(G) \geq 1-C_3 p^4 \cdot {d^{-1/2p^2}}-4\bI_{p<3}(1-\theta^{-2})
	\end{equation}
	with the same probability.

%%%%%%%%%%%%%%%%%%%%%%%%%%%%%%%%%%%%%%%%%%%%%%%%%%%%%%%%%%%%%%%%
\subsection{Simultaneous bounds in $p$}\label{ssec:G-m-d-simul-p}
We now get bounds on $\lambda_{1,p}(G)$ which hold
for a range of values of $p$ simultaneously.

Recall from \eqref{eq:eval-combine1},\eqref{eq:eval-combine2} above that 
	for any particular choice of $2 \leq p' \leq p$,
	and the choice of $K$ as in \eqref{eq:eval-combine1b}, we have 
	\[
		\lambda_{1,p'}(G) 
		\geq  1-C_3 (p')^4 \cdot \frac{1}{d^{1/2{(p')}^2}}-4\bI_{p'<3}(1-\theta^{-2})
		\geq 1-C_3 p^4 \cdot \frac{1}{d^{1/2p^2}}-4\bI_{p'<3}(1-\theta^{-2})
	\]
	with probability at least
	\begin{equation}\label{eq:simul-1}
		1-2 \exp \left( \left[ -\frac{1}{6000} K^2  
		+ \log C_1 + \frac{1}{p}\log(d)\right]m + o(m^{2/3}) \right)
		-o(m^{-\xi}),
	\end{equation}
	where we used $\eps \geq d^{-1/p}$.
	The last term $o(m^{-\xi})$ comes via the heavy bound Proposition~\ref{prop:heavy-bound}
	from Lemma~\ref{lem:block-edge-bound}.  
	This lemma describes properties of $G \in \bG(m,\bd)$ independent of $p$, so
	our heavy bounds will hold a.a.s.\ for all $2 \leq p' \leq p$.
	(The light bounds, however, are not independent of $p$, so the probabilities here
	decrease as we get bounds for more and more values of $p$.)
	
	Suppose we fix $2=p_0 < p_1 < \ldots < p_L = p$, with $p_{i+1}/p_i-1$ bounded by a constant
	$\tau \in(0,1)$ and $L \leq \log(p/2)/\log(1+\tau)+1$.  By \eqref{eq:simul-1} we have that for all $i=0, \ldots, L-1$ 
	\begin{equation*}
		\lambda_{1,p_i}(G) 
		\geq 1-C_3 p^4 \cdot {d^{-1/2p^2}}-4\bI_{p_i<3}(1-\theta^{-2})
	\end{equation*}
	simultaneously, with probability at least
	\begin{equation}\label{eq:simul-2}
		1-L \cdot 2 \exp \left( \left[ -\frac{1}{6000} K^2  
		+ \log C_1 + \frac{1}{p}\log(d)\right]m + o(m^{2/3}) \right)
		-o(m^{-\xi}).
	\end{equation}
	
	Since the number of edges in $G$ is at most $d m$, 
	Lemma~\ref{lem:eval-right-lower-semicts} gives
	that with probability at least as in \eqref{eq:simul-2}, we have for all $i=1,\ldots,L$ and for all $p'\in [p_i,p_{i+1}] \subset [2,p_{i+1}]$ that
	\begin{equation}\label{eq:simul-3}\begin{split}
		\lambda_{1,p'}(G) 
		& \geq (d m)^{1-p_{i+1}/p_i} \left(1-C_3 p_{i}^4 \cdot {d^{-1/2p_{i}^2}}-4\bI_{p_i<3}(1-\theta^{-2})\right)^{p_{i+1}/p_i}
		\\ & \geq (d m)^{-\tau} \left(1-C_3 p^4 \cdot {d^{-1/2p^2}}-4\bI_{p_i<3}(1-\theta^{-2})\right)^2
		\\ & \geq (dm)^{-\tau} \left(1-2C_3 p^4 \cdot d^{-1/2p^2}-8\bI_{p_i<3}(1-\theta^{-2})\right).
	\end{split}\end{equation}
	If we choose $(p_i)$ so that some $p_j$ equals $3$, then we may assume that $\bI_{p_i<3}=\bI_{p'<3}$ in the estimate above.
	We have
	\[
		(d m)^{-\tau} = \exp(-\log(dm) \tau) \geq 1-\tau \log(dm).
	\]
	Set $\tau = (\log(dm))^{-1} p^4 \cdot {d^{-1/2p^2}}$, and then
	since we may assume that $1-\tau\log(dm)\geq 0$ we conclude that
	for all $2 \leq p' \leq p$ ,
	\[
	\lambda_{1,p'}(G) \geq 1-3C_3 p^4 \cdot {d^{-1/2p^2}}-8\bI_{p'<3}(1-\theta^{-2}).
	\]
	
	It remains to bound the probability that this holds, using the lower bound \eqref{eq:simul-2}.
	We can assume that $\tau <1$.
	Therefore, $\log(1+\tau) \geq \frac{\tau}{2}$ and
	$L \leq \log(p/2) / \log(1+\tau) +1 \leq \frac{2}{\tau} \log(p/2) +1$.
	If $L=1$ we are done, so assume that $\frac{2}{\tau}\log(p/2) \geq 1$ and so $L \leq \frac{4}{\tau}\log(p/2)$.
	So by \eqref{eq:simul-2}, our probability of failure is at most $o(m^{-\xi})$ plus
	\begin{align*}
		& 8 \exp \left( \log\log(\tfrac{1}{2}p) - \log(\tau) + \left[ -\frac{1}{6000} K^2  
		+ \log C_1 + \frac{1}{p}\log(d)\right]m + o(m) \right).
	\end{align*}
	Now, since $d = o(m^{1/3})$, we have
	\begin{align*}
		\log\log(\tfrac{1}{2}p) - \log(\tau)
		& = \log\log(\tfrac{1}{2}p) - 
		\log\left(p^4 \cdot {d^{-1/2p^2}}\right) + \log\log(dm)
		\\ & \leq \log\log(\tfrac{1}{2}p) - \log(p^4)
			+ \frac{1}{2p^2} \log(d) + \log\log(m^{4/3})
		\\ & \leq \frac{1}{2p^2} \log(d) + o(m)
			\leq \frac{1}{p} \log(d) m +o(m),
	\end{align*}
	so our probability of failure is at most
	\begin{align*}
		& 8 \exp \left( \left[ -\frac{1}{6000} K^2  
		+ \log C_1 + \frac{2}{p}\log(d)\right]m + o(m) \right).
	\end{align*}
	By our choice of $K$ in \eqref{eq:eval-combine1b},
	this probability is $\leq 8 e^{-m+o(m)}$, and so 
	Theorem~\ref{thm:G-m-rho-eval-bound} is proved.	\qed

%%%%%%%%%%%%%%%%%%%%%%%%%%%%%%%%%%%%%%%%%%%%%%%%%%%%%%%%%%%%%%%%
\section{Bounding light terms}\label{sec:light-bound}

The aim of this section is to prove Proposition~\ref{prop:light-bound}, the bound on the contribution $X_x^l(G)$ of the light edges to $X_x(G)$.

Rather than working directly in the model 
$\bG(m,\bd)$, we use the 
\emph{configuration model} $\bG^*(m,\bd)$ (for an overview, 
see~\cite{Wormald-99-random-regular-survey}).
In this model the vertex set is 
$F_0 = \{ (i,s) \in \N^2 : 1 \leq i \leq m, 1 \leq s \leq d_i \}$,
and $F \in \bG^*(m,\bd)$ is a graph with vertex set $F_0$, and edge set 
$F_1$ which is a perfect matching of $F_0$, chosen uniformly at random from all such
perfect matchings.

Given $F \in \bG^*(m,\bd)$, we define a multigraph $M(F)$ with vertex set
$\{1,\ldots,m\}$, by adding an edge (or loop) between the vertices $i$ and $j$ in $M(F)$ for each edge $\{(i,s),(j,t)\}$ in $F_1$.  
Given $a = (i,s) \in F_0$, let $v(a) = i \in M(F)_0$.
We use the following two key properties of this model.
\begin{proposition}\label{prop:config-model}
	\begin{enumerate}
		\item[(a)] If $M(F)$ is simple, then it is equally likely to be any
			simple graph with degree sequence $\bd$.
		\item[(b)] Let $d = \max_i d_i$, and assume $d = o((\sum d_i)^{1/4})$.
			Then
		\[
			\bP(M(F) \text{ is simple, for } F \in \bG^*(m, \bd))
				\geq \exp(-d^2 +o(1)).
		\]
	\end{enumerate}
\end{proposition}
\begin{proof}
	Part (a) follows from the uniformity of $F \in \bG^*(m,\bd)$,
	and the fact that each simple $M(F)=G \in \bG(m,\bd)$ corresponds to the same number
	($\prod d_i!$) of pairings $F \in \bG^*(m,\bd)$.
	
	Part (b) follows from~\cite[Theorem 4.6]{McKay-85-asymptotics-random-simple}.	
\end{proof}

Recall that $X_x(G), X_x^l(G)$ and $X_x^h(G)$ sum $\Re(x_u,x_v)$ over 
endpoints of certain edges in $G$.
Let $\tilX_x(G), \tilX_x^l(G), \tilX_x^h(G)$ be the corresponding sums where $\Re$ is replaced
by $\tilR$, and $\ovX_x(G), \ovX_x^l(G), \ovX_x^h(G)$ the sums where $\Re$ is replaced by
$\ovP$ (see subsection~\ref{ssec:prelim-bounds}).

We define $\tilX_x, \tilX_x^l$ and $\tilX_x^h$ on 
$\bG^*(m,\bd)$ by extending the definition from
$G \in \bG(m,\bd)$ to $M(F)$, where $F \in \bG^*(m,\bd)$.
To be precise, define
\[
	\tilX_x(F) = \sum_{e \in F_1, \vtx(e)=\{a,b\}} \tilR\left( x_{v(a)},x_{v(b)} \right).
\]
We let $E_l = \{e \in F_1, \vtx(e)=\{a,b\} : \ovP \left( x_{v(a)},x_{v(b)} \right) \leq d^\beta/dm\}$,
and define 
\[
	\tilX_x^l(F) = \sum_{e \in E_l, \vtx(e)=\{a,b\}} \tilR\left( x_{v(a)},x_{v(b)} \right).
\]
Likewise, let $E_h = F_1 \setminus E_l$, and define $\tilX_x^h(F)$ analogously.

For $p\geq 3$, by \eqref{eq:p-tvp-fix}, $X_x^l(G) \leq p \tilX_x^l(G)$.
To bound $\tilX_x^l(G)$, we first show that for a fixed $x$ and for
$F \in \bG^*(m,\bd)$ both $\bE(\tilX_x^l)$ and $\tilX_x^l(F) - \bE(\tilX_x^l)$ 
have small upper bounds uniform in $x \in T_{p,\bd,R}(G_0)$, with probability close to $1$.  
For $p \in [2,3]$ we use a variation on this to show that both $\bE{X_x^l}$ and $X_x^l(F)-\bE(X_x^l)$ are small.
The bound on the size of $T_{p,\bd,R}(\{1,\ldots,m\})$ given by
Proposition~\ref{prop:size-of-t} then implies that this same bound holds with probability close to $1$ for all such $x$.
Finally, Proposition~\ref{prop:config-model} gives the bound for 
$G \in \bG(m,\bd)$. Further details are provided in Subsection~\ref{ssec:light-dev-proof}.

\subsection{Bounding the expected value for $p\geq 3$}
\begin{lemma}\label{lem:expectation-light}
For every $x \in \R^m$ such that $\sum_i |x_i|^p d_i \leq R$ for some $R \geq 1$ and 
$\sum_{i}\{ x_i\}^{p-1}d_i=0$,
	\[
		\bE(\tilX_x^l) \leq \frac{8 \theta^3 R^2}{d^{\beta/p}}.
	\]
\end{lemma}
To show this bound in expected value, the key step is the following lemma.
\begin{lemma}\label{lem:p-gamma-bound}
	Let $V = \{1,\ldots,m\}$ and $\bd = (d_i) \in \N^m$ with maximum value $d$,
	and minimum value at least $d/\theta$ for some $\theta\geq 1$.
	
	Suppose $x \in \R^m$ is such that $\sum |x_i|^p d_i\leq A$ for some $A \geq 1$.
	Given $\gamma > 0$, we have
	\begin{equation}\label{eq:genbound}
	\sum_{{i,j \in V :\ }{\ovP(x_i,x_j) \geq \gamma /dm}} \ovP(x_i,x_j) 
		\leq \frac{2m \theta^{2} A^2}{\gamma^{1/p}d}.
	\end{equation}
\end{lemma}
\begin{proof}
	Let $LHS$ denote the left hand side of \eqref{eq:genbound}.

	Let $V_1 = \{ (i,j) \in V \times V : |x_i|^{p-1}|x_j| \geq \gamma/dm, |x_i| \geq |x_j| \}$.
	Observe that if $\ovP(x_i, x_j) \geq \gamma/dm$, then either 
	$(i,j) \in V_1$ or $(j,i) \in V_1$ (or both).
	Therefore, by H\"older's inequality,
	\begin{align*}
		LHS 
		& \leq 2 \sum_{(i,j) \in V_1} |x_i|^{p-1}|x_j|
		 \leq 2
			\left( \sum_{(i,j) \in V_1} |x_i|^p \right)^{(p-1)/p}
			\left( \sum_{(i,j) \in V_1} |x_j|^p \right)^{1/p}.
	\end{align*}
	Clearly, the map $V_1 \ra V$ defined by $(i,j) \mapsto i$ is at most $m$ to $1$.
	On the other hand, the map $V_1 \ra V$ defined by 
	$(i,j) \mapsto j$ is at most $A\theta m/\gamma$ to $1$,
	since $\gamma/dm \leq |x_i|^{p-1}|x_j| \leq |x_i|^p$ implies that there
	are at most $A\theta m/\gamma$ possible values for $i$ because 
	$\sum |x_i|^p \leq \sum |x_i|^p d_i \theta/d\leq A\theta/d$.
	
	So we conclude that
	\begin{align*}
		LHS & \leq {2}
			\left( m \sum_{i \in V} |x_i|^p \right)^{(p-1)/p}
			\left( \frac{A\theta m}{\gamma} \sum_{j \in V} |x_j|^p \right)^{1/p}
			\\ & = \frac{2m (A\theta)^{1/p}}{\gamma^{1/p}} \sum_{i\in V} |x_i|^{p}
			\leq \frac{2m A^{1/p}\theta^{1+1/p}}{\gamma^{1/p}d} \sum_{i\in V} |x_i|^{p} d_i
			\\ & \leq \frac{2m (A\theta)^{1+1/p} }{\gamma^{1/p}d} \qedhere
	\end{align*}
\end{proof}

\begin{proof}[Proof of Lemma~\ref{lem:expectation-light}]
	Given $i,j \in M(F)_0$, let $\cE_{ij}(F)$ be the number of edges $e \in F_1$
	with endpoints $a,b \in F_0$ and $\{v(a),v(b)\}=\{i,j\}$.
	Let $E = \frac{1}{2} \sum d_i$ be the total number of edges in $F \in \bG^*(m,\bd)$.
	Each possible edge in $F_1$ appears with probability $1/(2E-1)$,
	so if $i,j \in M(F)_0, i\neq j,$ then $\bE \cE_{ij}(F) = d_id_j/(2E-1)$,
	while if $i=j \in M(F)_0$ then $\bE_{ii}(F) = \frac{1}{2}d_i(d_i-1)/(2E-1)$.
	So we can write
\begin{align*}
	\bE \tilX_x & = \bE \sum_{e \in F_1, \vtx(e)=\{a,b\}} \tilR\left( x_{v(a)},x_{v(b)} \right) \\
	& = \sum_{i<j} \bE(\cE_{ij}(F)) \tilR\left(x_{i},x_{j}\right) +
		\sum_{i} \bE(\cE_{ii}(F)) \tilR\left(x_{i},x_{i}\right) \\
	& = \sum_{i<j} \frac{d_id_j}{2E-1} \tilR\left(x_{i},x_{j}\right) +
		\sum_{i} \frac{\frac{1}{2}d_i(d_i-1)}{2E-1} \tilR\left(x_{i},x_{i}\right) \\
	& = \sum_{i,j} \frac{\frac{1}{2}d_id_j}{2E-1} \tilR\left(x_{i},x_{j}\right) -
		\sum_{i} \frac{\frac{1}{2}d_i}{2E-1} \tilR\left(x_{i},x_{i}\right) \\
	& = \frac{1}{2(2E-1)} \sum_{i,j} \left( \{x_i\}^{p-1} x_j + x_i \{x_j\}^{p-1}\right)d_id_j
		- \frac{1}{2E-1} \sum_{i} |x_i|^p d_i \\
	& = 0 - \frac{1}{2E-1} \sum_{i} |x_i|^p d_i \leq 0.	
\end{align*}
Obviously $\bE \tilX_x = \bE \tilX_x^l+\bE \tilX_x^h $, 
hence we can control $\bE \tilX_x^l$ by controlling $\bE \tilX_x^h$.
Now,
\begin{align*}
	\bE \tilX_x^h & = \sum_{i<j : \ovP(x_i,x_j) > d^\beta/dm} \frac{d_id_j}{2E-1} \tilR(x_{i},x_{j}) +
		\sum_{i : \ovP(x_i,x_i) > d^\beta/dm} 
			\frac{\frac{1}{2}d_i(d_i-1)}{2E-1} \tilR(x_{i},x_{i}) \\
	& = \sum_{i,j : \ovP(x_i,x_j) > d^\beta/dm} 
		\frac{\frac{1}{2}d_id_j}{2E-1} \tilR(x_{i},x_{j}) -
		\sum_{i : \ovP(x_i,x_i) > d^\beta/dm} \frac{\frac{1}{2}d_i}{2E-1} \tilR(x_{i},x_{i}).
\end{align*}
So, using Lemma~\ref{lem:p-gamma-bound}, we have
\begin{align*}
	|\bE \tilX_x^h| & \leq \frac{1}{2(2E-1)} \sum_{i,j : \ovP(x_i,x_j) > d^\beta/dm} 
		|\tilR(x_{i},x_{j})|d_id_j +
		\frac{1}{2E-1} \sum_{i : \ovP(x_i,x_i) > d^\beta/dm} |x_i|^p d_i \\
	& \leq \frac{d^2}{2E-1} \sum_{i,j : \ovP(x_i,x_j) > d^\beta/dm} 
		\ovP(x_{i},x_{j}) +
		\frac{R}{2E-1} \\
	& \leq \frac43 \cdot \frac{d^2 2\theta}{dm} \cdot \frac{2m\theta^2 R^2}{d^{\beta/p}d}
		+ \frac43 \cdot \frac{2R\theta}{dm} 
	\leq \frac{8 \theta^3 R^2}{d^{\beta/p}},
\end{align*}
where we assume that $E \geq 2$ so $2E-1 \geq \frac34 dm/\theta$ and that
$dm \geq d^{\beta/p}$, which is true when $\beta \leq 2$.
Finally, we have 
\begin{equation*}
	\bE \tilX_x^l \leq \bE \tilX_x + |\bE \tilX_x^h| \leq \frac{8 \theta^3 R^2}{d^{\beta/p}}. \qedhere
\end{equation*}
\end{proof}

\subsection{Bounding the expected value for $p \in [2,3]$}
\begin{lemma}
  \label{lem:expectation-light-fix}
  For $p \in [2,3]$, for every $x \in \R^m$ such that $\sum_i |x_i|^p d_i \leq R$ for some $R \geq 1$ and 
$\sum_{i}\{ x_i\}^{p-1}d_i=0$,
	\[
		\bE(X_x^l) \leq R \left(1-\theta^{-2}\right) + \frac{72\theta^3 R^2}{d^{\beta/p}}.
	\] 
\end{lemma}
\begin{proof}
  We write
  \[
  \bE X_x^l = \bE X_x-\bE X_x^h
  = \|x\|_{p,\bd}-\bE Z_x -\bE X_x^h.
  \]
  Similarly to above, by \eqref{eq:p-ovp} (using $1+p2^{p-1}\leq 13$) and Lemma~\ref{lem:p-gamma-bound} we have:
  \begin{align*}
	|\bE X_x^h|&=
	  \Bigg| \sum_{i,j : \ovP(x_i,x_j) > d^\beta/dm} 
		\frac{\frac{1}{2}d_id_j}{2E-1} \Re(x_{i},x_{j}) -
		\sum_{i : \ovP(x_i,x_i) > d^\beta/dm} \frac{\frac{1}{2}d_i}{2E-1} \Re(x_{i},x_{i})
		\Bigg|
	\\ & \leq \frac{13}{2(2E-1)} \sum_{i,j : \ovP(x_i,x_j) > d^\beta/dm} 
		\ovP(x_{i},x_{j})d_id_j +
		\frac{1}{2E-1} \sum_{i : \ovP(x_i,x_i) > d^\beta/dm} |x_i|^p d_i 
	\\
		& \leq \frac{13\cdot4}{3} \cdot \frac{d^2 2\theta}{dm} \cdot \frac{2m\theta^2 R^2}{d^{\beta/p}d}
		+ \frac43 \cdot \frac{2R\theta}{dm} 
	\leq \frac{72 \theta^3 R^2}{d^{\beta/p}}.
  \end{align*}

  Now we bound $\bE Z_x$.
  We use the lower bound \eqref{eq:lp-11} for $\|dx\|_p$ for the complete random graph $K_m$, given the value $\lambda_{1,p}(K_m) = \frac{m-2+2^{p-1}}{m-1}$ found by Amghibech (Theorem~\ref{thm:complete}).
  \begin{align*}
  \bE Z_x & = \sum_{i<j} \bE(\cE_{ij}(F)) |x_i-x_j|^p
  = \sum_{i<j} \frac{d_id_j}{2E-1} |x_i-x_j|^p
  \\& \geq \frac{d^2}{\theta^2(2E-1)} \sum_{i<j} |x_i-x_j|^p
  \\ \intertext{so, by the definition of $\lambda_{1,p}$,}
  \\ \bE Z_x & \geq \frac{d^2}{\theta^2(2E-1)} \lambda_{1,p}(K_m) \sum_i |x_i|^p(m-1)
  \\& \geq \frac{d^2 (m-2+2^{p-1})}{\theta^2(2E-1)} \sum_i |x_i|^p
  \\& \geq \frac{1}{\theta^2}\sum_i |x_i|^pd_i = \frac{1}{\theta^2} \|x\|^p_{p,\bd}.
  \end{align*}
  Combining our bounds, we have
  \[ 
  \bE X_x^l \leq \|x\|^p_{p,\bd} \left(1-\theta^{-2}\right) + \frac{72 \theta^3 R^2}{d^{\beta/p}} \leq R \left(1-\theta^{-2}\right) + \frac{72\theta^3 R^2}{d^{\beta/p}}. \qedhere
  \]
\end{proof}

\subsection{Light terms close to expected value}
Our next goal is to prove that, for fixed $x \in T_{p,\bd,R}$,
$\tilX_x^l$ is very close to its expected value.
\begin{proposition}\label{prop:deviation-light}
	For any $\alp \in (0,1)$, so that $2\beta+2\alp \leq 1$, 
	and any positive number $K > 0$, 
	the following inequality holds for every $x \in T_{p,\bd,R}(M(F)_0)$,
	\begin{equation*}%\label{eq-dev-bound-statement}
		\bP\left( | \tilX_x^l - \bE(\tilX_x^l) | \geq \frac{K}{d^{\alp}}\right) < 
		2 \exp \left( -\frac{1}{128} K^2 m \right)
		.
	\end{equation*}
\end{proposition}

The proof of this fact is similar in spirit to \cite{FKS-89-eigenvalue-rand-graph}, 
but we prove a weaker statement than they do, which suffices for our purposes.

We order the vertices of $F_0$ lexicographically: $(i,s) < (j,t)$ if $i<j$,
or if $i=j$ and $s<t$.
We now define a martingale on $\bG^*(m,\bd)$ by exposing the edges of $F$ sequentially.
First we reveal the edge connected to $(1,1)$, 
then the edge connected to the lowest remaining unconnected vertex, and so on.
This defines a filtration $(\cF_k)$, where $\cF_k$ is the $\sigma$-algebra generated by 
the first $k$ exposed edges.

Let $S_k = \bE(\tilX_x^l | \cF_k)$.  
Then $S_0 = \bE(\tilX_x^l)$, and at the end of the process we have $S_{E} = \tilX_x^l$.
To apply standard concentration estimates to $S_{E}-S_0$, 
we need to have control on the size of $S_k - S_{k-1}$. 

For simplicity, given $e \in F_1$ with $\vtx(e)=\{a,b\}$, we write
\begin{equation*}
	\tilR_l(e) = \begin{cases} \tilR \left( x_{v(a)},x_{v(b)} \right) & \text{ if } e \in E_l, \\
	             	0 & \text{ otherwise.}
	             \end{cases}
\end{equation*}
Thus $\tilX_x^l(F) = \sum_{e \in F_1} \tilR_l(e)$.

For $F,F' \in \bG^*(m,\bd)$, we write $F \equiv_k F'$ if and only if $F$ and $F'$ lie in the
same subsets of $\cF_k$, i.e., $F$ and $F'$ have the same first $k$ edges.

For a given $F \in \bG^*(m,\bd)$, we bound $|S_k(F)-S_{k-1}(F)|$ using a switching argument
(compare Wormald~\cite[Section 2]{Wormald-99-random-regular-survey}).
Suppose the $k$th edge of $F$ joins $a_1$ to $a_2$.
Let $J \subset F_0$ be $\{a_2\}$ union the set of endpoints of the remaining $E-k$ edges.
For each $b \in J$, let $\cS_b$ be the collection of $F' \in \bG^*(m,\bd)$ so that 
$F' \equiv_{k-1} F$ and $F'$ joins $a_1$ to $b$.
Then
\[
	S_k(F) = \frac{1}{|\cS_{a_2}|} \sum_{F' \in \cS_{a_2}} \tilX_x^l(F'),
\]
and
\[
	S_{k-1}(F) = \frac{1}{|J|} \sum_{b \in J} \frac{1}{|\cS_b|} \sum_{F' \in \cS_b} \tilX_x^l(F').
\]
For each $b \in J$, there is a bijection between $\cS_{a_2}$ and $\cS_b$ defined as follows:
for $F' \in \cS_{a_2}$ which joins $a_3$ to $b$, define $F'' \in \cS_b$ by deleting 
$\{a_1,a_2\}, \{a_3,b\}$ from $F'$ and adding $\{a_1, b\}, \{a_3, a_2\}$.
Since only at most two values of $\tilR_l(e)$ change, and $|\tilR_l(e)|\leq 2d^{\beta}/dm$ for
any edge $e$, we have $|\tilX_x^l(F')-\tilX_x^l(F'')| \leq 8d^{\beta}/dm$.
Thus $|S_k(F)-S_{k-1}(F)| \leq 8d^{\beta}/dm$.

With this, we can apply Azuma's inequality.
\begin{theorem}[{\cite[Theorem 2.25]{JLR-00-random-graphs}}]\label{thm:azuma}
	If $(S_k)_{k=0}^{N}$ is a martingale with $S_n = X$ and $S_0 = \bE X$,
	and there exists $c>0$ so that for each $0 < k \leq N$, $|S_k-S_{k-1}| \leq c$,
	then
	\[
		\bP(|X-\bE X| \geq T) \leq 2 \exp \left(- \frac{T^2}{2 N c^2}\right)
	\]
\end{theorem}

\begin{proof}[Proof of Proposition~\ref{prop:deviation-light}]
	We apply Theorem~\ref{thm:azuma} to $(S_k)$ with 
	$N=E, T = K/d^\alpha, c = 8 d^{\beta}/dm$, to get
	\[
		\bP\left( |\tilX_x^l -\bE \tilX_x^l| \geq \frac{K}{d^\alp} \right)
		\leq 2\exp \left( \frac{-K^2 d^2 m^2}{128 d^{2\alp+2\beta} E} \right)
		\leq 2 \exp \left( -\frac{1}{128} K^2 m \right),
	\]
	where we use that $E \leq dm$ and $2\alp+2\beta \leq 1$.
\end{proof}

In the case of $p \in [2,3]$, we want to bound $|X_x^l-\bE X_x^l|$.
Since each edge $e \in F_1$ contributes at most $|\Re(e)|\leq (1+p2^{p-1})\ovP(e) \leq 13d^{\beta}/dm$ to $X_x^l(F)$, we get a bound
$|X_x^l(F')-X_x^l(F'')|\leq 52d^\beta/dm$ in the analogous argument, and thus:
\begin{proposition}
  \label{prop:deviation-light-fix}
	For $p\in[2,3]$, for any $\alp \in (0,1)$, so that $2\beta+2\alp \leq 1$, 
	and any positive number $K > 0$, 
	the following inequality holds for every $x \in T_{p,\bd,R}(M(F)_0)$,
	\begin{equation*}%\label{eq-dev-bound-statement}
		\bP\left( | X_x^l - \bE(X_x^l) | \geq \frac{K}{d^{\alp}}\right) < 
		2 \exp \left( -\frac{1}{6000} K^2 m \right)
		.
	\end{equation*}
\end{proposition}

\subsection{Proof of Proposition~\ref{prop:light-bound}}\label{ssec:light-dev-proof}
Since $\tilX_x^l \leq \bE(\tilX_x^l) + |\tilX_x^l-\bE(\tilX_x^l)|$, by Lemma~\ref{lem:expectation-light}
and Proposition~\ref{prop:deviation-light} we have for fixed
$x \in T_{p,\bd,R}(M(F)_0)$,
\[
	\bP\left(F \in \bG^*(m,\bd) : \tilX_x^l(F) \geq 
		\frac{8\theta^3R^2}{d^{\beta/p}}+\frac{K}{d^\alp} \right) < 
		2 \exp \left( -\frac{1}{128} K^2 m \right).
\]
The size of $T=T_{p,\bd,R}(M(F)_0)$ is bounded by 
Proposition~\ref{prop:size-of-t}, so
\begin{multline}\label{eq:light-config-bd}
	\bP\left(F \in \bG^*(m,\bd) : \exists x \in T \text{ with } \tilX_x^l(F) \geq 
		\frac{8\theta^3R^2}{d^{\beta/p}}+\frac{K}{d^\alp} \right) \\ < 
		2 \exp \left( -\frac{1}{128} K^2 m 
		+ m \log \left(\frac{4eR}{\eps}\right)
		\right) 		
		.
\end{multline}
To optimise the bound, we set $\beta/p = \alp$ and $2 \alp + 2\beta=1$ to get $\beta/p=\alp=1/(2+2p)$, as previously described.

Suppose $H$ is an event in $\bG(m,\bd)$ with $H'$ an event in $\bG^*(m,\bd)$ so that
if $M(F)$ is simple, then $H'$ holds for $F$ if and only if $H$ holds for $M(F)$.
Then by Proposition~\ref{prop:config-model},
\begin{align*}
	\bP(G \in \bG(m,\bd) \text{ has } H) & = 
	{\bP( F \in \bG^*(m,\bd) \text{ has } H' \,|\, M(F) \text{ is simple})} \\ &
		\leq \frac{\bP( F \in \bG^*(m,\bd) \text{ has } H')}
		{\exp(-d^2+o(1))},
\end{align*}
provided $d = o((\sum d_i)^{1/4})$, which holds since 
$d/(\sum d_i)^{1/4} \leq d\theta^{1/4}/(md)^{1/4} = (\theta d^3/m)^{1/4} = o(1)$.

Applying this to \eqref{eq:light-config-bd} proves
Proposition~\ref{prop:light-bound} for $p \geq 3$ since in that case $X^l_x \leq p \tilde{X}^l_x$, $d^2 = o(m^{2/3})$, and $R \leq 4$. 

The $p \in [2,3]$ case follows a similar argument, using the bounds of Lemma~\ref{lem:expectation-light-fix} and Proposition~\ref{prop:deviation-light-fix}.
\qed

%%%%%%%%%%%%%%%%%%%%%%%%%%%%%%%%%%%%%%%%%%%%%%%%%%%%%%%%%%%%%%%%%
\section{Bounding heavy terms}\label{sec:heavy-bound}
In this section our goal is to prove, adjusting notation slightly, the following bound on the heavy terms of $X_x(G)$.
\begin{varthm}[Proposition \ref{prop:heavy-bound}.]
	For $\beta=p/(2+2p)$ and $d=o(m^{1/2})$ we have that 
for any $\xi >0$, there exists $C' = C'(\theta, \xi)$ so that
	\[
		\bP\left( \text{For all } x \in T_{p,\bd,R}(G_0), 
		\ X_x^h(G) \leq \frac{C'p}{d^{\beta/p}} + \frac{C'p^4}{d^{1/p^2}} + \frac{C'p^4\eps^{-q}}{d^{1/p}}
		\right)
		\geq 1-o(m^{-\xi}),
	\]
	and moreover this probability holds on a set in $\bG(m,\bd)$ defined independently of $p$.
\end{varthm}

We use \eqref{eq:p-tvp} to see that 
\[
	X_x^h(G) \leq 2p \ovX_x^h(G); 
\]
recall that
\[
	\ovX_x^h = \ovX_x^h(G) = \sum_{e \in E_h, \vtx (e) = \{u,v \}} \ovP(x_u, x_v),
\]
where $E_h = \{e \in G_1, \vtx (e) = \{u,v \} : \ovP(x_u, x_v) > d^\beta/dm\}$,
and $\beta = p/(2+2p)$.

We will bound $\ovX_x^h$ by showing that if we can control the number of edges between subsets of
a graph, then $\ovX_x^h$ has an explicit bound.

As previously, in what follows $\theta \geq 1$ is a fixed constant.

\begin{definition}\label{def:edge-density}
	Let $G$ be a graph with $|G_0|=m$ vertices,
	minimum degree $d_{min}$ and maximum degree $d=d_{max}$ such that $\theta \geq d/d_{min}$.
	
	Given subsets of vertices $A, B \subset G_0$, denote by $\mathcal{E}_{A,B}(G)$ 
	the number of edges in $G$ between $A$ and $B$, and set $\mu(A,B) = \theta |A||B| d/m$.
	
	We say that $G$ has \emph{$(\theta,C)$-controlled edge density}, where  $C \geq e$ is a given constant, if
	for every $A, B \subset G_0$, either
	\begin{itemize}	
	\item[(a)]  $\mathcal{E}_{A,B} \leq C \mu(A,B)$, or
	\item[(b)]  $\mathcal{E}_{A,B} \log \frac{\mathcal{E}_{A,B}}{\mu(A,B)} \leq
		C (|A| \vee |B|) \log \frac{m}{|A| \vee |B|}$.
	\end{itemize}
\end{definition}

This property is satisfied for a random graph in  $\bG(m,\bd)$, 
as can be seen, for example, in work of Broder--Frieze--Suen--Upfal.  
\begin{lemma}[{\cite[Lemma 16]{BFSU-99-rand-graphs}}]\label{lem:block-edge-bound}
	Let $G$ be a random graph in $\bG(m,\bd)$,
	where $\bd\in \N^m$ is a degree sequence with minimum degree $d_{min}$ and maximum
	degree $d=d_{max}$ such that $\theta \geq d/d_{min}$, and $d=o(m^{1/2})$.
	
	For every $\xi>0$ there exists $C=C(\theta, \xi)>e$ so that with probability at 
	least $1-o(m^{-\xi})$,
	$G$ has $(\theta,C)$-controlled edge density.
\end{lemma}
\begin{remark}
  The lemma in \cite{BFSU-99-rand-graphs} is stated for $\theta > d/d_{min}$ sufficiently large.
  A reading of the proof shows that one can take any $\theta \geq 2d/d_{min}$ and any $C \geq 100\theta+100\xi$.  However, considering Definition~\ref{def:edge-density}, this then implies that the lemma holds for $\theta \geq d/d_{min}$, at a cost of doubling $C$.
\end{remark}
\begin{proposition}\label{prop:edge-density-heavy-bound}
	If $G \in \bG(m,\bd)$ has minimum degree $d_{min}$, maximum degree $d=d_{max}$,
	and $(\theta,C)$-controlled edge density, then 
	there exists $C'=C'(\theta,C)$ so that 
	for all $x \in T_{p,\bd,R}(G_0)$,
	\[
		\ovX_x^h \leq \frac{C'}{d^{\beta/p}} + \frac{C'p^3}{d^{1/p^2}} +\frac{C'p^3\eps^{-q}}{d^{1/p}}.
	\]
\end{proposition}
Together with Lemma~\ref{lem:block-edge-bound}, this proposition immediately implies Proposition~\ref{prop:heavy-bound}.
The remainder of this section consists of the proof of
Proposition~\ref{prop:edge-density-heavy-bound}.

Given $x \in T_{p,\bd,R}(G_0) \subset \R^m$, 
the set of vertices $G_0$ splits into blocks as follows.
For $i > 0$, let
\[
	A_i = A_i(x) = \left\{ u\in G_0 : 
	2^{i-1} \frac{\eps}{(dm)^{1/q}} \leq |x_u|^{p-1} 
	< 2^{i} \frac{\eps}{(dm)^{1/q}} \right\}\mbox{ and }a_i = |A_i|\, .
\]
Those vertices with $x_u=0$ contribute nothing to $\ovX_x^h$, and so may be ignored.
Whenever $x_u \neq 0$, $|x_u|^{p-1} \geq \eps d^{1/p} / d_i m^{1/q} \geq \eps / (dm)^{1/q}$,
and so $u \in A_i$ for some $i \geq 1$.

Consider the function $\mathcal{E}_{ij} (G) = \mathcal{E}_{A_i, A_j} (G) $ 
defined as the number of unoriented edges between $A_i$ and $A_j$.
With $\theta\geq d/d_{min}$ as above, 
let $\mu_{ij} = a_i a_j \theta d/m$.

If $e \in E_h, \vtx (e) = \{u,v \}$, with $u \in A_i$ and $v \in A_j$, then
\begin{align*}
	\frac{d^\beta}{dm} & \leq \ovP(x_u,x_v)
		\leq \ovP\left( \left(2^i \frac{\eps}{(dm)^{1/q}}\right)^{1/(p-1)}, 
		\left(2^j \frac{\eps}{(dm)^{1/q}}\right)^{1/(p-1)} \right)
	\\ &= 2^{i \vee j + \frac{1}{p-1} i \wedge j} \left(\frac{\eps^q}{dm}\right) =: \ovP_{ij},
\end{align*}
where $i \wedge j$ denotes the minimum of $i$ and $j$.

So
\begin{align*}
	\ovX_x^h & = \sum_{e \in E_h, \vtx (e) = \{u,v \}} \ovP(x_u, x_v) 
		\leq \sum_{i, j >0 \,:\, \ovP_{ij} \geq d^{\beta}/dm} \mathcal{E}_{ij} \ovP_{ij}.
\end{align*}

Let $\cC = \{ (i,j) \in \N^2: \ovP_{i j} \geq d^\beta/dm\}$. 
Using the notation of Definition~\ref{def:edge-density}, we set
\[
	\cC_a (G) = \{ (i,j) \in \cC \mid (a) 
		\text{ holds for } \mathcal{E}_{A_i, A_j}(G)\}\, 
	\mbox{ and }\cC_b(G) = \cC \setminus \cC_a(G)\, .
\]
So for all pairs $(i,j) \in \cC_b (G)$, 
$\mathcal{E}_{A_i, A_j}(G)$ satisfies (b) but not (a).

We now bound $\ovX_x^h$ a.a.s.\ as follows.
\begin{align}\label{eq:heavy-split-ab}
	\ovX_x^h & \leq \sum_{(i, j)\in \cC_a (G)} \cE_{ij} \ovP_{ij}
		+ \sum_{(i, j)\in \cC_b (G)} \cE_{ij} \ovP_{ij}.
\end{align}
Let us call the first of these terms $\ovX_{ha}$, and the second $\ovX_{hb}$.
We bound each of these in turn.
\begin{lemma}\label{lem:bound-xha}
	\[
		\ovX_{ha} \leq  \frac{2^{2q} C \theta^3 R^2}{d^{\beta/p}}.
	\]
\end{lemma}
\begin{proof}
	Since $(i,j) \in \cC_a$, we have $\cE_{ij} \leq C \mu_{ij} = C a_i a_j \theta d/m$.
	So
	\begin{align*}
		\ovX_{ha} =  \sum_{(i, j)\in \cC_a} \cE_{ij} \ovP_{ij}
			\leq \frac{C\theta d}{m} \sum_{(i,j) \in \cC_a} a_i a_j \ovP_{ij}.
	\end{align*}
	Now if $u \in A_i$ and $v \in A_j$, and $(i,j) \in \cC_a$, then
	\begin{multline*}
		\ovP(x_u, x_v) \geq \\
		\ovP\left( \left(2^{i-1} \frac{\eps}{(dm)^{1/q}}\right)^{1/(p-1)},
			\left(2^{j-1} \frac{\eps}{(dm)^{1/q}}\right)^{1/(p-1)} \right)
		= 2^{-q} \ovP_{ij} \geq \frac{2^{-q} d^\beta}{dm}.
	\end{multline*}
	So
	\begin{align*}
		\ovX_{ha} & \leq \frac{C\theta d}{m} 
			\sum_{u,v \in V : \ovP(x_u,x_v) \geq 2^{-q}d^{\beta}/dm} \ovP(x_u,x_v)2^{q} \\
			& \leq \frac{2^{q}C \theta d}{m} \cdot 
			\frac{2m\theta^{2}R^2}{(2^{-q}d^{\beta})^{1/p}d}
			= \frac{2^{2q} C \theta^3 R^2}{d^{\beta/p}}, 
	\end{align*}
	where we use Lemma~\ref{lem:p-gamma-bound} with $\gamma = 2^{-q}d^\beta$.
\end{proof}

\begin{lemma}\label{lem-bound-xhb}
We have 
	\[
		\ovX_{hb} \leq 1000 \theta^3 R^2 C p^3
		\left( \frac{1}{d^{1/p^2}}+\frac{\eps^{-q}}{d^{1/p}} \right) .
	\] 
\end{lemma}

\begin{proof}
Let $\cC_b' = \{ (i,j) \in \cC_b : a_i \leq a_j \}$.  Then
\begin{align}
	\ovX_{hb} & = \sum_{(i,j) \in \cC_b} \cE_{ij} \ovP_{ij} 
		 \leq 2 \sum_{(i,j) \in \cC_b'} \cE_{ij} \ovP_{ij} % \notag \\
		 = 2 \sum_{(i,j) \in \cC_b'} \cE_{ij} 
			2^{i \vee j + \frac{1}{p-1} i \wedge j} \left(\frac{\eps^q}{dm}\right).
		\label{eq:xhb-bound-1}
\end{align}
We now split this sum into five terms (cf.~\cite{BFSU-99-rand-graphs}),
with $\cC_b' = \cD_1 \sqcup \cD_2 \sqcup \cD_3 \sqcup \cD_4 \sqcup \cD_5$,
where $\cD_l$ denotes the subset of $\cC_b'$ satisfying $(l)$, but not $(l')$ for any $l'<l$.
The parameter $\eta>0$ will be optimised later.
\begin{enumerate}
	\item $2^j d^\eta \leq 2^i$,
	\item $\cE_{ij}/\mu_{ij} \leq d^{-\eta} 2^{\frac{1}{p-1} i \vee j + i \wedge j}$,
	\item $\log(\cE_{ij}/\mu_{ij}) \geq \frac{1}{4(p-1)} \log(m/a_j)$,
	\item $(m/a_j)^{1/(4(p-1))} \leq 2^j$, if $i>j$, or $\leq 2^{j/(p-1)}$ if $i \leq j$,
	\item $(4)$ is false.
\end{enumerate}
We write \eqref{eq:xhb-bound-1} as $2 \sum_{l=1}^{5} A_l$, where
\[
	A_l = \sum_{(i,j) \in \cD_l} \cE_{ij} 2^{i \vee j + \frac{1}{p-1} i \wedge j} 
		\left(\frac{\eps^q}{dm}\right).
\]

One fact we will use repeatedly in the following is that
$\sum_{u \in V} |x_u|^p d_u \leq R$ implies that
\begin{equation}\label{eq:norm-block-bound}
	\sum_{i>0} a_i 2^{iq} \frac{\eps^q}{m} \leq 2^{q} \sum_{u \in G_0} |x_u|^p d
		\leq {2^q \theta} \sum_{u \in G_0} |x_u|^p d_u 
		\leq {4 \theta R}.
\end{equation}

\noindent\emph{Case (1):}
\begin{align*}
	A_1 & = \sum_{(i,j) \in \cD_1 : i \geq j} \cE_{ij} 2^{i + \frac{1}{p-1} j} 
			\left(\frac{\eps^q}{dm}\right)
		+ \sum_{(i,j) \in \cD_1: j > i} \cE_{ij} 2^{j + \frac{1}{p-1} i } 
			\left(\frac{\eps^q}{dm}\right) \\
		& \leq \sum_{i} \left( a_i 2^i \frac{\eps^q}{m} \sum_{j : 2^j d^\eta \leq 2^i} 
				2^{j/(p-1)} \right) +
			\sum_{i} \left( a_i 2^{\frac{1}{p-1}i} \frac{\eps^q}{m} 
				\sum_{j : 2^j d^\eta \leq 2^i} 2^{j} \right),
\end{align*}
because $\cE_{ij} \leq a_i d$.		
Since $\sum_{j : 2^j d^\eta \leq 2^i} 2^{j/(p-1)}$ is a geometric series with largest 
term $\leq 2^{i/(p-1)}/d^{\eta/(p-1)}$,
this sum is bounded by $C_1 2^{i/(p-1)}/d^{\eta/(p-1)}$,
with in fact $C_1 = 1/(1-2^{-1/(p-1)}) =(1+o(1)) (p-1)/\log(2) \leq 2p$.
Likewise, $\sum_{j : 2^j d^\eta \leq 2^i} 2^{j} \leq 2 \cdot 2^i/d^\eta$.
So as $1+ \frac{1}{p-1} = q$, and $1/d^{\eta} \leq 1/d^{\eta/(p-1)}$,
\eqref{eq:norm-block-bound} gives
\begin{align*}
	A_1 & \leq 4p \left( \sum_{i} a_i 2^{iq} \frac{\eps^q}{m} \right) \frac{1}{d^{\eta/(p-1)}}
		\leq \frac{16 \theta R p}{d^{\eta/(p-1)}}.
\end{align*}

\noindent\emph{Case (2):}
Applying (2), we see that
\begin{align*}
	A_2 & = \sum_{(i,j) \in \cD_2} \cE_{ij} 2^{i\vee j + \frac{1}{p-1} i \wedge j}
			\frac{\eps^q}{dm} \\
		& \leq \sum_{i,j} \frac{1}{d^{\eta}} \mu_{ij} 2^{iq+jq} \frac{\eps^q}{dm} \\
		& = \frac{\theta}{d^\eta \eps^q} \sum_{i,j}\left(a_i 2^{iq} \frac{\eps^q}{m}\right)
			\left(a_j 2^{jq} \frac{\eps^q}{m}\right), 
			\text{ as } \mu_{ij} = \theta a_i a_j d/m,\\
		& \leq \frac{16 \theta^3 R^2}{d^\eta \eps^q}, \text{ by } \eqref{eq:norm-block-bound}.
\end{align*}

\noindent\emph{Case (3):}
By (3) and Definition~\ref{def:edge-density}(b), we have
\[
	\cE_{ij} \frac{1}{4(p-1)} \log \left( \frac{m}{a_j} \right)
	\leq \cE_{ij} \log \left(\frac{\cE_{ij}}{\mu_{ij}} \right)
	\leq C a_j \log \left( \frac{m}{a_j} \right),
\]
and so $\cE_{ij} \leq 4C(p-1) a_j \leq 4Cp a_j$.  
Also, as (1) is false, $2^i < 2^j d^\eta$.  Thus,
\begin{align*}
	A_3 & = \sum_{(i,j) \in \cD_3} \cE_{ij} 2^{i\vee j + \frac{1}{p-1} i \wedge j} 
				\frac{\eps^q}{dm} \\
		& \leq \frac{4Cp}{d} \sum_{(i,j) \in \cD_3: i \geq j} a_j 2^{i+\frac{1}{p-1}j} 
				\frac{\eps^q}{m} +
			\frac{4Cp}{d} \sum_{(i,j) \in \cD_3: i < j} a_j 2^{j+\frac{1}{p-1}i} 
				\frac{\eps^q}{m} \\
		& \leq \frac{4Cp}{d}
			\sum_j \left( a_j 2^{\frac{1}{p-1}j} \sum_{i : 2^i < 2^j d^\eta} 2^i \right) 
				\frac{\eps^q}{m} + \frac{4Cp}{d}
			\sum_j \left( a_j 2^{j} \sum_{i : 2^i < 2^j d^\eta} 2^{\frac{1}{p-1}i} \right) 
				\frac{\eps^q}{m}.
\end{align*}
As in Case (1), $\sum_{i : 2^i < 2^j d^\eta} 2^i \leq 2 \cdot 2^j d^\eta$ and
$\sum_{i : 2^i < 2^j d^\eta} 2^{i/(p-1)} \leq 2p 2^{j/(p-1)} d^{\eta/(p-1)} 
\leq 2p 2^{j/(p-1)} d^\eta$, so \eqref{eq:norm-block-bound} gives
\begin{align*}
	A_3 & \leq \frac{16Cp^2 \cdot d^\eta}{d} \sum_j a_j 2^{jq} \frac{\eps^q}{m} 
	\leq 64\theta R Cp^2 \frac{1}{d^{1-\eta}}.
\end{align*}

\noindent\emph{Case (4):}
First note that, as we are not in Case (2) or Case (3), (4) gives
\[
	d^{-\eta} 2^{\frac{1}{p-1} i \vee j + i \wedge j}
		< \frac{\cE_{ij}}{\mu_{ij}} < \left( \frac{m}{a_j} \right)^{1/(4(p-1))}
		\leq \begin{cases}
		     	2^j	& 	\text{if $i>j$},\\
		     	2^{\frac{1}{p-1} j} &	\text{if $i \leq j$}.
		     \end{cases}
\]
Therefore, $2^{\frac{1}{p-1}i} < d^\eta$ if $i>j$, and $2^i < d^\eta$ if $i \leq j$.

Second, note that as we are not in $\cC_a$, we have $\cE_{ij}/\mu_{ij} > C > e$, so
\[
	\cE_{ij} \leq \cE_{ij} \log \left(\frac{\cE_{ij}}{\mu_{ij}} \right) 
	\leq C a_j \log \left( \frac{m}{a_j} \right) 
	\leq C4(p-1) a_j j \log(2)\leq 4Cpa_j j,
\]
where the second and third inequalities follow from 
Definition~\ref{def:edge-density}(b) and (4) respectively.

Using these estimates, we see that
\begin{align*}
	A_4 & = \sum_{(i,j) \in \cD_4: i > j} \cE_{ij} 2^{i+\frac{1}{p-1}j} \frac{\eps^q}{dm} +
			\sum_{(i,j) \in \cD_4: i \leq j} \cE_{ij} 2^{j+\frac{1}{p-1}i} \frac{\eps^q}{dm} \\
		& \leq \frac{4Cp}{d} \sum_j \left( a_{j} j 2^{\frac{1}{p-1}j} 
			\frac{\eps^q}{m} \sum_{i: 2^i < d^{(p-1)\eta}} 2^i \right) +
			\frac{4Cp}{d} \sum_{j} \left( a_j j 2^{j} \frac{\eps^q}{m}
			\sum_{i : 2^i < d^\eta} 2^{\frac{1}{p-1}i}   \right). \\
\end{align*}
Again, a geometric series argument shows that the two sums over $i$ are bounded 
by $2d^{(p-1)\eta}$ and $2p d^{\eta/(p-1)} \leq 2p d^{(p-1)\eta}$
respectively.  So
\begin{align*}
	A_4 & \leq \frac{16Cp^2 d^{(p-1)\eta}}{d} \sum_{j} a_j j 
		\left( 2^{\frac{1}{p-1} j} + 2^j \right) \frac{\eps^q}{m}.
\end{align*}
Now, $j 2^{\frac{1}{p-1} j} = j 2^{qj-j} \leq 2^{qj}$, as $j2^{-j} \leq 1$.
On the other hand, $j 2^j = 2^{qj} (j 2^{-(q-1)j}) \leq (p-1)2^{qj}$, as
easily follows from maximizing $j 2^{-(q-1)j} \leq ((q-1)\log(2)e)^{-1} \leq p-1$.
Thus
\begin{align*}
	A_4 & \leq \frac{16Cp^3}{d^{1-(p-1)\eta}} \sum_{j} a_j 2^{qj} \frac{\eps^q}{m} 
		\leq \frac{64R\theta Cp^3}{d^{1-(p-1)\eta}}.
\end{align*}

\noindent\emph{Case (5):}
Whether $i > j$ or not, (5) gives $a_j < 2^{-4j}m$.  
Thus as we are in $\cC_b$ we have that
\[ \cE_{ij} \leq \cE_{ij} \log \frac{\cE_{ij}}{\mu_{ij}} 
\leq C a_j \log \frac{m}{a_j} \leq C 2^{-4j}mj \cdot 4\log(2) \leq 3C 2^{-4j}mj. \]
Here we used that $x \log(m/x)$ is an increasing function
of $x$ on $[0, m e^{-1}]$, and that $2^{-4j} \leq e^{-1}$.
This gives
\begin{align*}
	A_5 & =  \sum_{(i,j) \in \cD_5: i > j} \cE_{ij} 2^{i+\frac{1}{p-1}j} \frac{\eps^q}{dm} +
			 \sum_{(i,j) \in \cD_5: i \leq j} \cE_{ij} 2^{j+\frac{1}{p-1}i} \frac{\eps^q}{dm}, 
			 \text{ so as (1) fails,} \\
		& \leq \frac{3C \eps^q}{d} \sum_{j} \left( j 2^{-4j} 2^{\frac{1}{p-1}j} 
				\sum_{i: 2^i < 2^j d^\eta} 2^i \right) +
			\frac{3C \eps^q}{d} \sum_{j} \left( j 2^{-4j} 2^{j} 
				\sum_{i: 2^i < 2^j d^\eta} 2^{\frac{1}{p-1}i} \right).\\
		\intertext{Summing the two geometric series in $i$ and then using $q \leq 2$ gives us} \\ A_5
		& \leq \frac{3C \eps^q}{d} \sum_j j 2^{-4j} \left( 2^{\frac{1}{p-1}j} \cdot 
				2 \cdot 2^{j}d^\eta
			+ 2^j \cdot 2p \cdot 2^{\frac{1}{p-1}j} d^{\eta/(p-1)} \right) \\
		& \leq \frac{3C \eps^q}{d^{1-\eta}} \sum_j j 2^{-4j} 
			\left( 2 \cdot 2^{qj} + 2p \cdot 2^{qj}  \right) \\
		& \leq \frac{12pC\eps^q}{d^{1-\eta}} \sum_j j 2^{-2j} 
		\leq \frac{12pC\eps^q}{d^{1-\eta}}.
\end{align*}

We now combine these estimates to see that
\begin{align*}
	\ovX_{hb} & \leq 2 \left(
		\frac{16 \theta R p}{d^{\eta/(p-1)}}
		+ \frac{16 \theta^3 R^2}{d^\eta \eps^q}
		+ \frac{64\theta R Cp^2}{d^{1-\eta}}
		+ \frac{64R\theta Cp^3}{d^{1-(p-1)\eta}}
		+ \frac{12pC\eps^q}{d^{1-\eta}} \right)
		\\ & \leq 128 \left(
		\frac{\theta R p}{d^{\eta/(p-1)}}
		+ \frac{\theta^3 R^2}{d^\eta \eps^q}
		+ \frac{\theta R Cp^2}{d^{1-\eta}}
		+ \frac{R\theta Cp^3}{d^{1-(p-1)\eta}}
		+ \frac{pC\eps^q}{d^{1-\eta}} \right)
		\\ & \leq 128 \theta^3 R^2 C p^3\left(
		\frac{1}{d^{\eta/(p-1)}}
		+ \frac{1}{d^\eta \eps^q}
		+ \frac{1}{d^{1-\eta}}
		+ \frac{1}{d^{1-(p-1)\eta}}
		+ \frac{\eps^q}{d^{1-\eta}} \right).
\end{align*}
Assuming $\eps^q \leq 1$, we have
\begin{align*}
	\ovX_{hb} & \leq 128 \theta^3 R^2 C p^3 \left(
		\frac{1}{d^{\eta/(p-1)}}
		+ \frac{1}{d^\eta \eps^q}
		+ \frac{3}{d^{1-(p-1)\eta}} \right).
\end{align*}
We set $\eta/(p-1) = 1-(p-1)\eta$ and thus take $\eta = (p-1)/((p-1)^2+1)$,
giving
\begin{align*}
	\ovX_{hb} & \leq 1000 \theta^3 R^2 C p^3 \left(
	\frac{1}{d^{1/((p-1)^2+1)}} + \frac{1}{d^{(p-1)/((p-1)^2+1)}\eps^q} \right).
\end{align*}
Since $(p-1)^2+1 \leq p^2$ and, for $p\geq 2$, $(p-1)/((p-1)^2+1) \geq 1/p$,
the proof of Lemma~\ref{lem-bound-xhb} is complete.
\end{proof}

Lemmas~\ref{lem:bound-xha} and \ref{lem-bound-xhb}
combine with \eqref{eq:heavy-split-ab} to complete the proof of
Proposition~\ref{prop:edge-density-heavy-bound}.
We have now completed the proof of Theorems~\ref{thm:G-m-rho-eval-bound}
and \ref{thm:G-m-deg-eval-bound}.

\section{Application to fixed point properties}\label{sec:fixed-pt-prop}

Our main interest in estimates of the eigenvalues of the $p$-Laplacian resides in the following application. The result is proved using a slight modification of arguments of Bourdon \cite{Bourdon:FLp}. 

\begin{varthm}[Theorem \ref{thm:bourdon-flp}.]
	Consider $p \in (1,\infty)$ and $\varepsilon < \frac12$. Let $X$ be a simplicial $2$-complex where the link $L(x)$ of every vertex $x$ has $\lambda_{1,p}(L(x)) > 1-\varepsilon$, and has at most $m$ vertices.
	If some group $\Gamma$ acts on $X$ simplicially, properly, and cocompactly, 
	then $\Gamma$ has the fixed point property $FL^p_{m+1,(2-2\varepsilon)^{1/2p}}$.
\end{varthm}

\begin{proof} We denote by $X_0$ the set of vertices of $X$ and by $X_1$ the set of edges of $X$. For every edge $e\in X_1$ we denote by $\Gamma_e$ its stabiliser in $\Gamma$. 

Let $\Xi_k$ be a system of representatives of $\Gamma \backslash X_k$, for $k\in \{ 0,1\}$.  

\medskip

Assume that $\Gamma $ acts by affine isometries on a Banach space $V$ with $L$-bi-Lipschitz $L^p$ geometry above dimension $m+1$, where $L=(2-2\varepsilon)^{1/2p}$.
Observe that the constant $L$ satisfies the condition $L >1$ due to the fact that $\varepsilon < \frac12$. Following the terminology and argument from \cite{Bourdon:FLp}, we denote by ${\mathcal{E}}$ the set of $\Gamma$--equi\-var\-i\-ant functions $\varphi : X_0 \to V$.

Given a function $\varphi \in {\mathcal{E}}$, we define its \emph{energy} as 
$$
E(\varphi ) = \sum_{e\in \Xi_1} \|\varphi (e_+) - \varphi (e_-) \|_V^p \frac{{\mathrm{val}}\, (e)}{\left|\Gamma_e \right|}.
$$

We say that a function in ${\mathcal{E}}$ is $p$--\emph{harmonic} if it minimizes the energy.

If $\inf_{\varphi \in {\mathcal{E}}} E(\varphi )=0$ and there exists a $p$--harmonic function then $\Gamma$ has a fixed point and the argument is finished. 

If $\inf_{\varphi \in \mathcal{E}} E(\varphi) = 0$ and there is no $p$--harmonic function then Proposition 3.1(ii) in \cite{Bourdon:FLp} implies that, up to replacing $V$ with a rescaled ultralimit of itself, one may assume that $\inf_{\varphi \in \mathcal{E}} E(\varphi) > 0$.
By again potentially replacing $V$ by a rescaled ultralimit of itself, Proposition 3.1(i) in \cite{Bourdon:FLp} lets us assume that there always exists a $p$--harmonic function $\varphi $ such that $E(\varphi )> 0$. 

In the two arguments above, the key fact is that by replacing the Banach space $V$ with a rescaled ultralimit of itself one does not lose any of the properties of the initial space $V$. Indeed, the new space $V_\omega= \omega\text{-lim}\, W_i$, where $W_i$ are rescalings of $V$, continues to have $L$-bi-Lipschitz $L^p$ geometry above dimension $m+1$: if $U_\omega \leq V_{\omega}$ is an affine subspace of dimension $m+1$ then $U_\omega$ is the ultralimit $U_\omega = \omega\text{-lim}\, U_i$ of subspaces $U_i \leq W_i$ of dimension $m+1$. By assumption, each $U_i$ is contained in a subspace $U'_i$ so that there is an $L$-bi-Lipschitz equivalence $U_i' \ra Y_i$ to some space $Y_i$ equal either to an $\ell^p_{n_i}$ for some ${n_i}\geq m+1$, or to $\ell^p_\infty$, or to some space $L^p(M_i,\mu_i )$.

 Taking an ultralimit of these maps gives an $L$-bi-Lipschitz equivalence of $U'_\omega=\omega\text{-lim}\, U_i'$ to the ultralimit $\omega\text{-lim}\, Y_i$. The latter is either an $\ell^p_{n}$ for some $n\geq m+1$, or an $\ell^p_\infty$ space, or an $L^p$ space, because every rescaled ultralimit of $L^p$ spaces is also an $L^p$ space. This follows from work of Kakutani \cite{Kakutani:Lp}, see 
 \cite[Corollary 19.18]{DrutuKapovich} for details.

Thus, when no $p$-harmonic function of energy zero exists, without loss of generality we may assume that there exists a $p$--harmonic function $\varphi $ such that $E(\varphi )> 0$. 
An arbitrary vertex $x$ has by hypothesis at most $m$ neighbours. In particular $\varphi (x)$ and $\varphi (y)$ for $y\sim x$ span a subspace of dimension at most $m+1$ hence, for $L=(2-2\varepsilon)^{1/2p}$, there exists an $L$-bi-Lipschitz map $F_x$ from a subspace $U_x$ containing $\varphi (x)$ and $\varphi (y)$ for $y\sim x$ to a space $W_x$ equal to an $\ell^p_n$ for some $n\geq m+1$, or to $\ell^p_\infty$, or to some space $L^p(Y,\mu )$.  

We now follow the calculation in \cite[Page 388]{Bourdon:FLp}.
Proposition 2.4 and Lemma 4.1 in \cite{Bourdon:FLp} hold in full generality, in particular for the action of $\Gamma $ on $V$.
For each fixed $x \in X_0$, Proposition 2.4 of \cite{Bourdon:FLp} applied to the $p$--harmonic function $\varphi$ combined with the bi-Lipschitz condition on $F_x$ gives
\begin{align}%\label{eq:cor14}
\sum_{e=\{ x,y\} } \|\varphi (y) - \varphi (x) \|_V^p {\mathrm{val}}\, (e) 
& = \inf_{v \in V} \sum_{e=\{ x,y\} } \|\varphi (y) - v \|_V^p {\mathrm{val}}\, (e) \nonumber
\\ & \leq \inf_{v \in U_x} \sum_{e=\{ x,y\} } \|\varphi (y) - v \|_V^p {\mathrm{val}}\, (e)  \nonumber
\\ & \leq L^p \inf_{w \in W_x} \sum_{e=\{ x,y\} } \|F_x \circ \varphi (y) - w \|_{W_x}^p {\mathrm{val}}\, (e). \nonumber
\\ \intertext{Continuing with Corollary 1.4 of \cite{Bourdon:FLp} applied to $F_x \circ \varphi |_{L(x)_0 \cup \{ x \}}$, the above is bounded by}
& \frac{L^{p}}{\lambda_{1,p}(L(x))} \sum_{e\in L(x)_1}\|F_x\circ\varphi (e_+) - F_x\circ\varphi (e_-) \|_{W_x}^p \nonumber
\\ & \leq \frac{L^{2p}}{\lambda_{1,p}(L(x))} \sum_{e\in L(x)_1}\|\varphi (e_+) - \varphi (e_-) \|_V^p\, .\label{eq:cor14}
\end{align}
According to Lemma 4.1 in \cite{Bourdon:FLp} we may write
$$
E(\varphi ) = \frac{1}{2}\sum_{x\in \Xi_0}\frac{1}{|\Gamma_x |}\sum_{y\in X_0,y\sim x}\|\varphi (y) - \varphi (x)\|_V^p {\mathrm{val}} (e_{xy}),
$$
 where $e_{xy}$ denotes the edge of endpoints $x,y$. This and equation (\ref{eq:cor14}) imply that
 $$
 E(\varphi ) \leq \frac{L^{2p}}{2\lambda_{1,p}} \sum_{x\in \Xi_0}\frac{1}{|\Gamma_x |}\sum_{e\in L(x)_1}\|\varphi (e_+) - \varphi (e_-) \|_V^p \, . 
$$  
Since $\lambda_{1,p}(L(x)) > \frac{1}{2}L^{2p}$ we have thus obtained that
$$
E(\varphi ) \leq \frac{L^{2p}}{2\lambda_{1,p}} E(\varphi ) \leq E(\varphi ) 
$$ with the latter a strict inequality for $E(\varphi )> 0$, which gives a contradiction. 

The assumption that $\lambda_{1,p}(L(x)) > 1-\varepsilon$ with $\varepsilon < \frac12$ has played an  essential part in the argument, in that it allowed us to find a bi-Lipschitz constant $L\geq 1$ such that $\lambda_{1,p}(L(x)) > \frac{1}{2}L^{2p}$. 
\end{proof}

\begin{corollary}\label{cor:bourdon-flp}
	Let $p \in (1,\infty)$ and $\varepsilon < \frac{1}{2}$. Suppose $X$ is a simplicial $2$-complex where the 
	link $L(x)$ of every vertex $x$ has $\lambda_{1,p}(L(x)) > 1-\varepsilon$.
	If a group $\Gamma$ acts on $X$ simplicially, properly, and cocompactly, 
	then $\Gamma$ has the property that every affine action on a space $L^p (X, \mu )$, with $(X, \mu )$ a measure space, action that is $(2-2\varepsilon)^{1/2p}$--Lipschitz, i.e. 
	$$
	\gamma \cdot v = \pi_\gamma v + b_\gamma
	$$ with $\| \pi_\gamma \|\leq (2-2\varepsilon)^{1/2p}$, has a fixed point.
\end{corollary}

\proof Given an action on a space $L^p (X, \mu )$ as described, a new norm can be defined on $L^p (X, \mu )$, equivalent to the initial one, by the formula
\begin{equation}
\label{equation: invariant norm for ub reps}
\Vert v\Vert_\pi=\sup_{\gamma\in \Gamma } \Vert \pi_\gamma v\Vert.
\end{equation}

With respect to this new norm the action of $\Gamma$ on $L^p (X)$ is isometric, and one can apply Theorem \ref{thm:bourdon-flp}.
\endproof

\section{Fixed point properties in the triangular binomial model}\label{sec:fp-rand-tri-group}

In this section we prove Theorem~\ref{thm:Itri-eventually-flp}, which finds fixed point properties with respect to actions on $L^p$ spaces, for random groups in the triangular binomial model.

Every finitely presented group has a finite triangular presentation, i.e.\ a presentation with all relators of length three.
If $\Gamma = \langle S | R \rangle$ is a triangular finite presentation of a group, 
then $\Gamma$ acts on a simplicial $2$-complex $X$ which is the Cayley complex. The link of a vertex in $X$ is the graph $L(S)$ with vertex set $S \cup S^{-1}$ and,
for each relator of the form $s_x s_y s_z$ in $R$, edges 
$(s_x^{-1}, s_y), (s_y^{-1}, s_z)$, and $(s_z^{-1}, s_x)$.
Thus, the edges of $L(S)$ decompose into three classes, corresponding to the order of appearance in the relators, and we decompose
$L(S)$ into three subgraphs $L^1(S), L^2(S), L^3(S)$, which each have the 
same vertex set as $L(S)$, but only edges of the corresponding type.

Recall that by Bourdon's Theorem~\ref{thm:bourdon-flp}, if $\lambda_{1,p}(L(S)) > 1-\varepsilon$ then $\Gamma$ has $FL^p_{m+1,(2-2\varepsilon)^{1/2p}}$. First we observe that it suffices to get eigenvalue bounds on $\lambda_{1,p}$
for each of the three graphs $L^i(S)$.
\begin{lemma}\label{lem:lp-bound-split-into-three}
	Suppose a graph $L$ can be written as $L = L^1 \cup L^2 \cup L^3$
	with each graph having the same vertex set $L_0$, but $L^1, L^2$ and $L^3$
	having pairwise disjoint edges.
	Suppose each $L^i$ has vertex degrees in $[(1-\iota)d, (1+\iota)d]$ for some positive number $d$ and $\iota \in (0,1)$.
	Then 
	\[ \lambda_{1,p}(L) \geq \frac{1-\iota}{1+\iota} \cdot 
	\frac{1}{3} \left( \lambda_{1,p}(L^1) +\lambda_{1,p}(L^2)+ \lambda_{1,p}(L^3) \right).\]
\end{lemma}
\begin{proof}
	Let $\mathbf{C}$ be the subspace of constant functions in $\R^{L_0}$.
	By \eqref{eq:lp-1}, we have:
	\begin{align*}
		\lambda_{1,p}(L) 
		& = \inf_{x \in \R^{L_0}\setminus \mathbf{C}}
			\frac{ \sum_{e \in L_1} | dx(e) |^p }
			{ \inf_{c\in\R} \sum_{u \in L_0} |x_u-c|^p \val_L(u)}
		\\ & \geq \inf_{x \in \R^{L_0}\setminus \mathbf{C}}
			\frac{ \sum_{e \in L_1} | dx(e) |^p }
			{ 3(1+\iota)\inf_{c\in\R} \sum_{u \in L_0} |x_u-c|^p d}
		\\ & = \inf_{x \in \R^{L_0}\setminus \mathbf{C}}
			\frac{ \sum_{e \in L_1^1} | dx(e) |^p + \sum_{e \in L_1^2} | dx(e) |^p
				+ \sum_{e \in L_1^3} | dx(e) |^p}
			{ 3(1+\iota)\inf_{c\in\R} \sum_{u \in L_0} |x_u-c|^p d},
			\\ \intertext{so by letting these three terms be infimised independently, we have}
		\\ \lambda_{1,p}(L) & \geq \sum_{i=1}^{3} \inf_{x \in \R^{L_0}\setminus \mathbf{C}} 
			\frac{ \sum_{e \in L_1^i} | dx(e) |^p }
			{ 3(1+\iota)\inf_{c\in\R} \sum_{u \in L_0} |x_u-c|^p d}
		\\ & \geq \frac{1-\iota}{3(1+\iota)} 
			\sum_{i=1}^{3} \inf_{x \in \R^{L_0}\setminus \mathbf{C}} 
			\frac{ \sum_{e \in L_1^i} | dx(e) |^p }
			{ \inf_{c\in\R} \sum_{u \in L_0} |x_u-c|^p \val_{L^i}(u)}
		\\ & = \frac{1-\iota}{3(1+\iota)} \sum_{i=1}^{3} \lambda_{1,p}(L^i).\qedhere
	\end{align*}
\end{proof}

We now show that adding a small number of edges to a graph cannot lower $\lambda_{1,p}$
significantly.
\begin{lemma}\label{lem:lp-add-few-edges}
	Let $G$ and $H$ be graphs with the same vertex set $G_0$, and let $G\cup H$ denote the graph with vertex set $G_0$ and edge set $G_1 \cup H_1$. 
	
If there exists $\iota >0$ so that for all $u \in G_0$, 
	$\val_H(u) \leq \iota \val_G(u)$ then 
	$$\lambda_{1,p}(G \cup H) \geq (1+\iota)^{-1} \lambda_{1,p}(G).$$
\end{lemma}
\begin{proof}
	By \eqref{eq:lp-1}, we have:
	\begin{align*}
		\lambda_{1,p}(G \cup H) 
		& = \inf_{x \in \R^{G_0}\setminus \mathbf{C}}
			\frac{ \sum_{e \in G_1\cup H_1} | dx(e) |^p }
			{ \inf_{c\in\R} \sum_{u \in G_0} |x_u-c|^p \val_{G\cup H}(u)}
		\\ & \geq \inf_{x \in \R^{G_0}\setminus \mathbf{C}}
			\frac{ \sum_{e \in G_1} | dx(e) |^p }
			{ (1+\iota) \inf_{c\in\R} \sum_{u \in G_0} |x_u-c|^p \val_{G}(u)}
		\\ & = \frac{1}{1+\iota} \lambda_{1,p}(G). \qedhere
	\end{align*}	
\end{proof}

We now follow \cite[Proof of Theorem 16]{ALS-15-random-triangular-at-third}
to describe the structure of link graphs for Cayley complexes of random groups in the model $\Gamma(m,\rho)$ in terms of
random graphs in a model $\bG(2m,\rho')$.
\begin{proposition}\label{prop:link-graph-structure}
	Suppose $\rho \leq m^{\delta }/m^2$, for some $\del < \frac{1}{4}$,
	and let $\rho' = 1-(1-\rho)^{4m-4}$.
	Given the link graph $L(S)=L^1(S) \cup L^2(S) \cup L^3(S)$ of a random group
	in $\Gamma(m,\rho)$, with probability $1-O(m^{-1+4\del})$ the graph $L^1(S)$ is the union of a graph in $\bG(2m,\rho')$ and a matching.
\end{proposition}
\begin{proof}
	This follows from \cite[Page 176]{ALS-15-random-triangular-at-third}.
	Indeed, for vertices $u \neq v$ in $L^1(S)$, with $u \neq v^{-1}$,
	there are $4m-4$ possible relations which could give an edge between $u$ 
	and $v$ in $L^1(S)$, while if $u=v^{-1}$ there are $4m-2$ possible relations that can give an edge between $u$ and $v$.  For each $u,u^{-1}$ pair, remove two of these possible relations from consideration: provided $2m \rho = o(1)$ a.a.s.\ none of these relations arise.

	For the remaining relations,  
	the probability that there is (at least) one edge
	between vertices $u \neq v$ is $\rho' = 1- (1-\rho)^{4m-4}$.
	Provided $(2m)^2(4m)^3\rho^3 = o(1)$ there are no triple edges,
	and provided $(2m)(2m)^2(2m)^2\rho^4=o(1)$ no double edges share an endpoint,
	and so one matching deals with possible multiple edges.

	Thus it suffices that $\rho = o(m^{-7/4})$, e.g.\ $\rho = m^{\del }/m^2$
	for some $\del <1/4$.
\end{proof}

We can now prove Theorem~\ref{thm:Itri-eventually-flp}, in fact we will show the
following stronger result.
\begin{theorem} \label{thm:tri-eventually-flp2}
For any $\varepsilon>0$ there exists $C>0$ so that for any function $f:\N \to (0, \infty )$, $C\log m \leq f(m) \leq m$, for $\rho (m) = {f(m)}/{m^2}$, a.a.s.\ a random triangular group in the model $\Gamma(m,\rho)$ has the property $FL^p_{(2-2\varepsilon)^{1/2p}}$ for every $p \in \left[2, \frac{1}{C} (\log f(m)/\log\log f(m))^{1/2}\right]$. In particular, for a random triangular group, every affine action on an $L^p$ space that is $(2-2\varepsilon)^{1/2p}$--Lipschitz has a fixed point. 

Moreover, if $f(m)/\log m \ra \infty$ as $m\ra\infty$, we can choose $C$ independent of $\epsilon$.
\end{theorem}

\begin{remark}
Observe that in the case of density $d>1/3$ we have $f(m) = m^\delta$ for some $\delta>0$, and that we get $FL^p$ in a range $[2, \frac{1}{C}(\log m / \log\log m)^{1/2}]$.
In the borderline case of $f(m) = C\log(m)$, we get $FL^p$ in the smaller, but still growing, range of $[2, \frac{1}{C} (\log\log m/ \log\log\log m)^{1/2}]$.
\end{remark}

\begin{remark}
As the random triangular groups are hyperbolic, this theorem is to be compared with the conjecture of Y. Shalom, stating that every Gromov hyperbolic group has an affine uniformly Lipschitz action on a Hilbert space that is proper \cite{shalom-abstract}.
\end{remark}

\begin{proof} %[Proof of Theorem~\ref{thm:tri-eventually-flp}]
	First we can assume $\rho \leq m^{\del }/m^2$, hence $f(m) \leq m^{\del }$, for some $\del < \frac{1}{4}$. Since $FL^p$ is preserved by taking quotients, this case suffices.

The Mean Value Theorem implies that
	\[
		\rho' = 1-(1-\rho)^{4m-4} \leq \rho(4m-4) \leq 4m^\del /m,
	\]
	and that for $m$ large enough 
	\[
		\rho' = 1-(1-\rho)^{4m-4} \geq \rho(4m-4)(1-\rho)^{4m-5}  \geq \tfrac12 \rho(4m-4) \geq f(m)/m.
	\]
	
	For $\Gamma \in \Gamma(m,\rho)$ by Proposition~\ref{prop:link-graph-structure}
	with probability $1-O(m^{-1 + 4\del }/m^2)$, $L^1(S)$ is the union of 
	a graph $G^1 \in \bG(2m,\rho')$ with a matching.
	Theorem~\ref{thm:G-m-rho-eval-bound} gives that for $C$ large enough there exists $C'$ so that a.a.s.\  
	$\lambda_{1,p}(G^1) \geq 1-{C'p^4}/{(\rho' m)^{1/2p^2}}-C'\bI_{p'<3}(\log m)^{1/2}/(\rho'm)^{1/2}$.
	Now
	\[
		\frac{C'p^4}{(\rho' m)^{1/2p^2}}
		\leq C' \exp \left( 4 \log(p) - \frac{1 }{2p^2}\log f(m)
			\right),
	\]
	so provided $p < \kappa (\log f(m) / \log\log f(m))^{1/2}$ for a suitable small $\kappa>0$, this bound goes to zero as $m \ra \infty$, and is certainly $\leq \varepsilon/8$ for any given $\varepsilon>0$.
	On the other hand, we have $C'(\log m)^{1/2}/(\rho'm)^{1/2} \leq C' (\log m / f(m))^{1/2} \leq C'/C^{1/2}$ which is $\leq \varepsilon/8$ for $C=C(\varepsilon)$ large enough; if $\log m / f(m) \ra 0$ then $C$ does not need to depend on $\varepsilon$.
	
	So we conclude that a.a.s.\ $\lambda_{1,p}(G^1) \geq 1-\varepsilon/4$.
	Since the matching gives a graph $H$ on the same vertex set of degree 
	$\leq 1$ while the degrees in $G^1$ are $(1+o(1))\rho'm \ra \infty$,
	Lemma~\ref{lem:lp-add-few-edges} gives that a.a.s.\ 
	$\lambda_{1,p}(L^1(S)) \geq 1-\varepsilon/3$.
	
	Now a union bound gives that a.a.s.\ $\lambda_{1,p}(L^i(S)) \geq 1-\varepsilon/3$ for $i=1,2,3$
	simultaneously, and so Lemma~\ref{lem:lp-bound-split-into-three} gives that
	a.a.s.\ $\lambda_{1,p}(L(S)) \geq 1-\varepsilon/2 > 1-\varepsilon$.
	Bourdon's Theorem~\ref{thm:bourdon-flp} then shows that $\Gamma$ has $FL^p_{(2-2\varepsilon)^{1/2p}}$ for every $\varepsilon >0$.
\end{proof}

\section{Monotonicity and conformal dimension}\label{sec:mono-confdim}
In this section we discuss two consequences of Theorem~\ref{thm:tri-eventually-flp2}:
First, we use monotonicity to show a corresponding statement in the triangular density model.  Second, we show conformal dimension bounds for random groups in both these models, which in turn shed light on the quasi-isometry types of such groups.

\subsection{Monotonicity}
We begin by comparing the triangular binomial/density models and the Gromov binomial/density models using standard monotonicity results for random structures, following \cite[Section 1.4]{JLR-00-random-graphs}.

A property of a group presentation is \emph{increasing} if it is preserved by adding relations, and it is \emph{decreasing} if it is preserved by deleting relations; it is \emph{monotone} if it is either increasing or decreasing.  For example, property $FL^p$ and being finite are both monotone (increasing) properties, and being infinite is a monotone (decreasing) property.

Let $\cM(m,f(m))$ be the triangular density model where we choose $f(m)$ cyclically reduced relators of length three when we have $m$ generators; the case of $f(m)=(2m-1)^{3d}$, $d \in (0,1)$, is the usual triangular density model.
\begin{proposition}\label{prop:triangle-models-monotone}
	Let $P$ be a monotone property of group presentations.
	Let a sequence $f(m)$ be given.
	Suppose for every sequence $\rho(m)$ with $\rho = f(m)(2m)^{-3}+O(\sqrt{f(m)}(2m)^{-3})$ we have that $P$ holds a.a.s.\ in $\Gamma(m,\rho)$.
	Then $P$ holds a.a.s.\ in $\cM(m,f(m))$.

	In particular, if for all $d>d_0$ a random group in $\Gamma(m,\rho), \rho=m^d/m^3,$ has $P$ a.a.s.\ then for all $d>d_0$ a random group in $\cM(m,d)$ has $P$ a.a.s.
\end{proposition}

Let $\cG (k,l,f)$ be the Gromov density model as described in Definition \ref{def:gromov-model}, where $f:\N \to \N$ is a sequence of integers.

\begin{proposition}\label{prop:gromov-models-monotone}
	Let $P$ be a monotone property of group presentations.
	Let $f:\N \to \N$ be a sequence of integers.
	Suppose that for every sequence $\rho(l)$ with $\rho = f(l)(2k-1)^{-l} +O(\sqrt{f(l)}(2k-1)^{-l})$ we have that $P$ holds a.a.s.\ in $\cB (k,l,\rho)$.
	Then $P$ holds a.a.s.\ in $\cG (k,l,f )$.

	In particular, if for all $d > d_0$ a random group in $\cB (k,l,\rho)$ with $\rho=(2k-1)^{-(1-d)l}$ has $P$ a.a.s., then for all $d>d_0$ a random group in $\cD (k,l,d)$ has $P$ a.a.s.
\end{proposition}

Propositions~\ref{prop:triangle-models-monotone} and \ref{prop:gromov-models-monotone} both follow immediately from \cite[Proposition 1.13]{JLR-00-random-graphs}.
Similar statements to translate a.a.s.\ properties from the density models back to the binomial models follow from \cite[Proposition 1.12]{JLR-00-random-graphs}, but we do not need these here.

Having $FL^p_L$ is a monotone property, so 
an immediate consequence of Proposition~\ref{prop:triangle-models-monotone} and Theorem~\ref{thm:tri-eventually-flp2} is the following.
\begin{varthm}[Corollary \ref{cor:Itri-density-flp}.]
	For any fixed density $d>1/3$ there exists $C >0$ so that for every $\varepsilon >0$ a.a.s.\ a random group in the triangular density model $\cM(m,d)$ has $FL^p_{(2-2\varepsilon)^{1/2p}}$  for every $p \in \left[ 2, C (\log m/\log\log m)^{1/2} \right]$. 
In particular, a.a.s.\ we have $FL^p$ for all $p$ in this range.
\end{varthm}

\subsection{Conformal dimension bounds}

As discussed in the introduction, the conformal dimension $\Confdim(\bdry \Gamma)$ of the boundary of a hyperbolic group $\Gamma$ is an analytically defined quasi-isometry invariant of $\Gamma$.  In this section we find the following bounds on conformal dimension in the triangular density model.  (Similar bounds hold in the triangular binomial model.) 
\begin{varthm}[Theorem \ref{thm:Iconfdim-tri-density-both-bounds}.]
	For any density $d\in (\frac{1}{3},\frac{1}{2})$, 
	there exists $C>0$ so that a.a.s.\ 
	$\Gamma \in \cM(m,d)$ is hyperbolic, and satisfies
	\[
	\frac{1}{C} \left(\frac{\log m}{\log\log m}\right)^{1/2} \leq \wp(\Gamma) \leq \Confdim(\bdry \Gamma) \leq C \log m.
	\]
	The same holds for $\Gamma(m,\rho)$ with $\rho=m^{3(d-1)+o(1)}$.
	In particular, as $m \ra \infty$, the quasi-isometry class of $\Gamma$ keeps changing.
\end{varthm}

The connection between conformal dimension and property $FL^p$ is given by the following result.
\begin{theorem}[Bourdon \cite{Bourdon16-properLp}]\label{thm:bourdon-flp-confdim}
	If $\Gamma$ is a Gromov hyperbolic group with $FL^p$, then
	the conformal dimension of its boundary satisfies
	$\Confdim(\bdry \Gamma) \geq p$; i.e.\ $\Confdim(\bdry \Gamma) \geq \wp(\Gamma)$.
\end{theorem}
\begin{proof}
	If $p > \Confdim(\bdry \Gamma)$, then $\Gamma$ has a proper isometric action on $\ell_p$, by
	\cite[Th\'eor\`eme 0.1]{Bourdon16-properLp}.
\end{proof}
This Theorem together with Theorem~\ref{thm:Itri-eventually-flp} and Corollary~\ref{cor:Itri-density-flp} immediately gives the lower bounds in Theorem~\ref{thm:Iconfdim-tri-density-both-bounds}.

It remains to find the upper bound for conformal dimension.
Ollivier's isoperimetric inequality for random groups in Gromov's density model
\cite[Theorem 2]{Oll-07-sc-rand-group} 
(see also~\cite[Section V]{Oll-05-rand-grp-survey})
can be proven for random groups in $\Gamma(m,\rho)$, as observed
by Antoniuk, {\L}uczak and {\'S}wi{\c{a}}tkowski.
\begin{lemma}[{\cite[Lemma 7]{ALS-15-random-triangular-at-third}}]\label{lem:rand-group-isop}
	If $\rho = m^{3(d-1)+o(1)}$ for some $d<\frac{1}{2}$, then
	for any $\eps >0$ a.a.s.\ for $\Gamma \in \Gamma(m,\rho)$ all reduced van Kampen
	diagrams $D$ for $\Gamma$ satisfy the isoperimetric inequality
	\[
		|\partial D| \geq 3(1-2d-\eps) |D|.
	\]	
\end{lemma}
By \cite[Proposition 15]{Oll-07-sc-rand-group}, which modifies Champetier's
bound in \cite[Lemma 3.11]{Champetier-94-petite-simp-hyp}, we have
\begin{lemma}\label{lem:hyp-constant-estimate}
	If $\rho = m^{3(d-1)+o(1)}$ for some $d<\frac{1}{2}$, then
	a.a.s.\ the Cayley graph of $\Gamma \in \Gamma(m,\rho)$ 
	is $\delta$-hyperbolic for $\delta = 5/(1-2d)$.
\end{lemma}
\begin{proof}
	Indeed, all relators have length three, so one can take 
	\[ \delta \geq 4 \frac{3}{3(1-2d-\eps)}. \]
	For sufficiently small $\eps>0$, it suffices to take $\delta \geq 5/(1-2d)$.
\end{proof}
This in turn yields our desired upper bound for the conformal dimension.
\begin{proposition}\label{prop:confdim-upper-bound}
	If $\rho = m^{3(d-1)+o(1)}$ for some $d<\frac{1}{2}$, then
	a.a.s.\ $\Gamma \in \Gamma(m,\rho)$ has
	\[ \Confdim(\bdry \Gamma) \leq \frac{30}{1-2d} \cdot \log(2m-1).\]
\end{proposition}
\begin{proof}
	This follows the proof of \cite[Proposition 1.7]{Mac-12-conf-dim-rand-groups}.
	The estimate $\delta = 5/(1-2d)$ of Lemma~\ref{lem:hyp-constant-estimate}
	allows us to find a visual metric on $\bdry \Gamma$ with visual exponent
	$\eps = 4\delta/\log(2) \geq 30/(1-2d)$.
	With this metric the boundary has Hausdorff dimension
	$\frac{1}{\eps} h(\Gamma)$, where $h(\Gamma)$ is the volume entropy of $\Gamma$.
	Since $\Gamma$ has $m$ generators, $h(\Gamma) \leq \log(2m-1)$, thus
	\[ \Confdim(\bdry \Gamma) \leq 30(1-2d)^{-1} \log(2m-1). \qedhere\]
\end{proof}
Each of these steps also applies to the model $\cM(m,d)$ for $d \in (\frac13,\frac12)$, so Theorem~\ref{thm:Iconfdim-tri-density-both-bounds} is proved.

\section{Multi-partite (random) graphs and bounding $\lambda_{1,p}$}\label{sec:multi-partite}

In the remainder of this paper, we wish to extend some of our results from the triangular models of random groups to the Gromov models.  This involves quite a few technicalities when done carefully; see for example Kotowski--Kotowski~\cite{KK-11-zuk-revisited}.  One approach they take to showing Property (T) for groups in the Gromov density model is to use an auxiliary bipartite model.  Unfortunately Proposition~\ref{prop:bipartite-est} implies that this strategy does not work for $FL^p$ with large $p$.  Instead we shall use a different auxilliary model based on complete multi-partite graphs.  

In this section, we bound $\lambda_{1,p}$ for random multi-partite graphs, and in Section~\ref{sec:Gromov} we apply it to random groups in the Gromov models.

\subsection{Complete multi-partite graphs}

Consider a complete $k$--partite graph with $k$ independent sets of vertices, each of $M$ vertices,
 and $m=kM$ the total number of vertices. We denote such a graph by $K_{k\times M}$. These are particular cases of Tur\'an graphs.  In this subsection we find bounds on $\lambda_{1,p}(K_{k \times M})$.

When $M=1$ we have the complete graph on $m$ vertices, and the following theorem gives the value of $\lambda_{1,p}$ in this case.
\begin{theorem}[Corollary 2, $\S 9$, in \cite{Amg-03-p-laplacian}]\label{thm:complete}
If $p > 2$ then the smallest positive eigenvalue of the $p$--Laplacian for the complete graph $K_m$ with $m$ vertices is 
$$
\lambda_{1,p}(K_m) = \frac{m-2+2^{p-1}}{m-1}. 
$$
\end{theorem} 

Using this, we can prove the following estimate

\begin{theorem}\label{thm:kpartite}
If $p>2$, $k, M \geq 2$ then the smallest positive eigenvalue of the $p$--Laplacian for the graph $K_{k\times M}$ satisfies 
$$
1 \geq \lambda_{1,p}(K_{k \times M}) \geq \frac{(m-2+2^{p-1})k}{m(k-1+2^{p+2})}, 
$$ where $m=kM$. 
\end{theorem}

\begin{proof} In what follows we fix the two arbitrary integers $k\geq 2$ and $M\geq 2$. Let $V$ be the set of vertices of $K_{k \times M}$ and let $V= V_1 \sqcup\cdots \sqcup V_k$ be the partition into $k$ sets containing $M$ vertices so that there is an edge between $u\in V_i$ and $v \in V_j$ if and only if $i \neq j$. Let $x$ be a non-constant function in $\R^V$ such that $\sum_{u \in V} \{x_u\}^{p-1} d_u = 0$ and $\| x \|_{p,\bd}^p = \sum_{u \in V}(1-1/k) m |x_u|^p = 1$.
We denote by $dx$ the total derivative of $x$ with respect to the set of edges in the graph $K_{k\times M}$, and by $d_c x$ the total derivative of $x$ with respect to the set of edges in the complete graph~$K_{m}$.

The upper bound is trivial: choose any such $x$ where $x$ is zero on $V_i$ for all $i \geq 2$, and then by \eqref{eq:lp-12}, $\lambda_{1,p}(K_{k \times M}) \leq \| dx \|_p^p = 1$.  In the remainder of the proof we show the lower bound for arbitrary such $x$.

Let $a\in V$ be the vertex such that $\sum_{v \sim a} |x_v - x_{a}|^p$ takes the minimal value among all the vertices in $V$. 
By summing over every edge twice, it follows that
$$
\sum_{v \sim a} |x_v - x_a|^p \leq \frac{2}{m} \| dx\|_p^p .
$$

Without loss of generality we may assume that $a\in V_1$, which means that the sum can be re-written as $\sum_{i=2}^k\sum_{v\in V_i} |x_v - x_a|^p$. 

H\"older's inequality implies that for any two positive numbers $\alpha, \, \beta ,$
$$
(\alpha + \beta)^p\leq 2^{p-1}(\alpha^p + \beta^p).
$$

Therefore for every $v,w\in V_i$ we can write, using the triangle inequality and the inequality above, that
$$
|x_v - x_w|^p\leq 2^{p-1} \left( |x_v - x_a|^p + |x_w - x_a|^p  \right).
$$

We may therefore write that 
\begin{equation}\label{eq:firstsum}
  \begin{split}\sum_{i=2}^k \sum_{v,w\in V_i}|x_v - x_w|^p & \leq 2^{p-1} \sum_{i=2}^k \sum_{v,w\in V_i} \left( |x_v - x_a|^p + |x_w - x_a|^p  \right) \\
  & \leq 2^{p-1} \sum_{i=2}^k\frac{2m}{k} \sum_{v\in V_i} |x_v - x_a|^p\leq 2^{p +1}\frac{1}{k}\| dx\|_p^p.\end{split}    
\end{equation}

We now consider the vertex $b\in V\setminus V_1$ which minimizes the sum $\sum_{v \sim b} |x_v - x_b|^p$ among all the vertices in $V\setminus V_1$. It follows that
$$
\sum_{v \sim b} |x_v - x_b|^p \leq \frac{2}{m(1-1/k)} \| dx\|_p^p. 
$$

Without loss of generality we may assume that $b\in V_k$, and an argument as above implies that
\begin{equation}\label{eq:secondsum}
\sum_{i=1}^{k-1} \sum_{v,w\in V_i}|x_v - x_w|^p\leq 2^{p-1} \sum_{i=1}^{k-1} \frac{2m}{k} \sum_{v\in V_i} |x_v - x_b|^p\leq 2^{p+1}\frac{1}{k-1}\| dx\|_p^p. 
\end{equation}

The inequalities \eqref{eq:firstsum} and \eqref{eq:secondsum} imply that
$$
\sum_{i=1}^{k} \sum_{v,w\in V_i}|x_v - x_w|^p\leq \frac{2^{p+2}}{k-1}\| dx\|_p^p.
$$

Therefore 
$$
\| d_cx\|_p^p =  \sum_{i=1}^{k} \sum_{v,w\in V_i}|x_v - x_w|^p + \| dx\|_p^p \leq \left( 1+\frac{2^{p+2}}{k-1}  \right) \| dx\|_p^p. 
$$

Let $y$ be the function $y=((1-1/k)/(1-1/m))^{1/p} x$, so that $\sum_{u\in V} (m-1)|y_u|^p=1$.
Since $y$ is an eligible function for $K_m$ in \eqref{eq:lp-12}, by Theorem~\ref{thm:complete} we have that 
$$
\| d_cy\|_p^p \geq \frac{m-2+2^{p-1}}{m-1}
$$ whence
$$
\| d_cx\|_p^p \geq \frac{(m-2+2^{p-1})(1-1/m)}{(m-1)(1-1/k)}.
$$ 

It follows that 
\[
\| dx\|_p^p \geq \frac{m-2+2^{p-1}}{m(1-1/k) \left( 1+\frac{2^{p+2}}{k-1}  \right)}.\qedhere
\]  
\end{proof}

For $p=2$, we can do better; this will be useful when showing property $(T)$ later.
\begin{proposition}\label{prop:k-partite-2-laplacian}
	For any $k, M \geq 2$, $\lambda_{1,2}(K_{k\times M})=1$.
\end{proposition}
\begin{proof}
	Denote the values of a function $x$ on the vertices of $K_{k\times M}$
	by $x_{i,u}$ for $1 \leq i \leq k, 1 \leq u \leq M$, with the first subscript indicating
	the partition into $k$ sets. 
	Then, if $x \in S_{2,\bd}((K_{k\times M})_0)$, by \eqref{eq:lp-12} we have
	\begin{align*}
		\|dx\|_2^2 & = \sum_{1 \leq i < j \leq k} \sum_u \sum_v |x_{i,u}-x_{j,v}|^2
		\\ & = \sum_i \sum_u x_{i,u}^2 \cdot (k-1)M
			- 2 \sum_{1 \leq i < j \leq k} \sum_u x_{i,u} \sum_v x_{j,v}
		\\ & = 1 - \sum_i \sum_u x_{i,u} \sum_{j \neq i} \sum_v x_{j,v}
		= 1 + \sum_{i=1}^k \left(\sum_{u=1}^M x_{i,u}\right)^2, 
	\end{align*}
	and equality is attained for any function $x$ with $\sum_u x_{i,u} = 0$ for all $1 \leq i
	\leq k$.
\end{proof}

Given this bound, one might wonder about the sharpness of Theorem~\ref{thm:kpartite}.
In particular, for some fixed $k \geq 2$ can we find $C>1/2$ so that for all $p \geq 2$, for all $M$ large enough $\lambda_{1,p}(K_{k \times M}) > C$?
This would remove the dependence of $k$ on $p$ in Theorem~\ref{thm:Igromov-flp}.
However, the
following proposition shows that, at least in the case of $k=2$, the theorem's estimate is fairly accurate.

\begin{proposition}\label{prop:bipartite-est}
	For any fixed $p > 2$, as $M \ra \infty$ we have
	\[
		\big(\tfrac{1}{2}-o(1)\big)\cdot \frac{1}{2^p} \leq \lambda_{1,p}(K_{2\times M}) \leq (1+o(1)) 
			\left( \frac{2}{\sqrt{5}} \right)^p.
	\]
\end{proposition}
\begin{proof}
	The lower bound of $(1/2-o(1))/2^p$ follows from Theorem~\ref{thm:kpartite} above.
	We use an explicit function to give an upper bound for $\lambda_{1,p}(K_{2\times m})$
	via \eqref{eq:lp-11}.
	We define a function $x$ on the $2M$ vertices of $K_{2\times M}$ which depends on two
	parameters $\delta, t \in (0,1)$.
	On $\delta M$ of the left (respectively right) vertices, let $x$ take the value $1$
	(resp.\ $-1$).
	On the remaining $(1-\delta)M$ of the left (resp.\ right) vertices, let $x$
	take the value $-t$ (resp.\ $t$).
	This function $x$ satisfies the conditions of \eqref{eq:lp-11}, so
	we can use it to give an upper bound for $\lambda_{1,p}(K_{2\times M})$.
	We do so with the (near optimal) choices of $t=1/5$, $\delta = t^{p/2} = 5^{-p/2}$.
	(The error caused by rounding $\delta M$ to the nearest integer disappears as $M \ra \infty$.)
	\begin{align*}
		\lambda_{1,p}(K_{2\times M}) & \leq \frac{\|dx\|_p^p}{\|x\|_{\bd}^p}
			\leq (1+o(1)) \frac{\delta^2 2^p + 2(1-\delta)\delta (1-t)^p + (1-\delta)^2 (2t)^p}
				{2(\delta 1^p + (1-\delta) t^p)} \\
			& \leq (1+o(1)) \frac{\delta^2 2^p + 2 \delta (1-t)^p + 2^pt^p}{2\delta}			
			\\ & = (1+o(1)) \left( \tfrac{1}{2} \cdot 5^{-p/2} 2^p + (\tfrac{4}{5})^p+ 
				\tfrac{1}{2} 2^p 5^{-p/2} \right)
			\\ & \leq (1+o(1)) \left(\tfrac{2}{\sqrt{5}}\right)^p \qedhere
	\end{align*}	
\end{proof}

\subsection{Multi-partite random graphs}
We can view random graphs $\bG(m,\rho)$ as arising from taking the complete graph $K_m$ and keeping each edge with probability $\rho$.  The following model is defined analogously using $K_{k \times M}$ as the base graph.
\begin{definition}\label{def:k-partite-rand-graph}
	For $k \geq 2$, $M \in \N$ and $\rho \in [0,1]$
	a \emph{random $k$-partite graph $G$ in the model $\bG_k(M,\rho)$} is 
	found by taking the graph $K_{k \times M}$ and keeping each edge with 
	probability $\rho$.  A property holds a.a.s.\ if it holds with probability $\ra 1$ as $M\ra \infty$.
\end{definition}

In this model, we show the following two bounds on $\lambda_{1,p}$ at slightly different ranges of $\rho$.
\begin{theorem}\label{thm:G-k-m-rho-eval-bound1}
	For any $\delta>0$ and $k\geq 2$, there exists $C>0$ so that
	for $\rho$ satisfying $\rho \geq C\log(kM)/kM$ and  
	$\rho =	o(M^{1/3}/M)$, we have that, for an arbitrary $p=p(M)\geq 2$, a random $k$-partite graph $G \in \bG_k(M,\rho)$ satisfies a.a.s.\ 
	\[
		\forall p' \in [2,p],\; \lambda_{1,p'}(G) \geq 
		(1-\delta)\inf_{p'' \in [2,p]}\lambda_{1,p''}(K_{k\times M}) - 2^p \delta -
		\frac{C 3^p}{(\rho kM)^{1/2p^2}}.
	\]
\end{theorem}

\begin{theorem}\label{thm:G-k-m-rho-eval-bound2}
	For any $\delta>0$, there exists $C>0$ so that
	for $k \geq 2$ and $k = o(M^{1/6\wedge \delta/2})$, for $\rho$ satisfying $\rho \geq (kM)^{\delta}/kM$ and  
	$\rho =	o(M^{1/3}/M)$ 
	we have that, for an arbitrary $p=p(M)\geq 2$, a random $k$-partite graph $G \in \bG_k(M,\rho)$ satisfies a.a.s.\ 
	\[
		\forall p' \in [2,p],\; \lambda_{1,p'}(G) \geq 
		\left(1-\frac{C}{(kM)^{\delta/3p}}\right)\inf_{p''\in[2,p]}\lambda_{1,p''}(K_{k\times M}) - 
		\frac{C 3^p}{(kM)^{\delta/2p^2}}.
	\]
\end{theorem}

Recall that in $K_{k \times M}$ the vertex set splits as $V = \bigsqcup_{i=1}^k V_i$ with an edge
joining $u \in V_i$ to $v \in V_j$ if and only if $i \neq j$.
For a graph $G \in \bG_k(M, \rho)$, and $u \in V, 1 \leq i \leq k$, let $d_{u,i}$ be the
number of edges with one endpoint at $u$ and the other endpoint in $V_i$.
So the degree of $u$ is $d_u = \sum_{i=1}^k d_{u,i}$, and $d_{u,i}=0$ when $u \in V_i$.
Let $D = D(G) = (d_{u,i})_{u,i}$ be the \emph{degree matrix} of $G$.

We call a matrix $D = (d_{u,i})$ with integer entries an \emph{admissible degree matrix} if for $i \neq j$, $\sum_{u \in V_i} d_{u,j} = \sum_{v \in V_j} d_{v,i}$; we denote by $\Delta_{i,j}$ the common value of the two sums.

Given an admissible degree matrix $D$, we define a random graph model $\bG_k(M,D)$ as follows. We
attach to each $u \in V_i$ a collection of $d_u$ half-edges, 
$d_{u,j}$ of which ``point towards'' $V_j$ for each $j$, 
and then for each $i\neq j$ we join the collections of $\Delta_{i,j}$
half-edges pointing to each other by a random matching.

In the particular case of $k=2$, this is just a random bipartite graph with specified degrees.

Given $G \in \bG_k(M,\rho)$, let $Y_{u,i}$ be the random variable which is $d_{u,i}$,
and let $Y_u = \val_G(u) = \sum_i Y_{u,i}$.
These satisfy $\bE Y_{u,i} = M \rho$ and $\bE Y_u = (k-1)M\rho =: \bar d$.
\begin{lemma}\label{lem:degree-concentration2}
	Given $\iota = \sqrt{10\log(Mk)/M\rho}$, a.a.s.\ 
	for all $u,i$, 
	\[
	  (1-\iota)M\rho \leq Y_{u,i} \leq (1+\iota)M\rho.
	\]
\end{lemma}
\begin{proof}
	As in the proof of Lemma~\ref{lem:degree-concentration}, we have
	$Mk(k-1)$ binomial random variables with expected value $M\rho$, so the probability that
	the claim fails is at most
	\[
		2 Mk(k-1) \exp\left( -\frac{1}{3}\iota^2 M\rho \right)
		\leq 2 Mk(k-1) (Mk)^{-10/3},
	\] and the latter upper bound converges to $0$ when $Mk \to \infty$.
\end{proof}

For an admissible degree matrix $D$, all $G \in \bG_k(M,\rho)$ with $D(G)=D$ arise with equal probability,
so to show Theorems~\ref{thm:G-k-m-rho-eval-bound1} and \ref{thm:G-k-m-rho-eval-bound2}
it suffices to find a.a.s.\ bounds on
$\lambda_{1,p}(G)$ for $G \in \bG_k(M,D)$, with all $d_{u,i} = (1+o(1)) M\rho$ (cf. Theorem~\ref{thm:G-m-deg-eval-bound}).
In particular, by Lemma~\ref{lem:degree-concentration2} we will assume that $D$ satisfies:
\begin{equation}\label{eq:GkM-deg-bound}
\begin{alignedat}{3}	\forall u \notin V_i, &\quad& \frac{(1-\iota)\bar{d}}{k-1} & \leq  d_{u,i} && \leq \frac{(1+\iota)\bar{d}}{k-1},\\
	\forall u, && {(1-\iota)\bar{d}} & \leq d_u && \leq {(1+\iota)\bar{d}},\\
	\forall i \neq j,&&  \frac{(1-\iota)M\bar{d}}{k-1} & \leq \Delta_{i,j} && \leq \frac{(1+\iota)M\bar{d}}{k-1}.
  \end{alignedat}
\end{equation}

Given $G \in \bG_k(M,D)$, we want to bound $Z_y(G)$ from below for all $y \in S_{p,\bd}(G_0)$.
By Proposition~\ref{prop:suffices-bound-t2}, either $Z_y(G) \geq 1$ (and we are done) or
\begin{equation*}%\label{eq:G-M-D-1}
  Z_y(G) \geq Z_x(G) - 2p (\eps \theta)^{1/(p-1)} (1+2(\eps \theta)^{1/(p-1)})^{p-1}
\end{equation*}
for some $x \in T_{p,\bd,R}(G_0)$ with $\|x\|^p_{p,\bd} \geq R_-$.
We will see later that we can assume $(1+2(\eps\theta)^{1/(p-1)})^{p-1} \leq 2$, so we record that
\begin{equation}\label{eq:G-M-D-1}
  Z_y(G) \geq Z_x(G) - 4p (\eps \theta)^{1/(p-1)}.
\end{equation}

Now 
\[
	Z_x(G) = \|x\|^p_{p,\bd} - X_x(G) = \|x\|^p_{p,\bd} - X_x^l(G) - X_x^h(G),
\]
so $\|x\|^p_{p,\bd} = \bE Z_x + \bE X_x^l + \bE X_x^h$, and thus
\begin{equation*}
	Z_x(G) = \bE Z_x + \bE X_x^h + (\bE X_x^l-X_x^l(G)) - X_x^h(G).
\end{equation*}
Applying this to \eqref{eq:G-M-D-1}, we find that all $y \in S_{p,\bd}(G_0)$ 
satisfy
\begin{equation}\label{eq:G-M-D-2}
	Z_y(G) \geq \bE Z_x - |\bE X_x^h| - |\bE X_x^l - X_x^l| - |X_x^h| -
		4p(\eps\theta)^{1/(p-1)}
\end{equation}
for some $x \in T_{p,\bd,R}(G_0)$ with $\|x\|^p_{p,\bd} \geq R_-$.

In the following subsections we bound each of the terms on the right hand side of \eqref{eq:G-M-D-2} from below.
 Here we use these bounds to finish the proofs of Theorems \ref{thm:G-k-m-rho-eval-bound1} and \ref{thm:G-k-m-rho-eval-bound2}. We use Assumption~\ref{assump:eps-theta} to simplify $\eps \theta \leq 1$ and $R \leq 4$.

Without loss of generality we may assume that $\iota  \leq 1/10$. Indeed, the hypothesis in Theorem \ref{thm:G-k-m-rho-eval-bound1} implies that $\iota = \sqrt{10k/C}$ and it suffices to take $C$ large enough, while according to the the hypothesis in Theorem \ref{thm:G-k-m-rho-eval-bound2},  $\iota = O( (kM)^{-\delta/6})$ (see estimates in \eqref{eq:estimate2} and the line following it).    

Using \eqref{eq:exp-Z-M-D-4}, \eqref{eq:exp-EXh-M-D}, Proposition~\ref{prop:deviation-light-M-D}, and \eqref{eq:Xh-M-D}, we obtain that for every $\xi >0$ there exists a suitable constant $C'$ such that with probability at least
\[
 1- 2 \exp \left(-\frac{1}{128}K^2m + m\log\left(\frac{16e}{\eps}\right)\right)
		-o(m^{-\xi})
\] 
all $y \in S_{p,\bd}(G_0)$ satisfy
\begin{align*}
	Z_y(G) & 
	\geq \left( \frac{(1-\iota)^4}{(1+\iota)^2} \lambda_{1,p}(K_{k\times M})(1-2\eps
	\theta^{1/p}) -2^{p+7}\iota \right)
	- \frac{p2^{p+8}}{d^{\beta/p}}
	- \frac{p2^p K}{d^{\beta/p}} \\
	& \qquad - \left( \frac{p2^p C'}{d^{\beta/p}} + \frac{p2^p C' p^3 }{d^{1/p^2}} + \frac{p2^pC'p^3\eps^{-q}}{d^{1/p}} 
		\right) -4p (\eps \theta)^{1/(p-1)} \\
	& \geq \left( \frac{(1-\iota)^4}{(1+\iota)^2} \lambda_{1,p}(K_{k\times M})(1-2\eps
	\theta^{1/p}) -2^{p+7}\iota \right)
 	\\ & \qquad	- C'' p2^p \left( \frac{1+K}{d^{1/(2+2p)}}
	+ \frac{p^3}{d^{1/p^2}} + \frac{p^3\eps^{-q}}{d^{1/p}} + \eps^{1/(p-1)} \right),
	\stepcounter{equation}\tag{\theequation}\label{eq:G-M-D-3}
\end{align*}
for some $C''$ depending on $C'$.

The same choices of $K = C_2(1+\sqrt{\log(d)/p})$ and $\eps=d^{-(p-1)/p(p+1)}$ as
in subsection~\ref{ssec:G-m-d-single-p} allow to deduce from the above that with probability at least $1-o(m^{-\xi})$, for all $y \in S_{p,\bd}(G_0)$, $Z_y(G)$ is at least
\begin{equation}\label{eq:G-M-D-4}
	\left( \frac{(1-\iota)^4}{(1+\iota)^2} \lambda_{1,p}(K_{k\times M})(1-2\eps
	\theta^{1/p}) -2^{p+7}\iota \right)
	- C_3 \frac{p^4 2^p}{d^{1/2p^2}}.
\end{equation}
As before, we can assume that $(1+2(\eps \theta)^{1/(p-1)})^{p-1} \leq 2$.

We now conclude the proofs of Theorems~\ref{thm:G-k-m-rho-eval-bound1} and
\ref{thm:G-k-m-rho-eval-bound2} using \eqref{eq:G-M-D-4} and \eqref{eq:lp-12}.

\begin{proof}[Proof of Theorem~\ref{thm:G-k-m-rho-eval-bound1}]
	Consider an arbitrary $\delta>0$.
	For $\rho = C \log(Mk)/Mk$, $\iota$ in Lemma~\ref{lem:degree-concentration2} becomes $\sqrt{10k/C}$.
	So for sufficiently large $C$ we can assume that $(1-\iota)^4/(1+\iota)^2 \geq 1-\delta/2$ and $2^7 \iota < \delta$.  
	We can also assume that $\eps$ is less than the fixed constant $\delta/4\theta^{1/p}$, and so $(1-\iota)^4(1+\iota)^{-2}(1-2\eps\theta^{1/p}) \geq (1-\delta/2)^2 \geq 1-\delta$. 

	Since $d \geq (1-\iota)\bar{d} = (1-\iota)(k-1)M \rho$,
	\begin{equation*}
		\frac{p^4 2^p}{d^{1/2p^2}} 
		\leq \frac{p^4 2^p}{((1-\iota)\rho M k)^{1/2p^2}} \cdot
			\left( \frac{k}{k-1} \right)^{1/2p^2}
		\leq C_4 \frac{3^p}{(\rho M k)^{1/2p^2}}.
	\end{equation*}
	Applying these estimates to \eqref{eq:G-M-D-4} and \eqref{eq:lp-12}
	shows that $\lambda_{1,p}(G)$ has the required bound for fixed $p$.
	
	The bound for all $p'$ in the given range
	follows from the argument of subsection~\ref{ssec:G-m-d-simul-p}.
\end{proof}

\begin{proof}[Proof of Theorem~\ref{thm:G-k-m-rho-eval-bound2}]
	We estimate the terms in \eqref{eq:G-M-D-4}.
	Since $k = o(M^{\delta/2})$,
	\begin{equation}\label{eq:estimate2}
	\frac{10 \log(Mk)}{M\rho}
		= \frac{10 k \log(Mk)}{(kM)^{\delta}} 
		= O\left( \frac{M^{\delta/2}\log(M)}{(kM)^{\delta}}	\right)
		= O\left( (kM)^{-\delta/3} \right),
	\end{equation}
	thus we can take $\iota = O( (kM)^{-\delta/6})$.
	Because $\iota = o(1)$, $(1-\iota)^4/(1+\iota)^2 = 1-6\iota-o(\iota) \geq 1-O((kM)^{-\delta/6})$.

	So $d =(1+o(1)) \bar{d} = (1+o(1)) (kM)^\delta $,
	and we have $\eps = d^{-(p-1)/p(p+1)} \leq d^{-1/3p}$,
	so $(1-2\eps\theta^{1/p}) = 1-O((kM)^{-\delta/3p})$.
	
	Likewise,
	\[
	\frac{p^42^p}{d^{1/2p^2}} = O\left(\frac{3^p}{(kM)^{\delta/2p^2}}\right).
	\]
	Thus \eqref{eq:G-M-D-4} is bounded from below by
	\begin{multline*}
	\left[1-O\left( (kM)^{-\delta/6} \right)\right] \lambda_{1,p}(K_{k\times M})
		\left[1-O\left( (kM)^{-\delta/3p}\right)\right]
		\\ - O\left(\frac{2^p}{(kM)^{\delta/6}}\right)
		-O\left(\frac{3^p}{(kM)^{\delta/2p^2}}\right),
	\end{multline*}
	which simplifies to give the claimed bound.
\end{proof}

\subsection{Expectation of Z}
For each $i \neq j$, let $V_{i \ra j}$ be the collection of $\Delta_{i,j}$ 
endpoints of half-edges from $V_i$ pointing towards $V_j$.
Given $a \in V_{i \ra j}$ and $b \in V_{j \ra i}$, let $\bI_{a,b}(G)$ be the random variable
which is $1$ or $0$ according to whether $a$ and $b$ are matched in $G$ or not.
For $a \in V_{i \ra j}$, denote by $v(a) \in V_i$ the other endpoint 
of the half-edge ending at $a$.
Then
\begin{align}
	\bE Z_x(G) & = \bE \sum_{i<j} \sum_{a \in V_{i \ra j}} 
		\sum_{b \in V_{j \ra i}} \bI_{a,b}(G) |x_{v(a)}-x_{v(b)}|^p \nonumber
	\\ & = \sum_{i<j} \sum_{u \in V_i} \sum_{v \in V_j} \frac{d_{u,i} d_{v,j}}{\Delta_{i,j}}
	| x_u - x_v |^p \nonumber
	\\ & \geq \frac{(1-\iota)^2 \bar{d}^2/(k-1)^2}{(1+\iota)M\bar{d}/(k-1)} 
	\sum_{i<j} \sum_{u \in V_i} \sum_{v \in V_j} | x_u - x_v |^p \quad\text{by \eqref{eq:GkM-deg-bound}}\nonumber
	\\ & = \frac{(1-\iota)^2\bar{d}}{(1+\iota)M(k-1)} \| dx \|_{K_{k \times M}, p}^p.
	\label{eq:exp-Z-M-D-1}
\end{align}

Let $x'_u = x_u (d_u/\bar{d})^{1/(p-1)}$ for each $u \in G_0 = (K_{k \times M})_0$.
Under this rescaling,
\[ 
  \sum_{u} \{x'_u\}^{p-1} (k-1)M = \sum_u \{x_u\}^{p-1} d_u \bar{d}^{-1}(k-1)M=0, 
\]
and 
\[
	\frac{(1-\iota)^q}{(1+\iota)\bar{d}} \sum_u |x_u|^p d_u 
	\leq \sum_u |x'_u|^p \leq \frac{(1+\iota)^q}{(1-\iota)\bar{d}} \sum_u |x_u|^p d_u.
\]
Thus we can use the definition of $\lambda_{1,p}(K_{k \times M})$ in \eqref{eq:lp-11}
to find
\begin{align}
	\| dx' \|_{K_{k \times M},p}^p
	\geq \lambda_{1,p}(K_{k \times M}) \|x'\|_{K_{k \times M}, p}^p
	\geq \lambda_{1,p}(K_{k \times M}) \frac{(1-\iota)^q R_- (k-1)M}{(1+\iota)\bar{d}}.
	\label{eq:exp-Z-M-D-2}
\end{align}
Using the Mean Value Theorem we have
\begin{equation}\label{eq:exp-Z-M-D-2-1}
| x_u'-x_u | = \left| \left(\frac{d_u}{\bar{d}}\right)^{1/(p-1)} -1 \right| |x_u|
	\leq \frac{2\iota}{p-1} |x_u|,
\end{equation}
for $(1+\iota)^{1/(p-1)}-1$ and $1-(1-\iota)^{1/(p-1)}$ are both at most $ 2 \iota/(p-1)$, for $\iota<1/10$.
Consequently
\begin{equation}\label{eq:exp-Z-M-D-2-2}
|x_u'|^{p-1} \leq (1+2\iota/(p-1))^{p-1} |x_u|^{p-1} \leq e^{2\iota} |x_u|^{p-1} \leq 2 |x_u|^{p-1}.
\end{equation}

We also require the following inequality, a straightforward consequence of the Mean Value Theorem (see also \cite[Lemma 4]{Matousek97}):
 	For any $p\geq 1$ and $a,b \in \R$ we have
	\begin{equation}\label{eq:exp-Z-M-D-2-3}
 		|\{a\}^p-\{b\}^p| \leq p |a-b| \left(|a|^{p-1}+|b|^{p-1}\right).
	\end{equation}
We combine these estimates to find, for the graph $K=K_{k\times M}$:
\begin{align*}
	& \quad \left| \|dx'\|_{K_{k\times M},p}^p-\|dx\|_{K_{k \times M},p}^p \right|
	\\ & \leq \sum_{e \in K_1, \vtx(e)=\{u,v\}} \big| |x'_u-x'_v|^p - |x_u-x_v|^p \big|
	\\ & \leq p \sum_{e \in K_1, \vtx(e)=\{u,v\}} \big| |x'_u-x'_v| - |x_u-x_v| \big| \big( |x_u'-x_v'|^{p-1}+|x_u-x_v|^{p-1} \big)
	\\ \intertext{by \eqref{eq:exp-Z-M-D-2-3}, and since $(a+b)^{p-1} \leq 2^{p-1}(a^p+b^p)$ this is}
	\\ & \leq 2^{p-1}p \sum_{\substack{e \in K_1,\\ \vtx(e)=\{u,v\}}} \big( |x_u'-x_u|+|x_v'-x_v|\big) \big(|x_u'|^{p-1}+|x_u|^{p-1}+|x_v'|^{p-1}+|x_v|^{p-1}\big)
	\\ & \leq \frac{2^{p-1}p \cdot 2\iota\cdot 3}{p-1} 
		\sum_{e \in K_1, \vtx(e)=\{u,v\}} (|x_u| + |x_v|) 
			\left( |x_u|^{p-1} + |x_v|^{p-1} \right)
	\\ \intertext{by \eqref{eq:exp-Z-M-D-2-1} and \eqref{eq:exp-Z-M-D-2-2}, and counting edges in $K$ gives that this is }
	\\ & \leq {2^{p+3} \iota}
		\left( \sum_u |x_u|^p(k-1)M + \sum_u \sum_v |x_u|\, |x_v|^{p-1} \right).
		\\ \intertext{By \eqref{eq:GkM-deg-bound}, this is bounded by}
	\\ & \leq 2^{p+3} \iota \left( 
		\frac{(k-1)M}{(1+\iota)\bar{d}} \sum_u |x_u|^p d_u
		+ \frac{1}{(1+\iota)^2 \bar{d}^2} \sum_u |x_u| d_u \sum_v |x_v|^{p-1} d_v \right)
	\\ & \leq 2^{p+3} \iota \left( 
		\frac{(k-1)M}{(1+\iota)\bar{d}} \sum_u |x_u|^p d_u
		+ \frac{1}{(1+\iota)^2 \bar{d}^2} \sum_u d_u \sum_v |x_v|^{p} d_v \right)
		\\ \intertext{by H\"older's inequality, and so we use that $x \in T_{p,\bd,R}(G_0)$ to get}
	\\ & \leq 2^{p+3} \iota \left( \frac{(k-1)MR}{(1+\iota)\bar{d}} +
		\frac{1}{(1+\iota)^2 \bar{d}^2} (1+\iota)\bar{d}kM R \right).
\end{align*}
All together, this gives
\begin{align*}
	\left| \|dx'\|_{K_{k\times M},p}^p-\|dx\|_{K_{k \times M},p}^p \right|
	\leq \frac{ 2^{p+4} \iota k M R }{ (1+\iota) \bar{d} }.
	\stepcounter{equation}\tag{\theequation}\label{eq:exp-Z-M-D-3}
\end{align*}

Finally, we combine \eqref{eq:exp-Z-M-D-1}, \eqref{eq:exp-Z-M-D-2} and \eqref{eq:exp-Z-M-D-3}
to conclude that
\begin{align}
	\bE Z_x(G)
	& \geq \frac{(1-\iota)^2 \bar{d}}{(1+\iota)M (k-1)}
		\left( \lambda_{1,p}(K_{k \times M}) \frac{(1-\iota)^q R_-(k-1)M}{(1+\iota)\bar{d}} - 
			\frac{2^{p+4} kMR}{(1+\iota)\bar{d}} \iota \right) \nonumber
	\\ & \geq \frac{(1-\iota)^4 }{(1+\iota)^2 }
	\lambda_{1,p}(K_{k \times M}) R_- - \frac{2^{p+4} k R \iota}{k-1} \nonumber
	\\ & \geq \frac{(1-\iota)^4 }{(1+\iota)^2 }
		\lambda_{1,p}(K_{k \times M}) (1-2\eps \theta^{1/p}) - 2^{p+7} \iota.
	\label{eq:exp-Z-M-D-4}
\end{align}
Observe that $R_- = (1-\eps \theta^{1/p})^q \geq 1- 2\eps \theta^{1/p}$.

\subsection{Expectation of heavy terms}
Let $\bI_{\ovP(x_u,x_v) \geq d^{\beta}/dm}$ equal $1$ or $0$ according to whether the given
inequality holds or not.
Then
\begin{align*}
	| \bE X_x^h |
	& = \left| \bE \sum_{i<j} \sum_{a \in V_{i \ra j}} \sum_{b \in V_{j \ra i}} \bI_{a,b}(G)
		\bI_{\ovP(x_{v(a)},x_{v(b)}) \geq d^{\beta}/dm} \Re(x_{v(a)},x_{v(b)}) \right|
	\\ & \leq \sum_{i<j} \sum_{u \in V_i} \sum_{v \in V_j} \frac{d_{u,i} d_{v,j}}{\Delta_{i,j}}
		\bI_{\ovP(x_u,x_v) \geq d^{\beta}/dm} (1+p2^{p-1}) \ovP(x_u,x_v)
		\quad\text{by \eqref{eq:p-ovp}}
	\\ & \leq \frac{(1+\iota)^2 \bar{d} p2^p}{(1-\iota) M(k-1)}
		\sum_{{u,v \in G_0 :\ }{\ovP(x_u,x_v) \geq d^{\beta} /dm}} \ovP(x_u,x_v)
		\quad\text{by \eqref{eq:GkM-deg-bound}}
	\\ & \leq \frac{(1+\iota)^2 \bar{d} p2^p}{(1-\iota) M(k-1)}
	  \frac{ 2m (1+\iota)^2 \big( \sum_u |x_u|^p d_u\big)^{2}}{(1-\iota)^2 d^{\beta/p}d} 
	  \quad\text{by Lemma~\ref{lem:p-gamma-bound} and \eqref{eq:GkM-deg-bound}}
	\\ & \leq \frac{(1+\iota)^4 p2^{p+2} R^2}{(1-\iota)^3 d^{\beta/p}},
\end{align*}
where we use that $\bar{d} \leq d$ and $m=Mk \leq 2M(k-1)$.
Therefore
\begin{equation}
	|\bE X_x^h| \leq \frac{p2^{p+8}}{d^{\beta/p}} = \frac{p2^{p+8}}{d^{1/(2+2p)}}.
\label{eq:exp-EXh-M-D}
\end{equation}
The final choice of $\beta = p/(2+2p)$ is the same as in Section~\ref{sec:G-m-d-overview}.

\subsection{Deviation of light terms}
\begin{proposition}\label{prop:deviation-light-M-D}
	For any $\alpha \in (0,1)$, so that $2\beta+2\alpha\leq 1$, and any $K>0$, for every 
	$x \in T_{p,\bd,R}(G_0)$ we have
	\[
		\bP\left( |X_x^l - \bE X_x^l| \geq \frac{p2^p K}{d^{\alpha}}\right)
		< 2 \exp \left(-\frac{1}{128}K^2m + m\log\left(\frac{16e}{\eps}\right)\right)\ .
	\]
\end{proposition}
As in subsection~\ref{ssec:light-dev-proof}, we will apply this with $\beta = p/(2+2p),
\alpha=\beta/p = 1/(2+2p)$.
\begin{proof}
As in the proof of Proposition~\ref{prop:deviation-light}, we order the vertices of each 
$V_{i \ra j}$ and define a filtration $(\cF_t)$ on $\bG_k(M,D)$ as follows:
first expose the edge connected to the first vertex of $V_{1 \ra 2}$, then the second, and so
on, then continue with $V_{1 \ra 3}, \ldots, V_{1 \ra k}, V_{2 \ra 3}$ etc. 
Let $\cF_t$ be the $\sigma$-algebra generated by the first $t$ exposed edges.

As before, let $S_t = \bE(X_x^l | \cF_t)$, so $S_0 = \bE(X_x^l)$ and $S_E = X_x^l$.  For edges
$e$ which contribute to $X_x^l$, we have 
$|\Re(e)| \leq p2^p \ovP(e) \leq p2^p d^\beta/dm$,
thus the same argument as before gives
$|S_t(G)-S_{t-1}(G)| \leq p2^{p+2} d^\beta/dm$.
Azuma's inequality tells us that $|X_x^l-\bE X^l_x|$ has the desired lower bound with probability less than
\[
2\exp \left( - \frac{(p2^p K/d^\alpha)^2}{2(dm)(p2^{p+2}d^{\beta}/dm)^2}\right) 
= 2 \exp \left( - \frac{K^2 dm}{32 d^{2\alpha+2\beta}} \right) \leq 2 \exp\left(-\frac{K^2m}{128}\right).
\]
The desired inequality follows from Proposition~\ref{prop:size-of-t} (compare \eqref{eq:light-config-bd}).
\end{proof}

\subsection{Controlled edge density}
In light of Proposition~\ref{prop:edge-density-heavy-bound}, to bound $X_x^h$ it suffices to show that $G \in
\bG_k(M,D)$ has controlled edge density.

The following lemma and its proof follow \cite[Lemma 16]{BFSU-99-rand-graphs}.
\begin{lemma}\label{lem:G-D-M-edge-density}
	Let $G$ be a random graph in $\bG_k(M,D)$, where $k=o(M^{1/6})$ and the matrix $D$ with $\min d_{u,i} = d_{min}/(k-1)$, and $\max d_{u,i} = d/(k-1)=d_{max}/(k-1)$, satisfies $d = o(M^{1/2})$.
	Then for $\theta \geq d/d_{min}$ sufficiently large, for any $\xi>0$
	there exists $C=C(\theta,\xi)>e$ so that with probability at least $1-o(m^{-\xi})$,
	$G$ has $(\theta,C)$-controlled edge density.
\end{lemma}
Proposition~\ref{prop:edge-density-heavy-bound} and \eqref{eq:p-ovp} then give that, with probability $1-o(m^{-\xi})$ for suitable $C'$, $G \in \bG_k(M,D)$ satisfies:
\begin{equation}\label{eq:Xh-M-D}
  |X_x^h(G)| \leq p2^p \overline{X}_x^h(G) \leq \frac{p2^pC'}{d^{\beta/p}}+\frac{p2^pC'p^3}{d^{1/p^2}} + \frac{p2^pC'p^3\eps^{-q}}{d^{1/p}}.
\end{equation}
\begin{proof}[Proof of Lemma~\ref{lem:G-D-M-edge-density}]
As in Definition~\ref{def:edge-density}, for $A,B \subset G_0$ we set 
$\mu(A,B)=\theta |A| |B| d/m$, where $m=Mk = |G_0|$, and $\theta \geq d/d_{min}$ will be
determined below.

We may assume that $|A| \leq |B|$, and that $|B| \leq m/4\theta$, for otherwise
$\mathcal{E}_{A,B} \leq |A|d \leq 4 \mu(A,B)$.

Suppose $a,b$ with $a \leq b \leq m/4\theta$ are given, and consider an arbitrary $t \in \N$. In what follows 
we provide an upper bound on the probability that there exist $A, B \subset G_0$ with $a =|A|, b=|B|$ and
$\mathcal{E}_{A,B}=t$:
\begin{equation}\label{eq:edgedens1}
\bP(\exists A, B : |A|=a, |B|=b, \mathcal{E}_{A,B}=t) \leq \left(\frac{me}{b}\right)^{2b}
	\left(\frac{\mu(A,B)}{t}\right)^t \left(ek\right)^{2t}.
\end{equation}

We frequently use the bounds
$ (n/k)^k \leq \binom{n}{k} \leq (en/k)^k$.

There are at most $\binom{m}{a} \binom{m}{b} \leq (me/b)^{2b}$ ways to choose $A$ and $B$.
There are then at most $\binom{ad}{t} \binom{bd}{t} \leq (abd^2e^2/t^2)^t$ ways to 
choose $t$ half-edges with endpoints in $A$ and $B$ to connect.

Suppose there are $t_{i,j}$ edges required to join $A$ and $B$ between $V_i$ and $V_j$,
for each $1 \leq i < j \leq k$, with $\sum t_{i,j}=t$.
Then the probability that a random matching of the $\Delta_{i,j}$ half-edges connects the two
specified sets of $t_{i,j}$ edges is $\binom{\Delta_{i,j}}{t_{i,j}}^{-1}$. 
Now $Md/(k-1)\theta \leq \Delta_{i,j}$, so 
\[
	\prod_{i,j} \binom{\Delta_{i,j}}{t_{i,j}}^{-1}
	\leq \prod_{i,j} \left(\frac{t_{i,j}}{\Delta_{i,j}}\right)^{t_{i,j}}
	\leq \prod_{i,j} \left(\frac{t\theta(k-1)}{Md}\right)^{t_{i,j}}
	= \left(\frac{t \theta(k-1)}{Md}\right)^t.
\]
These bounds give \eqref{eq:edgedens1} since 
\[
	\left(\frac{abd^2e^2}{t^2}\right)^t \left(\frac{t \theta(k-1)}{Md}\right)^t
	= \left( \frac{\mu(A,B) e^2 k(k-1)}{t} \right)^t. 
\]
Having shown \eqref{eq:edgedens1}, we continue with the proof of Lemma~\ref{lem:G-D-M-edge-density}.
\medskip

If Definition~\ref{def:edge-density}(a,b) fails for blocks $A, B$ with $\mathcal{E}_{A,B}=t$, then
$\mu(A,B)/t \leq 1/C$ and $(\mu(A,B)/t)^t \leq (b/m)^{Cb}$.
Thus, in the case $t \geq \log^2 m$, the right hand side in \eqref{eq:edgedens1} is bounded by
\begin{align*}
	&\left(\frac{me}{b}\right)^{2b}
	\left(\frac{\mu(A,B)}{t}\right)^{t/2} 
	\left(\frac{\mu(A,B)}{t}\right)^{t/2} \left(ek\right)^{2t}
	\\ & \leq \left(\frac{me}{b}\right)^{2b} 
	\left(\frac{b}{m}\right)^{Cb/2} \left(\frac{1}{C}\right)^{t/2} \left(ek\right)^{2t}
	\\ & = \left(\frac{b}{m}\right)^{Cb/4-2b} \left(e^2 \left(\frac{b}{m}\right)^{C/4}\right)^{b}
		\left(\frac{e^2 k^2}{\sqrt{C}}\right)^t \leq e^{-t} \leq m^{-\log m},
\end{align*}
provided $\theta \geq e^2/4$ (so $e^2b/m \leq 1$), and $C \geq e^6k^4\geq 8$.
Summing over the $m^2$ possibilities for $a,b$ and the $dm \leq m^2$ possibilities for $t$
proves the lemma in the case of $\mathcal{E}_{A,B} \geq \log^2 m$.

Now suppose $\mathcal{E}_{A,B}=t \leq \log^2 m$, for some $A,B$ failing
Definition~\ref{def:edge-density}.
Then
\[
	\left(\frac{1}{m^2}\right)^t
	\leq \left(\frac{\theta abd/m}{\log^2m}\right)^t
	\leq \left(\frac{\mu(A,B)}{t}\right)^t
	\leq \left(\frac{b}{m}\right)^{Cb} 
	\leq \left(\frac{1}{4\theta}\right)^{Cb}
	\leq e^{-2b},
\]
so $b \leq t \log m \leq \log^3 m$.
Since $(1/m)^{2t} \leq ((\log^3 m)/m)^{Cb}$, we have $t \geq Cb/3$.
So, using \eqref{eq:edgedens1}, the probability $P$ that there exists some $A, B$ with
$\mathcal{E}_{A,B} \leq \log^2 m$ failing Definition~\ref{def:edge-density} is 
\begin{align*}
	P & \leq \sum_a \sum_b \sum_{t=Cb/3}^{\log^2 m} \left(\frac{me}{b}\right)^{2b}
		\left(\frac{\theta abde^2k}{tM}\right)^t.
\end{align*}
Since $d=o(M^{1/2})$, $k=o(M^{1/6})$, and $ab \leq \log^6 m$, the sum in $t$ is bounded by a geometric series of ratio $\leq 1/2$, so
\begin{align*}
	P & \leq \sum_a \sum_b 2 \left(\frac{me}{b}\right)^{2b}
		\left(\frac{\theta abde^2k^2}{mCb/3}\right)^{Cb/3}
	\\ & \leq \sum_a \sum_b 2 \left( m^{2-C/3} b^{-2+C/3} (3\theta de^2k^2/C)^{C/3} e^2 \right)^b.
\end{align*}
Now $dk^2 = (dk^{1/2}) \cdot k^{18/14}k^{3/14} \leq M^{1/2}k^{1/2} \cdot M^{3/14}k^{3/14} = m^{5/7}$,
and so for $C > 21$ there exists $C_1$ so that for $m$ large enough
\[
  m^{2-C/3} b^{-2+C/3} (3\theta de^2k^2/C)^{C/3} e^2 
  \leq C_1 m^{2-C/3} (\log^3 m)^{-2+C/3} m^{5/7 \cdot C/3} \leq 1.
\]
Thus
\begin{align*}
  P \leq (\log^6 m) 2 C_1 (\log^3 m)^{-2+C/3} m^{2-2C/21},
\end{align*}
which suffices to complete the proof.
\end{proof}

\section{Fixed points for random groups in Gromov's density model}\label{sec:Gromov}

  In this section we use the bounds on $\lambda_{1,p}$ for random multi-partite graphs to show the following fixed point properties for the Gromov binomial and density models (see Definitions~\ref{def:gromov-model} and \ref{def:gromov-model-binom}).
\begin{theorem}\label{thm:gromov-flp}
	Choose $p \geq 2$ and $k \geq 10 \cdot 2^p$.
	Fix a density $d>1/3$.
	Then a.a.s.\ a random $k$-generated group at density $d$ has $FL^{p'}$ for all $2 \leq p' \leq p$, both in the standard Gromov density model $\cD (k,l,d)$ and in the binomial model $\cB (k,l,(2k-1)^{-(1-d)l})$.
\end{theorem}
The arguments in this section owe a debt to those of
\cite{ALS-15-random-triangular-at-third} and \cite{KK-11-zuk-revisited},
though our approach to Gromov's density model is new even for property (T), and gives new results at density $d=1/3$, see Theorem~\ref{thm:gromov-T} and Corollary~\ref{cor:gromov-T2}.
It is reasonable to expect that the dependence of $k$ on $p$ is unnecessary in Theorem~\ref{thm:gromov-flp}, however our methods are not at present able to avoid this obstacle.

Suppose we are given $l \in \N$ that is a multiple of $3$.
Let $W_{l/3}$ be the collection of all reduced words in $\langle \cA\rangle$ of length $\frac{l}{3}$, so $|W_{l/3}| = 2k(2k-1)^{l/3-1}$.
The map $w \mapsto w^{-1}$ on $W_{l/3}$ is a fixed point free involution.
Choose a set $S$ of size $\frac12 |W_{l/3}|$ and a injection $\phi:S \ra W_{l/3}$ so that
$\phi(S)$ is a collection of orbit representatives of this involution.

Given $\Gamma =\langle \cA | R \rangle \in \cB (k,l,\rho)$, we can lift $\Gamma$ to a group
$\tilde{\Gamma} = \langle S | \tilde{R} \rangle$ as follows:
$\phi:S \ra W_{l/3}$ extends naturally to a bijection $\phi: S \cup S^{-1} \ra W_{l/3}$.
For each $r \in R$, write $r$ as a concatenation $\phi(s)\phi(t)\phi(u)$ for some
$s,t,u \in S \cup S^{-1}$, and then define $\tilde{r} = stu$.
Let $\tilde{R}$ be the collection of all such $\tilde{r}$.
The map $\phi$ extends to a homomorphism $\phi:\tilde{\Gamma}\ra \Gamma$.
As in \cite[Lemma 3.15]{KK-11-zuk-revisited}, $\phi(\tilde{\Gamma})$ 
has finite index in $\Gamma$:
\begin{lemma}\label{lem:lift-finite-index-image}
	The image $\phi(\tilde{\Gamma})$ has finite index in $\Gamma$.
\end{lemma}
\begin{proof}
	For any reduced word $ab$ of length $2$ in the generators $\cA \cup \cA^{-1}$,
	we can find a word $w$ of length $l/3-1$ so that $aw$ and $b^{-1}w$ are both reduced.
	Thus there are generators $s,t \in S \cup S^{-1}$ so that 
	$ab=(aw)(w^{-1}b)=\phi(s)\phi(t)=\phi(st)$.

	Therefore for $g \in \Gamma$, if $g$ has even length it lies in 
	$\phi(\tilde{\Gamma})$, and if $g$ has odd length it lies in one of the finitely many
	cosets $a\phi(\tilde\Gamma)$, $a \in \cA \cup \cA^{-1}$.
\end{proof}

So to show that $\Gamma$ has $FL^p$ is suffices to show that $\tilde\Gamma$ has the same
property, and for this we show that the link graph $L(S)$ of $\tilde\Gamma$ has
$\lambda_{1,p}(L(S))>1/2$.
As in the proof of Theorem~\ref{thm:tri-eventually-flp2}, 
we split $L(S)$ as a union of three graphs $L(S) = L^1(S) \cup L^2(S) \cup L^3(S)$ where for
each relation $stu \in \tilde R$ we put the edge $(s^{-1},t)$ in $L^1(S)$, the edge
$(t^{-1},u)$ in $L^2(S)$, and the edge $(u^{-1},s)$ in $L^3(S)$.

For each $a \in \cA \cup \cA^{-1}$, let $S_a$ be the subset of $S$ consisting of generators $s$
so that $\phi(s) \in W_{l/3}$ has initial letter $a$; $S_a$ has size $M=(2k-1)^{l/3-1}$.
Observe that $st$ can begin a relation $stu \in \tilde{R}$ if and only if
$\phi(s)\phi(t)$ is a reduced word in $\langle \cA \rangle$,
which holds exactly when $\phi(s)^{-1}=\phi(s^{-1})$ and $\phi(t)$ have different initial
letters.
In other words, $s^{-1}$ and $t$ lie in different sets of the partition
$S = \bigsqcup_{a \in \cA \cup \cA^{-1}} S_a$.
We now show that each $L^i(S)$ is the union of a random $2k$-partite graph with two matchings.

	We require a count on the number of ways to complete a cyclically reduced word.
	\begin{lemma}[{See Lemma 2.4, \cite{Mac-12-conf-dim-rand-groups}}]
		\label{lem:cyc-word-count}
		Let $\langle A \rangle$ be a free group with $|A|=k$.
		Fix $a,b \in A \cup A^{-1}$.
		The number of reduced words $w$ of length $n+2$ with initial letter $a$ and final
		letter $b$ equals $q_n$ or $q_n+1$, where
		\begin{equation*}
			q_n = 
			\begin{cases}
				\frac{1}{2k} \big( (2k-1)^{n+1} - 1 \big) & \text{if $n$ is odd},\\
				\frac{1}{2k} \big( (2k-1)^{n+1} - (2k-1) \big) & \text{if $n$ is even}.
			\end{cases}
		\end{equation*}
		In either case, for $n \geq 2$, $\frac12 (2k-1)^n \leq q_n \leq (2k-1)^n$.
	\end{lemma}
	(Whether the number is $q_n$ or $q_n+1$ depends on whether $a=b$ or $b^{-1}$, and the 
	parity of $n$; we do not need this here.)

\begin{proposition}\label{prop:link-graph-structure-dens}
	Suppose $\rho=(2k-1)^{-(1-d)l}$ for $d<5/12$, 
	and let $\rho' = 1-(1-\rho)^{2q_{l/3}}$ and $M=(2k-1)^{l/3-1}$.
	Let $L^1(S)$ be the first link graph of the lift $\tilde{\Gamma}$ of $\Gamma \in \cB (k,l,\rho)$.
	Then $L^1(S)$ is the union of a graph in $\bG_{2k}(M,\rho')$ and two matchings.
\end{proposition}
\begin{proof} We define an auxiliary multigraph $\cK$ with the same vertex set as $K_{2k \times M}$.
	For $s^{-1} \in S_a \subset S$ and $t \in S_b \subset S$ with $a \neq b$, there are 
	$q_{l/3}$ or $q_{l/3}+1$ possible ways to complete $\phi(s)\phi(t)$ to a cyclically reduced
	word $\phi(s)\phi(t)\phi(u)$ of length $l$, and the same number of ways to complete
	$\phi(t^{-1})\phi(s^{-1})$, depending on the final letters of $\phi(s^{-1})$ and $\phi(t)$.
	Accordingly, add $2q_{l/3}$ or $2q_{l/3}+2$ edges to $\cK$ between $s^{-1}$ and $t$.

	Then $L^1(S)$ can be viewed as the random graph obtained from $\cK$ by retaining each edge
	with probability $\rho'$.
	
	For each pair of vertices in $\cK$ connected by $2q_{l/3}+2$ edges, delete two edges
	and call the resulting graph $\hat\cK$; let $\cD$ be the collection of deleted edges.
	Let $L^1(S) = \hat L \cup \hat D$ where $\hat L$ is the portion of $L^1(S)$ coming from
	$\hat\cK$ and $\hat D$ the portion coming from $\cD$.
	
	First we show $\hat D$ is a matching.
	For convenience, we write $d=1/3+\eps$ and so $\rho = (2k-1)^{(\eps-2/3)l}$.
	The probability that a vertex in $\hat D$ has two edges connected to it is at most
	\[
		(2kM)(2kM)^2 2^2 \rho^2 = O((2k-1)^{3l/3+2(\eps-2/3)l}) \ra 0,
	\]
	since $\eps < 1/6$.

	Second we show $\hat L$ has no triple edges: the probability is at most 
	\[
		(2kM)^2(2q_{l/3})^3 \rho^3 = O((2k-1)^{5l/3+3(\eps-2/3)l}) \ra 0,
	\]
	since $\eps < 1/9$.

	Finally we show that no pair of double edges in $\hat L$ share a vertex:
	the probability is at most
	\[
	(2kM)(2kM)^2(2q_{l/3})^4\rho^4 = O((2k-1)^{7l/3+4(\eps-2/3)l}) \ra 0,
	\]
	since $\eps < 1/12$.
	
	So $L^1(S)$ is the union of a graph $L' \in \bG_{2k}(M, \rho')$,
	a matching coming from the double edges of $\hat L$ and the matching $\hat D$.
\end{proof}

Since both $\rho = (2k-1)^{-(1-d)l}$ and $\rho q_{l/3} =(1+o(1)) (2k-1)^{-(2/3-d)l}$ are small for $l$ large enough, by the Mean Value Theorem we have
\[
	\tfrac12 (2k-1)^{-(2/3-d)l} \leq \rho q_{l/3} \leq 
	\rho' = 1-(1-\rho)^{2q_{l/3}} \leq 2\rho q_{l/3} \leq 4(2k-1)^{-(2/3-d)l}.
\]

We are now able to show $FL^p$ for random groups in the Gromov binomial model, and hence the Gromov density model as well. 
\begin{proof}[Proof of Theorem~\ref{thm:gromov-flp}]
	By Proposition~\ref{prop:gromov-models-monotone} it suffices to consider
	the model $\cB (k,l,\rho)$ for $\rho(l) = (2k-1)^{-(1-d)l}$.
	By Proposition~\ref{prop:link-graph-structure-dens}, a.a.s.\ $\Gamma \in \cB (k,l,\rho)$
	is, up to finite index, the quotient of a group $\tilde{\Gamma }$ whose link graph is the union of three graphs each consisting of a graph
	$G \in \bG_{2k}(M,\rho')$ and two matchings.
	
	Let us write $d=1/3+\eps$, and recall that $M=(2k-1)^{l/3-1}$.
	Then $\rho' \asymp (2k-1)^{(\eps-1/3)l} \asymp (2kM)^{3\eps}/(2kM) = O(M^{3\eps}/M)$, where $A \asymp B$ means that $A = O(B)$ and $B=O(A)$.
	We can assume that $3\eps < 1/3$, since $FL^p$ is preserved by increasing density.
	By Theorem~\ref{thm:G-k-m-rho-eval-bound2},
	a.a.s.\ we have that $G \in \bG_{2k}(M,\rho')$ satisfies that for all $2 \leq p' \leq p$
	\begin{equation}\label{eq:gromov-flp-1}
	  \lambda_{1,p'}(G) \geq \left(1-O\left((2kM)^{-3\eps/3p}\right)\right) \inf_{p''\in[2,p]}\lambda_{1,p''}(K_{2k \times M})
			- O\left(\frac{3^p}{(2kM)^{3\eps/2p^2}}\right).
	\end{equation}	
	Because $k \geq 10 \cdot 2^p$, Theorem~\ref{thm:kpartite} gives that
	$\lambda_{1,p''}(K_{2k\times M}) \geq (1-o(1))\frac23$ for all $p''\in[2,p]$.
	Thus \eqref{eq:gromov-flp-1} shows that $\lambda_{1,p'}(G) > \frac23 - o(1)$.

	Lemmas \ref{lem:lp-add-few-edges} and \ref{lem:lp-bound-split-into-three} imply that after adding two matchings to $G$ and combining three such graphs, the link
	graph of $\tilde{\Gamma}$ still has $\lambda_{1,p'} > 1/2$ for all $p' \in [2,p]$.
	So by Theorem~\ref{thm:bourdon-flp}, a.a.s.\ $G$ has $FL^{p'}$ for all $p' \in [2,p]$.
\end{proof}

Finally, we use our results to give a new proof of Kazhdan's property (T) for random groups in Gromov's density model at $d>1/3$, and moreover give new information at $d=1/3$.
\begin{theorem}\label{thm:gromov-T}
	For any $k \geq 2$, there exists $C>0$ so that for $\rho(l) \geq Cl(2k-1)^{l/3}/(2k-1)^l$ 
	a.a.s.\ a random $k$-generated group in $\cB (k,l,\rho)$
	has Kazhdan's property $(T)$.
\end{theorem}
\begin{proof}
	We follow the proof of Theorem~\ref{thm:gromov-flp}, only pointing out the differences.

	Since $(2k-1)M = (2k-1)^{l/3}$, 
	$\rho(l) \geq (2k-1)^{(1/3+\eps)l}/(2k-1)^l $
	for $\eps = \log_{2k-1}(Cl) /l \ra 0$.
	As above, $\rho' \asymp (2k-1)^{(\eps-1/3)l} \asymp Cl(2k-1)^{-l/3}=O(M^{3\eps}/M)$.
	
	On the one hand,
	$\rho'\geq \frac12 (2k-1)^{(\eps-1/3)l} = \frac12 Cl (2k-1)^{-l/3}$,
	while on the other hand,
	$\log(2kM)/(2kM) \leq C' l/(2k-1)^{l/3}$ for fixed $C'$.
	Therefore, for any choice of $\delta>0$, we can choose $C$ large enough so that
	$\rho' \geq \tilde{C} \log(2kM)/2kM$, where $\tilde{C}>0$ is the constant corresponding to
	$\delta$ given by Theorem~\ref{thm:G-k-m-rho-eval-bound1}.
	
	Recall that Proposition~\ref{prop:k-partite-2-laplacian} gives 
	$\lambda_{1,2}(K_{2k \times	M}) = 1$ for all $k \geq 2, M \geq 2$.
	Taking $\delta=1/100$, Theorem~\ref{thm:G-k-m-rho-eval-bound1} gives that
	\[
		\lambda_{1,2}(G) \geq \frac{99}{100} - \frac{4}{100} - 
			O\left(\frac{1}{(\rho'2kM)^{1/8}}\right)
		\geq \frac{9}{10} - O\left(\frac{1}{l^{1/8}}\right).
	\]
	Again, adding two matchings and combining three such graphs does not significantly lower $\lambda_{1,2}$, so
	a.a.s.\ $\Gamma \in \bG(k,l,\rho)$ is the quotient of a group whose link graph has
	$\lambda_{1,2} > 1/2$.  Thus a.a.s.\ $\Gamma$ has Kazhdan's property (T).
\end{proof}
\begin{corollary}\label{cor:gromov-T2}
	For any $k\geq 2$, there exists $C>0$ so that for every sequence of integers $f:\N \to \N$ satisfying $f(l)\geq Cl(2k-1)^{l/3}$,
	a.a.s.\ a random $k$-generated group in the Gromov model $\cG (k,l,f )$ has Kazhdan's property (T).
	In particular, this holds in the  Gromov density model $\cD (k,l,d )$ for all densities $d>1/3$.
\end{corollary}
\begin{proof}
	Follows immediately from Theorem~\ref{thm:gromov-T} and Proposition~\ref{prop:gromov-models-monotone}, upon increasing $C$ by an arbitrarily small amount.
\end{proof}

%%%%%%%%%%%%%%%%%%%%%%%%%%%%%%%%
%\bibliographystyle{amsalpha} %
%\bibliography{cornelia} %
%%%%%%%%%%%%%%%%%%%%%%%%%%%%%%%%

%DIFDELCMD < %%%
\def\cprime{$'$} \def\cprime{$'$} \def\cprime{$'$} \def\cprime{$'$}
  \def\cprime{$'$}
\providecommand{\bysame}{\leavevmode\hbox to3em{\hrulefill}\thinspace}
\providecommand{\MR}{\relax\ifhmode\unskip\space\fi MR }
\providecommand{\MRhref}[2]{%
  \href{http://www.ams.org/mathscinet-getitem?mr=#1}{#2}
}
\providecommand{\href}[2]{#2}

\end{document}